\documentclass[10pt]{amsart}
\usepackage{amsmath, latexsym, amssymb}
\usepackage{endnotes}
\usepackage{ulem} 
\usepackage[usenames]{color} 
\usepackage[all, knot]{xy}
\input grcawol.sty
\input grcalcx.sty

\newcommand\hbr{\sigma}

\def\a{\alpha}
\def\b{\beta}
\def\e{\epsilon}
\def\w{\omega}
\def\G{\Gamma}
\def\hG{\hat\Gamma}
\def\k{\Bbbk}

\def\fs{\mathfrak{s}}
\def\ft{\mathfrak{t}}
\def\fg{\mathfrak{g}}
\def\fj{\mathfrak{j}}

\def\BQ{{\mathbb{Q}}}
\def\BC{{\mathbb{C}}}

\def\BZ{{\mathbb{Z}}}

\def\BN{{\mathbb{N}}}
\def\bX{{\mathbf{X}}}
\def\bU{{\mathbf{U}}}
\def\uS{\underline{S}}
\def\bY{{\mathbf{Y}}}

\def\SZ{I^{\operatorname{SZ}}}
\def\qed{{\ \ \ \mbox{$\square$}}}
\newcommand\mtx[1]{\begin{bmatrix} #1 \end{bmatrix}}

\newcommand\im{\operatorname{Im}}
\renewcommand{\c}{\operatorname{c}}
\renewcommand{\u}{\operatorname{u}}

\def\Hom{{\mbox{\rm Hom}}}
\def\End{{\mbox{\rm End}}}

\def\PGL{{\mbox{\rm PGL}}}
\def\GL{{\mbox{\rm GL}}}
\def\Gal{{\mbox{\rm Gal}}}
\def\FSexp{{\mbox{\rm FSexp}}}
\def\Cent{{\mbox{\rm Center}}}

\newcommand\C[1]{{#1\mbox{-\bf{mod}}_{\operatorname{\mathsf fin}}}}

\newcommand\lcm{\operatorname{lcm}}

\newtheorem{thm}{Theorem}[section]
\newtheorem{cor}[thm]{Corollary}
\newtheorem{prop}[thm]{Proposition}
\newtheorem{lem}[thm]{Lemma}

\newtheorem{defn}[thm]{Definition}
\newtheorem{example}[thm]{Example}
\newtheorem{remark}[thm]{Remark}

\numberwithin{equation}{section}

\newcommand\ord{\operatorname{ord}}

\newcommand\Z{\BZ}
\newcommand\ol[1]{\overline{#1}}
\newcommand\replace[1]{}
\newcommand\ptr{\operatorname{\underline{ptr}}}

\newcommand\ptrl{\ptr^\ell}
\newcommand\ptrr{\ptr^r}
\newcommand\FS{\operatorname{\underline{FS}}}

\newcommand\id{\operatorname{id}}
\renewcommand\o{\otimes}

\newcommand\Tr{\operatorname{Tr}}
\newcommand\SL{\operatorname{SL}_2(\BZ)}
\newcommand\SLN{\operatorname{SL}_2(\BZ/N\BZ)}
\newcommand\qsl[1]{{\operatorname{SL}_2(\BZ / #1 \BZ)}}

\newcommand\inv{^{-1}}
\DeclareMathOperator\ev{{\operatorname{ev}}}
\DeclareMathOperator\db{\operatorname{db}}
\newcommand\CC{\mathcal C}
\newcommand\CCstr{{\mathcal{C}_{\rm str}}}
\newcommand\DD{\mathcal D}
\newcommand\FF{\mathcal F}
\newcommand\KK{\mathcal K}

\newcommand\A{\mathcal A}

\newcommand\du{^{\vee}}
\newcommand\bidu{^{\vee\vee}}

\newcommand{\ou}[1]{\mathrel{\mathop{\otimes}_{#1}}}

\newcommand\op{{\operatorname{op}}}
\newcommand\cop{{\operatorname{cop}}}

\makeatletter
\def\namelabel#1#2{\@bsphack
  \protected@write\@auxout{}%
         {\string\newlabel{#1.nme}{{#2}{#2}}}%
  \@esphack}
\def\nmlabel#1#2{\label{#2}\namelabel{#2}{#1}}

\makeatother

\title[Congruence Subgroups and Generalized Frobenius-Schur Indicators]{Congruence Subgroups and Generalized Frobenius-Schur Indicators}
\author{Siu-Hung Ng}
\address{Department of Mathematics, Iowa State University, Ames, IA 50011, USA}
\email{rng@iastate.edu}
\thanks{The first author is supported by the NSA grant H98230-08-1-0078.}
\author{Peter Schauenburg}
\address{Mathematisches Institut der Universit\"at M\"unchen,
Theresienstr.\ 39, 80333 M\"unchen, Germany}
\email{schauenburg@math.lmu.de}
\thanks{The second author is supported by \emph{Deutsche Forschungsgemeinschaft} through a Heisenberg fellowship}
\begin{document}
\begin{abstract}
  We introduce generalized Frobenius-Schur indicators for pivotal categories. In a spherical fusion category $\CC$, an equivariant indicator of an object in $\CC$ is defined as a functional on the Grothendieck algebra of the quantum double $Z(\CC)$ via generalized Frobenius-Schur indicators. The set of all equivariant indicators admits a natural action of the modular group. Using the properties of equivariant indicators, we prove a congruence subgroup theorem for modular categories. As a consequence, all modular representations of a modular category have finite images, and they satisfy a conjecture of Eholzer. In addition, we obtain two formulae for the generalized indicators, one of them a generalization of Bantay's second indicator formula for a rational conformal field theory. This formula implies a conjecture of Pradisi-Sagnotti-Stanev, as well as a conjecture of Borisov-Halpern-Schweigert.
\end{abstract}
\maketitle

\section*{Introduction}\label{s:intro}
The importance of the role of the modular group $\SL$ in conformal
field theory has been known since the work of Cardy \cite{Cardy}.
Associated to a 2D rational conformal field theory (RCFT) is a
finite-dimensional representation of $\SL$ with a distinguished
basis formed by the characters of the primary fields. This modular
representation conceives
some interesting algebraic and arithmetic properties.
One notable example is the Verlinde formula (cf.
\cite{Ver88}, \cite{MS89}). The kernel of the modular
representation associated with a RCFT is of particular interest.
It has been conjectured the kernel is always a congruence subgroup
of $\SL$ (cf. \cite{Moore87}, \cite{E95}, \cite{ES95},
\cite{DM96}, \cite{BCIR}), i.e.\ it contains some principal
congruence subgroup $\Gamma(N)$ of $\SL$. The conjecture was first
addressed by Coste and Gannon in \cite{CG99}, and they showed that
the conjecture holds if $\ft=\mtx{1&1\\0 & 1} \in \SL$ is
represented by an odd order matrix under the modular
representation. The conjecture was later established by Bantay in
\cite{Bantay03} under certain assumptions. More recently,
Xu also solved the conjecture for the modular representation associated
to a local conformal net \cite{Xu}.

The language of modular tensor categories, termed by I. Frenkel,
constitutes a formalization of the chiral data of a conformal
field theory (cf. \cite{MS90}, \cite{BaKi}).  A modular tensor
category may be thought of as the representation category of some
chiral algebra which corresponds to a conformal field theory.
Huang has proved this for some vertex operator algebras
\cite{HYZ05} (see also \cite{Lep05}). The recent progress in
representation theory has revealed that a modular tensor category over
an algebraically closed field $\k$ of characteristic zero can
always be realized as the representation category of some
connected ribbon factorizable semisimple weak Hopf algebra over
$\k$ (cf. \cite{Szl}, \cite{NTV}). Moreover, M\"uger has also
shown in \cite{MugerII03} that the center (quantum double)
$Z(\CC)$ of a spherical fusion category $\CC$ over $\k$ is
naturally a modular tensor category. In particular, the
representation category of a semisimple factorizable Hopf algebra
and the representation category of the Drinfeld double $D(H)$ of a
semisimple Hopf algebra $H$ are modular tensor categories.

Parallel to rational conformal field theory, each modular tensor
category $\A$ over $\k$ is associated with a natural projective
modular representation $\ol\rho_\A$ on $\KK_0(\A)\o_\BZ \k$, where
$\KK_0(\A)$ is the Grothendieck (fusion) ring of $\A$. This
projective modular representation is projectively equivalent to an
ordinary representation, but such a lifting is not unique. However,
if $\A=Z(\CC)$ for some spherical fusion category $\CC$, then
there exists a \emph{canonical ordinary} modular representation
$\rho_{Z(\CC)}$ which is a lifting of $\ol \rho_{Z(\CC)}$. It is
natural to ask whether the kernels of these canonical projective
or ordinary modular representations are congruence subgroups of
$\SL$.   These questions were answered affirmatively by
Sommerh\"auser and Zhu in \cite{SZh} for factorizable semisimple
Hopf algebras and the Drinfeld doubles of semisimple Hopf
algebras. In this paper, we generalize their results to
spherical fusion categories, and prove the congruence subgroup conjectures in Theorems \ref{t:cong1} and
\ref{t:cong2}. Moreover, every lifting of the projective
modular representation of a modular tensor category has a finite
image. We approach these questions by studying the generalized
Frobenius-Schur indicators for spherical fusion categories
introduced in this paper.

The classical notion of the second Frobenius-Schur (FS) indicators
for the representations of a finite group has been generalized to
many different contexts. A version for semisimple Hopf algebras
was introduced by Linchenko and Montgomery \cite{LM00}. A more
general version for semisimple quasi-Hopf algebras was studied by
Mason and Ng in \cite{MN05}, and Schauenburg in \cite{Sch04}.
Some categorical versions of the 2nd FS indicator were studied by
Fuchs, Ganchev, Szlach\'anyi, and Vescerny\'es in \cite{FGSV99}
and by Fuchs and Schweigert in \cite{FucSch:CTCBC}. Bantay also
introduced another version of the 2nd FS indicator for RCFT as a
formula in terms of the modular data. The less
well-known higher FS
indicators for the representations of a finite group were
generalized to semisimple Hopf algebras in \cite{LM00}, and have been studied extensively by Kashina, Sommerh\"auser
and Zhu \cite{KSZ}, and to semisimple quasi-Hopf algebras by Ng
and Schauenburg \cite{NS052}. All these FS indicators in different
contexts are specializations of the higher FS indicators
for pivotal categories introduced in \cite{NS05}.

The main tool employed in \cite{SZh} to prove the congruence
subgroup theorems is the equivariant indicators for semisimple
Hopf algebras, which are extensions of the higher FS indicators
for semisimple Hopf algebras. Their discovery suggests a more
general version of indicators for pivotal categories. In this
paper, we introduce the generalized Frobenius-Schur (GFS)
indicator $\nu_{m,l}^\bX(V)$ for a pair $(m,l)$ of integers, an
object $V$ of a pivotal category $\CC$ and an object $\bX$ in the
center $Z(\CC)$. For a pair of integers $m,l\in\mathbb Z$,
the indicator $\nu_{m,l}^\bX(V)$ is defined as the
trace of a linear endomorphism $E_{\bX,
V}^{(m,l)}$ on the vector space $\CC(X, V^{\o m})$ where $X$ is the
underlying $\CC$-object of $\bX$. If $\bX$ is the
unit object of $Z(\CC)$ and $m>0$, then $\nu_{m,l}^\bX(V)$ coincides with
the $(m,l)$-th FS indicator $\nu_{m,l}(V)$ of $V$ defined in
\cite{NS05}. In a spherical fusion category $\CC$, one can extend
the assignment $\bX \mapsto \nu_{m,l}^\bX(V)$ for each simple $\bX
\in Z(\CC)$ to a linear functional $I_V((m,l), -)$ on the fusion
algebra $\KK_\k(Z(\CC))=\KK_0(Z(\CC))\o_\BZ \k$ for each pair $(m,l)$
of integers and $V \in \CC$; this extension is called the
equivariant indicator in Section 6. It is
equivalent to the corresponding notion introduced by
Sommerh\"auser and Zhu when $\CC$ is the representation category
of a semisimple Hopf algebra.

The set of all equivariant indicators is closed under the
$\SL$-action on $\KK_\k(Z(\CC))^*$ induced by the contragredient of
the modular representation $\rho_{Z(\CC)}$. Moreover, the indicators are
invariant under the action of the principal congruence subgroup
$\Gamma(N)$, which is the kernel of the epimorphism $\SL \to
\SLN$, where $N$ is the Frobenius-Schur exponent of $\CC$. The
study of the relation between the equivariant indicators and the
modular representations associated with the center of a modular
tensor category leads to our major theorems. These theorems imply
that all the modular representations of a modular category have finite images, and a conjecture of Eholzer on  these representations.

In the course of studying the equivariant indicators, we obtain two formulae for the GFS indicators. The first formula, obtained in Corollary \ref{c:nu}, expresses $\nu_{m,1}^\bX(V)$ for a spherical fusion category $\CC$ in terms of the modular data of $Z(\CC)$; it contains the FS indicator formula discovered in \cite[Theorem 4.1]{NS07} as a special case. The second formula, described in Proposition \ref{p:Bantay} {which is a consequence of the first formula,} expresses $\nu_{m,1}^\bX(V)$ for a modular category $\A$  in terms of its modular data. It implies Bantay's indicator formula \cite{Bantay97} when we specify $m=2$ and $\bX$ to be the unit object of $Z(\A)$. More importantly, this formula suggests a close relationship between the GFS indicators and a family of scalars $Y_{ab}^c$ indexed by the primary fields $a,b,c$ of a RCFT introduced in \cite{PSS}. It is conjectured in \cite{PSS} that $Y_{ab}^c$ are integers and it is further conjectured in \cite{BHS} certain inequality holds  for $Y_{ab}^c$. Gannon has proved these conjectures under the condition that the $T$ matrix of the RCFT has odd order \cite{Gan}. More recently, Kac, Longo and Xu have proved these conjectures via $\BZ_2$-orbifolds of conformal nets \cite{KLX}.  As an application of GFS indicators, we prove these conjectures hold for all modular categories.

The organization of this paper is as follows: In Section \ref{s:prelim} we
cover some basic definitions, notations, conventions and
preliminary results on pivotal categories for the remaining
discussion. In Section \ref{s:def} we define the generalized FS indicators,
discuss their basic properties and an alternative characterization.
This continues in Section \ref{s:sscase} under the additional assumption that the
category is semisimple, and we give another characterization of
the GFS indicators for spherical fusion categories. In Section \ref{s:hopf}, we show how this characterization recovers the equivariant indicators introduced in
\cite{SZh} when the underlying spherical category is the representation category of a semisimple Hopf algebra.
We define
the equivariant indicators for a spherical fusion category in
Section \ref{s:equiv}. We show that the set of equivariant indicators admits
a natural action of $\SL$, and derive some important consequences
of this modular action. In Section \ref{s:congsub}, we study the equivariant
indicators for a modular tensor category and its center, and
prove the congruence subgroup theorems. We also provide an example for the congruence subgroups arising. The study of modular representations of a modular category continues in Section \ref{s:eholzer}. We prove the images of these representations are finite, and a conjecture of Eholzer for modular categories holds.
In Section \ref{s:conjecture}, we prove a conjecture of Pradisi-Sagnotti-Stanev and a conjecture of Borisov-Halpern-Schweigert using a generalized Bantay's formula for GFS indicators.
In Section \ref{s:endom}, we introduce the definition of generalized Frobenius-Schur endomorphisms in a pivotal fusion
category $\CC$. For relatively prime positive integers $m,l$ these turn
out to be natural endomorphisms of
the identity functor $\id_\CC$, and the corresponding GFS indicators can
be expressed as their pivotal traces. This is a generalization of the formulas expressing higher indicators as character values on certain central
elements of a quasi-Hopf algebra.

\section{Preliminaries}\label{s:prelim}
In this section, we will collect some conventions and facts on
pivotal categories. Most of these are quite well-known, and the
readers are referred to to \cite{NS05,NS052,NS07} and the
literature cited there. Additional key results on fusion categories and their centers are taken from M\"uger's work \cite{MugerII03}

\subsection{Pivotal and spherical monoidal categories}

In a monoidal category $\CC$ with tensor product $\o$, we denote
$\Phi\colon (U\o V)\o W\to U\o(V\o W)$ the associativity
isomorphism. If $X,Y\in\CC$ are obtained by tensoring together the
same sequence of objects with two different arrangements of
parentheses, one can obtain an isomorphism between them by
composing several instances of the tensor products of $\Phi$,
$\Phi\inv$ and the identity. It is unique by the coherence theorem, and will be denoted by $\Phi^?\colon X\to Y$.

We will assume throughout that the unit object $I \in \CC$ is strict. A left dual of an
object $V \in \CC$ is an object $V\du \in \CC$ together with the
morphisms $\ev\colon V\du\o V\to I$ and $\db\colon I\to V\o V\du$
such that
\begin{gather*}
\id_V=\left(V\xrightarrow{\db\o V}(V\o V\du)\o V\xrightarrow{\Phi}V\o(V\du\o
V)\xrightarrow{V\o\ev}V\right),
\\
\id_{V\du}=\left(V\du\xrightarrow{V\du\o\db}V\du\o(V\o
V\du)\xrightarrow{\Phi\inv}(V\du\o V)\o V\du\xrightarrow{\ev\o
V\du}V\du\right).
\end{gather*}
A right dual of an object can be defined similarly. A monoidal
category $\CC$ is called left (resp.\ \ right) rigid if every object
of $\CC$ admits a left (resp.\ \ right) dual. If $\CC$ is a left
rigid monoidal category, then taking duals can be extended to a monoidal
functor $(-)\du: \CC\to \CC^\op$, and so $(-)\bidu\colon\CC\to\CC$
is consequently a monoidal functor. Moreover, we can choose
$I\du=I$ and $\ev_I =\db_I =\id_I$.

A \emph{pivotal category} is a left rigid  monoidal  category
equipped with an isomorphism $j\colon Id \to (-)\bidu$, called a
pivotal structure, of monoidal functors.

Let $\CC$ be a pivotal category, and $V \in \CC$. Then $V\du$
together with the morphisms
$$
\ol \db:=\left(I \xrightarrow{\db} V\du \o V\bidu\xrightarrow{V\du \o j_V\inv} V\du \o V\right),
\text{ and }
\ol \ev:=\left(V \o V\du\xrightarrow{j_V\o V\du} V\bidu \o V\du \xrightarrow{\ev} I\right)
$$
becomes a right dual of $V$. In particular, $\CC$ is also right
rigid. Let $f\colon V\to V$ be a morphism in the pivotal category
$\CC$. The left and right pivotal traces of $f$ are respectively
\begin{gather*}
\ptrr(f)=\bigg(I\!\xrightarrow{\db}\!V\o
V\du\!\xrightarrow{f\o \id}\!V\o V\du\! \xrightarrow{\ol\ev}\!I\bigg)\text{ and }
\ptrl(f)=\bigg(I\!\xrightarrow{\ol\db}\!V\du\o V\!\xrightarrow{\id\o f}\! V\du\o
V\xrightarrow{\ev}\!I\bigg).
\end{gather*}
The left and right pivotal dimensions of $V\in\CC$ are
$d_\ell(V)=\ptr^\ell(\id_V)$ and $d_r(V)=\ptr^r(\id_V)$.

A \emph{spherical category} is a pivotal category in which the left and
right pivotal traces of every morphism are identical. In a
spherical category, the pivotal traces and dimensions will be denoted by
$\ptr(f)$ and $d(V)$, respectively.

A pivotal category is called \emph{strict} if the associativity
isomorphism $\Phi$, the pivotal structure $j$, and the canonical
isomorphism $(V\o W)\du \to W\du \o V\du$ are identities. It has
been shown in \cite[Theorem 2.2]{NS05} that every pivotal category
$\CC$ is equivalent to a strict pivotal category $\CCstr$;
equivalence as pivotal categories means that the monoidal
equivalence $\CC\to \CCstr$ preserves pivotal structures in a
suitable sense \cite{NS05}. If $\CC$ is spherical, then so is
$\CCstr$.

In a strict pivotal category, we make free use of graphical
calculus. Our convention for a morphism is a diagram with the
source at the top and the target at the bottom.  For instance, the
morphisms $\ev: V\du \o V\to I$, $\db: I \to V \o V\du$, $\ol\ev:
V \o V\du\to I$ and $\ol \db: I \to V\du \o V$ are respectively
the diagrams:
$$
\def\objectstyle{\scriptstyle}
\xy (0,-1)="ctext",,
{"ctext"+(-4,4)="v1"*{V\du}; "ctext"+(4,4)="v2"*{V} **\crv{"v1"+(0,-7)&"v2"+(0,-7)}}\endxy, \quad
\xy (0,-1)="ctext",,
{"ctext"+(-4,0)="v1"*+{V}; "ctext"+(4,0)="v2"*+{V\du} **\crv{"v1"+(0,8)&"v2"+(0,8)}}\endxy,\quad
\xy (0,-1)="ctext",,
{"ctext"+(-4,4)="v1"*{V}; "ctext"+(4,4)="v2"*{V\du} **\crv{"v1"+(0,-7)&"v2"+(0,-7)}}\endxy,\quad\text{and}\quad
\xy (0,-1)="ctext",,
{"ctext"+(-4,0)="v1"*+{V\du}; "ctext"+(4,0)="v2"*+{V} **\crv{"v1"+(0,8)&"v2"+(0,8)}}\endxy.
$$
Notice that $\ol\ev_V = \ev_{V\du}$ and $\ol\db_V = \db_{V\du}$ in a strict pivotal category.

The left and right pivotal traces of a morphism $f:V \to V$ are given by the diagrams:
$$
  \ptrl(f)=
  \def\objectstyle{\scriptstyle}
  \xy (0,0)="ctext"*{f},
  {"ctext"+(-2,-2); "ctext"+(2,2) **\frm{-}},
  {"ctext"+(-7,-2)="v1"; "ctext"+(0,-2)="v2" **\crv{"v1"+(0,-5)&"v2"+(0,-5)}},
  {"ctext"+(-7, 2)="w1"; "ctext"+(0,2)="w2" **\crv{"w1"+(0,5)&"w2"+(0,5)}},
  {"w1"; "v1" **\dir{-}},
  "ctext"+(-9, 0)*{V\du}
  \endxy \quad\text{and}\quad
  \ptrr(f)=\,
  \xy (0,0)="ctext", "ctext"+(-7,0)*{f}="ctext1",
  {"ctext1"+(-2,-2); "ctext1"+(2,2) **\frm{-}},
  {"ctext"+(-7,-2)="v1"; "ctext"+(0,-2)="v2" **\crv{"v1"+(0,-5)&"v2"+(0,-5)}},
  {"ctext"+(-7, 2)="w1"; "ctext"+(0,2)="w2" **\crv{"w1"+(0,5)&"w2"+(0,5)}},
  {"w2"; "v2" **\dir{-}},
  "ctext"+(3, 0)*{V\du},
  \endxy.
$$
If the pivotal category is spherical, the two pivotal traces coincide.

Now let $\CC$ be a left rigid braided monoidal category. In the graphical calculus, we depict the braiding $c$ and its inverse as
$$
 c_{VW}=
  \def\objectstyle{\scriptstyle}
  \xy
  {\vcross~{(-2,3)*{V}}{(2,3)*{W}}{(-2,-3)*{W}}{(2,-3)*{V}}},
  \endxy \quad\text{and}\quad
   c_{VW}\inv=
  \xy
  {\vcrossneg~{(-2,3)*{W}}{(2,3)*{V}}{(-2,-3)*{V}}{(2,-3)*{W}}}
  \endxy\,.
$$
Associated to $c$ is the Drinfeld  isomorphism $u:
Id \rightarrow (-)\bidu$ defined by
\begin{multline*}
 u_V:= \bigg(V \xrightarrow{\db\o\id} (V\du\o V\bidu)\o V \xrightarrow{\Phi}
  V\du\o (V\bidu\o V)\xrightarrow{\id \o c\inv}V\du\o (V\o V\bidu)\\ \xrightarrow{\Phi\inv}
  (V\du\o V)\o V\bidu \xrightarrow{\ev\o\id} V\bidu\bigg)
\end{multline*}
which also satisfies
$$
u_{V\o W}=(u_V\o u_W)c_{VW}\inv c_{WV}\inv \quad \text{for } V, W \in \CC.
$$
In particular, the equation
$\theta= u\inv j$ describes a one-to-one correspondence between pivotal structures $j$ and twists $\theta$ on $\CC$. Here, a twist is by definition an automorphism of the
identity functor on $\CC$ satisfying
\begin{equation} \label{def:twist}
\theta_{V\o W}=(\theta_V\o\theta_W)c_{WV}c_{VW} \quad
\text{and}\quad \theta_I=\id_I\,.
\end{equation}
For a strict pivotal category with a braiding $c$, the
Drinfeld isomorphism and the associated twist $\theta$ are
respectively given by
$$
u_V=
\def\objectstyle{\scriptstyle}
\xy
{\vcrossneg~{(-2,2)="x1"}{(2,2)="x2"}{(-2,-2)="y1"}{(2,-2)="y2"}},
{"x1"+(-4, 0)="x0"; "x1"**\crv{"x0"+(0,2.5)& "x1"+(0,2.5)}},
{"y1"+(-4, 0)="y0"; "y1"**\crv{"y0"+(0,-2.5)& "y1"+(0,-2.5)}},
{"x0"; "y0"**\dir{-}?(.5)+(-3,0)*{V\du}},
{"y2"+(0,-4)*{V}; "y2"**\dir{-}},
{"x2"+(0,4)*{V}; "x2"**\dir{-}},
\endxy
\quad\text{and}\quad\theta_V = u_V\inv =
\xy
{\vcross~{(-2,2)="x1"}{(2,2)="x2"}{(-2,-2)="y1"}{(2,-2)="y2"}},
{"x2"; "x2"+(4, 0)="x3"; **\crv{"x3"+(0,2.5)& "x2"+(0,2.5)}},
{"y2"; "y2"+(4, 0)="y3"; **\crv{"y3"+(0,-2.5)& "y2"+(0,-2.5)}},
{"x3"; "y3"**\dir{-}?(.5)+(2.5,0)*{V\du}},
{"y1"+(0,-5)*+{V}; "y1"**\dir{-}},
{"x1"+(0,5)*+{V}; "x1"**\dir{-}},
\endxy\,.
$$

A twist $\theta$ on $\CC$ is called a ribbon structure if it
satisfies $\theta\du_V = \theta_{V\du}$. The triple $(\CC, c,
\theta)$ is called a ribbon category if $\theta$ is a ribbon
structure on the braided monoidal category $\CC$ with the braiding
$c$. In a ribbon category, the associated pivotal structure on
$\CC$ is spherical. If $\CC$ is strict, then the
associated ribbon structure $\theta$ can be depicted as
 $$
 \theta_V \quad =\quad \def\objectstyle{\scriptstyle}
\xy
{\vcross~{(-2,2)="x1"}{(2,2)="x2"}{(-2,-2)="y1"}{(2,-2)="y2"}},
{"x2"; "x2"+(4, 0)="x3"; **\crv{"x3"+(0,2.5)& "x2"+(0,2.5)}},
{"y2"; "y2"+(4, 0)="y3"; **\crv{"y3"+(0,-2.5)& "y2"+(0,-2.5)}},
{"x3"; "y3"**\dir{-}?(.5)+(2.5,0)*{V\du}},
{"y1"+(0,-5)*+{V}; "y1"**\dir{-}},
{"x1"+(0,5)*+{V}; "x1"**\dir{-}},
\endxy \quad =\quad
\xy
{\vcross~{(-2,2)="x1"}{(2,2)="x2"}{(-2,-2)="y1"}{(2,-2)="y2"}},
{"x1"+(-4, 0)="x0"; "x1"**\crv{"x0"+(0,2.5)& "x1"+(0,2.5)}},
{"y1"+(-4, 0)="y0"; "y1"**\crv{"y0"+(0,-2.5)& "y1"+(0,-2.5)}},
{"x0"; "y0"**\dir{-}?(.5)+(-3,0)*{V\du}},
{"y2"+(0,-4)*{V}; "y2"**\dir{-}},
{"x2"+(0,4)*{V}; "x2"**\dir{-}},
\endxy\,.
$$

\subsection{The center construction}

The (left) center $Z(\CC)$ of a monoidal category $\CC$ is a category whose objects are
pairs $\mathbf{X}=(X, \hbr_X)$ in which $X$ is an object of $\CC$ and the \emph{half-braiding} $\hbr_{X}(-): X \o (-)\;
\rightarrow \; (-) \o X$ is a natural isomorphism satisfying the
properties $\hbr_X(I)=\id_X$ and
$$
(V\o \hbr_X(W))\circ \Phi_{V,X,W} \circ (\hbr_X(V) \o W)  = \Phi_{V,W,
X}\circ \hbr_X(V \o W)\circ \Phi_{X,V,W}
$$
for all $V, W \in \CC$. We will often write $\hbr_{X,V}$ in place of $\hbr_X(V)$. It is well-known that $Z(\CC)$ is a braided monoidal category (cf. \cite{Kassel}). The tensor product
$
  (X, \hbr_X)\o (Y, \hbr_Y) := (X \o Y, \hbr_{X \o Y}),
$ of two objects $(X,\hbr_X)$ and $(Y,\hbr_Y)$, and the unit object
$(I, \hbr_I)$ are given by
$$
\hbr_{X \o Y}(V) = \Phi_{V, X, Y} \circ (\hbr_X(V) \o Y)\circ  \Phi_{X, V,
Y}\inv \circ (X \o \hbr_Y(V)) \circ \Phi_{X, Y, V}
$$
and $\hbr_I(V)=\id_V$ for any $V \in \CC$.
The associativity isomorphisms are
inherited from $\CC$, so that the forgetful functor $Z(\CC)\to\CC$ is a strict monoidal functor.
The canonical braiding
on $Z(\CC)$ is given by $c_{\mathbf{X},\mathbf{Y}}=\hbr_X(Y)$ for
$\mathbf{X}=(X, \hbr_X), \mathbf{Y}=(Y, \hbr_Y) \in Z(\CC)$. If there is no danger of confusion with a previously given braiding on $\CC$, we will depict the half-braiding of $Z(\CC)$ by
$$
 \hbr_{X}(V)=
  \def\objectstyle{\scriptstyle}
  \xy
  {\vcross~{(-2,3)*{X}}{(2,3)*{V}}{(-2,-3)*{V}}{(2,-3)*{X}}},
  \endxy \quad\text{and}\quad
   \hbr_X(V)\inv=
  \xy
  {\vcrossneg~{(-2,3)*{V}}{(2,3)*{X}}{(-2,-3)*{X}}{(2,-3)*{V}}}
  \endxy\,.
$$

If $\CC$ is left rigid, then so is $Z(\CC)$. If $\CC$ is a
pivotal (resp.\  spherical) category, then $Z(\CC)$ is also a
pivotal (resp.\  spherical) category with the pivotal structure
inherited from $\CC$. Any equivalence $\FF\colon\CC\to\mathcal D$ of monoidal categories naturally
induces an equivalence $\hat\FF\colon Z(\CC)\to
Z(\mathcal D)$ of braided monoidal categories. In addition, if $\CC$
and $\DD$ are pivotal categories and $\FF$ preserves their
pivotal structures, then $\hat\FF\colon Z(\CC)\to Z(\mathcal D)$
also preserves their pivotal structures and the twists  associated with their pivotal
structures.

There is a one-to-one correspondence between braidings on a monoidal category and sections of the forgetful functor $Z(\CC)\to\CC$, where the section $\CC\to Z(\CC)$ corresponding to a braiding $c$ maps $X\in\CC$ to
$(X,\hbr_X)\in Z(\CC)$ with $\hbr_X(V)=c_{X,V}$. Since the inverse of a braiding gives another braiding, we can combine the two resulting sections of the forgetful functor to yield a functor $\CC\times\CC\to Z(\CC)$ which maps $(X, Y)\in\CC \times \CC$ to $(X\o Y, \hbr_{X\o Y})$ given in the strict case by
\begin{equation}\label{doublebraid}
 \hbr_{X\o Y}(V)=(c_{XV}\o Y)(X\o c_{VY})\inv=
\def\objectstyle{\scriptstyle}
\xy
{(0,0)="c2"},
{"c2"+(-4, 0)="c1"}, {"c2"+(4,0)="c3"},
{\vcrossneg~{"c2"+(0, 6)="y2"*{Y}}{"c3"+(0, 6)="y3"*{V}}{"c2"}{"c3"}},
{\vcross~{"c1"}{"c2"}{"c1"+(0, -6)="x1"*{V}}{"c2"+(0,-6)="x2"*{X}}},
{"c1"+(0,6)*{X}; "c1"**\dir{-}},
{"c3"+(0,-6)*{Y}; "c3"**\dir{-}},
\endxy\,.
\end{equation}
If $\CC$ is left rigid, one can check that the dual of $(X\o Y,\sigma_{X\o Y})$ is
\begin{equation}\label{dualdoublebraid}
 (X\o Y,\sigma_{X\o Y})\du=(X\du\o Y\du,\sigma_{X\du\o Y\du})
\end{equation}
with evaluation  and coevaluation  morphisms given by
$$
 \def\objectstyle{\scriptstyle}
\xy
{(0,0)="z"},
{"z"+(-3, 5)="a1"}, {"a1"+(6,0)="a2"},
{"z"+(-3, -5)="b1"}, {"b1"+(6,0)="b2"},
{\vcrossneg~{"a1"*{Y\du}}{"a2"*{X}}{"b1"}{"b2"}},
{"b1";"b1"+(-6,0)="b0"**\crv{"b1"+(0,-2.5)& "b0"+(0,-2.5)}},
{"b2";"b2"+(6,0)="b3"**\crv{"b2"+(0,-2.5)& "b3"+(0,-2.5)}},
{"b0";"b0"+(0,10)*{X\du}**\dir{-}},
{"b3";"b3"+(0,10)*{Y}**\dir{-}}
\endxy
 \quad\text{ and }
 \xy
{(0,0)="z"},
{"z"+(-3, 5)="a1"}, {"a1"+(6,0)="a2"},
{"z"+(-3, -5)="b1"}, {"b1"+(6,0)="b2"},
{\vcrossneg~{"a1"}{"a2"}{"b1"*{Y}}{"b2"*{X\du}}},
{"a1";"a1"+(-6,0)="a0"**\crv{"a1"+(0,2.5)& "a0"+(0,2.5)}},
{"a2";"a2"+(6,0)="a3"**\crv{"a2"+(0,2.5)& "a3"+(0,2.5)}},
{"a0";"a0"+(0,-10)*{X}**\dir{-}},
{"a3";"a3"+(0,-10)*{Y\du}**\dir{-}}
\endxy\,.
$$

\subsection{$\k$-linear and semisimple monoidal categories}\label{s:1.3}

Almost all results obtained in this paper pertain to $\k$-linear monoidal categories, where we assume throughout that $\k$ is an algebraically closed
field of characteristic zero, although it may be worth to note that
Sections \ref{s:def} and \ref{s:sscase} do not require any additional assumptions on the field $\k$.

We also fix the convention that a $\k$-linear monoidal category $\CC$ is said to be semisimple if the underlying $\k$-linear category is semisimple with finite-dimensional morphism spaces, and \emph{the
unit object $I$ is simple}. Following \cite{ENO}, a fusion
category over $\k$ is a semisimple left rigid $\k$-linear monoidal
category with finitely many simple objects.

Note that if $\CC$ is a pivotal category over $\k$ and $I$ is
absolutely simple, then the pivotal traces $\ptrl(f)$ and $\ptrr(f)$ of an endomorphism $f$, which were defined as endomorphisms of $I$, can be identified with scalars in $\k$. In this case we use the pivotal
traces to define bilinear pairings
\[(\cdot,\cdot)_\ell, (\cdot, \cdot)_r: \CC(U, V)\times \CC(V, U) \to
\k\] by
\begin{equation}\label{eq:pairing}
(f,g)_\ell:= \ptrl(f\circ g) ,\quad \text{and}\quad (f,g)_r:=
\ptrr(f\circ g)
\end{equation}
Note that $(f,g)_\ell=(g,f)_\ell$ holds  (cf.\ \cite{NS05}), and these bilinear pairings coincide when $\CC$ is spherical. In this case, we simply denote $(f,g)_r$ or $(f,g)_\ell$ by $(f,g)$.
If $\CC$ is also semisimple, then the pairings in \eqref{eq:pairing} are non-degenerate (cf. \cite{GK96}). It follows that for a braided spherical semisimple $\k$-linear category the twist associated to the pivotal structure is always a ribbon structure.

For any object $V$ in a pivotal fusion category $\CC$ over $\k$,
we write $[V]$ for its isomorphism class. If $\{[V_i]\mid i \in
\Gamma\}$ denotes the finite set of isomorphism classes of simple
objects in $\CC$, then the index set $\Gamma$ is always assumed to
contain $0$ by setting $V_0$ to be the unit object of $\CC$. For
$i \in \Gamma$, we define $\ol i \in \Gamma$ by the isomorphism
$V_i\du \cong V_{\ol i}$. By \cite[Theorem 2.3 and Proposition 2.9]{ENO}, the
pivotal dimension $d_\ell(V_i)$ of $V_i$ is a non-zero algebraic
integer in $\k$, and
$$
\dim \CC = \sum_{i \in \Gamma} d_\ell(V_i) d_\ell(V_i\du) \ne 0\,.
$$
In addition, if $\CC$ is spherical, then $d(V_i)= d(V_i\du)$ (cf. \cite[Corollary 2.10]{ENO}) and so
$$
\dim \CC = \sum_{i \in \Gamma} d(V_i)^2  \ne 0\,.
$$

We denote by $\KK_0(\CC)$ the Grothendieck ring of the fusion category $\CC$, and by $\KK_\k(\CC):=\KK_0(\CC)\o_{\mathbb Z}\k$ its Grothendieck algebra.
The set $\{[V_i]\mid i \in \Gamma\}$ is a basis for the free $\BZ$-module $\KK_0(\CC)$, and (with the obvious identification) a $\k$-basis of $\KK_\k(\CC)$; we will refer to it as the canonical basis.

The center of a $\k$-linear monoidal category is itself $\k$-linear monoidal in the obvious way.
By \cite[Theorem 3.16, Corollary 5.6]{MugerII03}
(see also \cite[Theorem 2.15]{ENO}) the center
of a pivotal (resp.\  spherical) fusion category over $\k$ is also a
pivotal (resp.\  spherical) fusion category over $\k$ (\cite{MugerII03} assumes a spherical category, but note \cite[Remark 3.17]{MugerII03}). Moreover, by  \cite[Proposition 8.1]{MugerII03} the
forgetful functor $F: Z(\CC) \to \CC$, which maps $\mathbf{X}=(X,
\hbr_X)$ to $X$, admits a two-sided adjoint functor $K\colon \CC\to Z(\CC)$.

Alternatively, \cite{Szl} and \cite{NTV} also imply the existence of a left adjoint $K$ to $F$. Consider such an adjoint, and a natural isomorphism $\ol \Psi_{V, \bX}: \CC(V, X) \to  \DD(K(V), \bX)$. Since $\CC$ is a pivotal fusion category over $\k$,  so is $\DD=Z(\CC)$ by \cite[Theorem 2.15]{ENO}. Therefore, the bilinear forms  $(\cdot, \cdot)_\ell$ defined in $\DD$ is non-degenerate. Let
$\Psi_{\bX, V}: \DD(\bX, K(V))\to \CC(X, V)$ be the adjoint operator of $\ol \Psi_{V, \bX}$ with respect to  $(\cdot, \cdot)_\ell$, i.e.
\begin{equation}\label{arrangement}
(\ol \Psi (f), g)_\ell  = (f, \Psi(g))_\ell \quad \text{for all } f \in \CC(V, X),\, g\in \DD(\bX, K(V))\,.
\end{equation}
By the naturality of $\ol \Psi$, the linear isomorphism $\Psi$ is also  natural in $\bX$ and $V$. In particular, $K$ is a right adjoint to $F$, and $\Psi$ is an associated adjunction isomorphism.

In fact if either of the adjunction isomorphisms $\Psi$ for the right or $\ol\Psi$ for the left adjoint is given, we can define the other one by \eqref{arrangement}, which will define a natural isomorphism because $(\cdot ,\cdot)_\ell$ is nondegenerate.


\subsection{Modular categories}\label{s:modular}
A modular tensor category over $\k$ (see \cite[Chapter 3]{BaKi}), also simply called a modular category,  is a
ribbon fusion category $\A=(\A, c, \theta)$ over $\k$ such that, for the set $\{[U_i]\mid i \in \Pi\}$ of isomorphism classes of simple objects,
the matrix $S=[S_{ij}]_{\Pi}$ defined by
\begin{equation}\label{eq:Smatrix}
S_{ij} = \ptr \left(c_{U_j, U_{\ol i}}\circ c_{U_{\ol i}, U_j}\right)
\end{equation}
is non-singular. This matrix $S$ is called the \emph{$S$-matrix} of $\A$. In the strict case, $S_{ij}$ can be depicted as \\
$$
S_{ij} =
\def\objectstyle{\scriptstyle}
\xy
{(0,0)="ctext"},
{"ctext"+(-8, 6)="a1"}, {"ctext"+(-2,6)="a2"},
{"ctext"+(2,6)="a3"}, {"ctext"+(8,6)="a4"},
{"ctext"+(-2,0)="m2"}, {"ctext"+(2,0)="m3"},
{"ctext"+(-8, -6)="b1"}, {"ctext"+(-2,-6)="b2"},
{"ctext"+(2,-6)="b3"}, {"ctext"+(8,-6)="b4"},
{\vcross~{"a2"}{"a3"}{"m2"}{"m3"}},
{\vcross~{"m2"}{"m3"}{"b2"}{"b3"}},
{"a3"; "a4"*\crv{"a3"+(0,3) & "a4"+(0,3)}},
{"b3"; "b4"*\crv{"b3"+(0,-3) & "b4"+(0,-3)}},
{"a1"; "a2"*\crv{"a1"+(0,3) & "a2"+(0,3)}},
{"b1"; "b2"*\crv{"b1"+(0,-3) & "b2"+(0,-3)}},
{"a4";"b4"**\dir{-}?(0.3)+(3,0)*{U_j\du}},
{"a1";"b1"**\dir{-}?(0.1)+(-2,0)*{U_i}},
\endxy \vspace{8pt}\,.
$$

Let $\theta_{U_i}=\w_i \id_{U_i}$ for some $\w_i \in \k$. The matrices $T$ and $C$ (\emph{charge conjugation matrix}) of $\A$ are defined as
$$
T=[\delta_{ij} \w_i]_{\Pi} \quad \text{and} \quad
C=[\delta_{i\ol j}]_{\Pi}.
$$
These matrices $S$, $T$ and $C$ satisfy the conditions:
\begin{equation}\label{eq:STrelations}
(ST)^3 = p_\A^+ S^2, \quad S^2=p_\A^+ p_\A^- C, \quad CT=TC,\quad C^2=\id,
\end{equation}
where $p_\A^{\pm} = \sum_{i\in \Pi} \w_i^{\pm 1} d(U_i)^2$ are called the \emph{Gauss sums} of $\A$.
Note that $p_\A^{\pm}$ are non-zero scalars and
\begin{equation}\label{eq:gausssum}
p_\A^+ p_\A^- = \sum_{i\in \Pi} d(U_i)^2 =\dim \A.
\end{equation}
By \cite{Vafa88},  $\w_i$ and the quotient $\frac{p_\A^+}{p_\A^-}$
are roots of unity, and so $T^N=1$ where $N=\ord \theta$.

Recall that the modular group $\SL$ is the group generated by
$$
\fs= \mtx{0 & -1\\ 1 & 0}, \quad \ft=\mtx{1 & 1\\ 0 & 1} \text{ with defining relations }
(\fs\ft)^3=\fs^2 \text{ and } \fs^4=1.
$$
Therefore, the relations \eqref{eq:STrelations} imply that
\begin{equation}\label{eq:projrep}
  \ol\rho_\A\colon\SL\to\PGL(\KK_\k(\A));\quad \fs\mapsto S\quad \text{and}\quad \ft\mapsto T
\end{equation}
defines a projective representation of $\SL$ on the Grothendieck algebra of $\A$, where we identify the $S$ and $T$ matrices with automorphisms of $\KK_\k(\A)$ using the latter's canonical basis of simple objects. This projective representation will be called as the \emph{projective modular representation of} $\A$.

The projective representation \eqref{eq:projrep} can be lifted to an ordinary representation
\begin{align}\label{eq:repC}
 \rho_\A^{\lambda, \zeta}\colon\SL\to\GL(\KK_\k(\A));\quad\fs\mapsto s:=\frac 1\lambda S\quad\text{and}\quad\ft\mapsto t:=\frac 1\zeta T,
 \end{align}
by choosing scalars $\lambda, \zeta \in \k$ such that
\begin{equation}\label{eq:ls}
\lambda^2=\dim \A\quad\text{and}\quad \zeta^3=\dfrac{p_\A^+}{\lambda}.
\end{equation}

It follows from \eqref{eq:STrelations} that
\begin{equation}\label{eq:strelations}
(st)^3 = s^2,\quad s^2 =C, \quad \text{and}\quad
s^4=1.
\end{equation}
The following well-known properties of the matrix $s=[s_{ij}]_{\Pi}$ will be used frequently (cf. \cite[Chapter 3]{BaKi}) : for $i,j \in \Pi$,
\begin{equation}\label{eq:Sprop}
 s_{0i} = d(U_i)/\lambda, \quad s_{ij}=s_{ji}=s_{\ol i\, \ol j}, \quad\text{and} \quad
s^{-1}=[s_{\ol ij}]_{\Pi}\,.
\end{equation}

By M\"uger's results \cite{MugerII03}, the center of a spherical fusion category is a modular fusion category, whose Gauss sums are $p_{Z(\CC)}^\pm=\dim \CC$. In this case, the projective representation \eqref{eq:projrep} for $\A=Z(\CC)$ can be lifted in a canonical way to an ordinary representation by choosing $\lambda=\dim \CC$ and $\zeta=1$, which satisfy \eqref{eq:ls}. We will call $s=\dfrac{1}{\dim\CC} S$ the \emph{normalized} $S$-matrix of $Z(\CC)$. This ordinary representation
$$\rho_{Z(\CC)}: \SL \to \GL(\KK_\k(Z(\CC))); \quad \fs \mapsto \dfrac{1}{\dim\CC} S\quad\text{and}\quad \ft \mapsto T$$
is called the \emph{canonical modular representation of} $Z(\CC)$.

The center of a modular fusion category $\A=(\A, c, \theta)$ can be described explicitly as follows: Let $\{U_i\mid i \in \Pi\}$ be a complete set of non-isomorphic simple objects.Then by \cite{MugerII03},
$\{(U_i\o U_j, \hbr_{U_i\o U_j})\mid i,j \in \Pi\}$ is a complete set of non-isomorphic simple objects of $Z(\A)$, where the half-braiding $\hbr_{U_i\o U_j}$ is defined in \eqref{doublebraid}. In other words, we have isomorphisms
\begin{align}
 \KK_0(\A)\ou{\BZ}\KK_0(\A)&\to \KK_0(Z(\A)), \quad
 \KK_\k(\A)\o\KK_\k(\A)\to\KK_\k(Z(\A))\label{tpiso};\\
 [U_i]\o [U_j]&\mapsto  [\bU_{ij}]:=[(U_i\o U_j,\hbr_{U_i\o U_j})].\notag
\end{align}
Note that \eqref{dualdoublebraid} implies
\begin{equation}\label{dualUij}
\bU_{ij}\du\cong \bU_{\ol i\ol j}.
\end{equation}

\section{Generalized Frobenius-Schur indicators}\label{s:def}

In this section, we introduce the definition of generalized
Frobenius-Schur indicators for each object in a pivotal category
over the field $\k$, and we derive some properties of these
indicators from the definition.

Let $\CC$ be a pivotal category over $\k$. For $V \in \CC$ and $m\in\BZ$ we define $V^m\in\CC$ by setting $V^0=I$ and $V^m=V \o V^{m-1}$ if $m>0$, and $V^{m}:=(V\du)^{-m}$ for $m<0$. Duality $(-)\du$ is a contravariant monoidal functor with respect to a canonical isomorphism $\xi:Y\du \o X\du \to (X
\o Y)\du$ coherent with the associativity isomorphisms $\Phi$ in
$\CC$. For a non-negative integer $m$ and $V \in \CC$, there
exists, by the coherence theorem, a unique isomorphism $t_m:
V^{-m} \to (V^m)\du$ which is a composition of instances of tensor
products of $\id$, $\Phi^{\pm 1}$ and $\xi^{\pm 1}$. Combining
with the pivotal structure $j$ of $\CC$, we can extend the
definition to negative $m$ as follows:
$$
t_m:=\left(V^{-m}\xrightarrow
j(V^{-m})\bidu\xrightarrow{t_{-m}\du}(V^m)\du\right)\,.
$$
Using $t_m$ we define, for any $m\in\Z$:
\[\ev_m:=\left(V^{-m}\o
V^m\xrightarrow{t_m\o \id}(V^m)\du\o V^m\xrightarrow\ev I\right)\]
and
\[\db_m:=\left(I\xrightarrow\db V^m\o(V^m)\du\xrightarrow{\id\o t_m\inv}V^m\o V^{-m}\right).\]
Note that if $\CC$ is  a strict pivotal category, then $t_m$ is
the identity, $\ev_m=\ev_{V^m}$, and $\db_m=\db_{V^m}$.

Next, for any $m, l\in\Z$, there is a canonical morphism
$J_{m,l}(V)\colon V^{-l}\o (V^m\o V^l)\to V^m$ defined using only
evaluation and coherence. More precisely,
$$
J_{m,l}(V):=\left\{
\begin{array}{ll}
V^{-l}\o(V^m\o V^l)\xrightarrow{\Phi^?}
(V^{-l}\o V^{l})\o V^m\xrightarrow{\ev_l\o \id}V^m &\text{if}\quad ml \ge 0, \\ \\
V^{-l}\o(V^m\o V^l)\xrightarrow{\Phi^?}V^m\o(V^{-l}\o
V^l)\xrightarrow{\id\o\ev_{l}}V^m & \text{if}\quad  ml \leq 0.
\end{array}
\right.
$$
Note that there is no difference between these two expressions for
$J_{m,l}(V)$ if $ml=0$. We  write $J_{m,l}$ for $J_{m,l}(V)$ when the context is clear.

Now for $\bX=(X,\hbr_X)\in Z(\CC)$, $V\in\CC$, and $l\in\Z$ set
\begin{multline*}
   D_{\bX,l}:=\bigg(X\xrightarrow{X\o\db_{-l}}X\o (V^{-l}\o
V^l)\xrightarrow{\Phi\inv} (X \o V^{-l}) \o V^l\\
\xrightarrow{\hbr_X\o V^l}(V^{-l}\o X) \o V^l  \xrightarrow{\Phi}
V^{-l}\o (X \o V^l) \bigg)\,.
\end{multline*}

Finally for $m,l\in\Z$, $\bX=(X,\hbr_X)\in Z(\CC)$, $V\in\CC$, we
define the $\k$-linear map $E_{\bX,V}^{(m,l)}: \CC(X,V^m) \to
\CC(X,V^m)$ as
\begin{align*} E_{\bX,V}^{(m,l)}(f)&:=
\left( X\xrightarrow{D_{\bX,l}}V^{-l}\o(X\o
V^l)\xrightarrow{V^{-l}\o f\o V^l}V^{-l}\o (V^m\o
V^l)\xrightarrow{J_{m,l}(V)}V^m\right)
\end{align*}
for $f \in \CC(X, V^m)$. It will sometimes be convenient to write
$E_{\bX,V}^{(m)}:=E_{\bX,V}^{(m,1)}$ for $m>0$ and
$F_{\bX,V}:=E_{\bX,V}^{(0,1)}$.

\begin{defn} \label{def:GFS}
 Let $\CC$ be a pivotal category over $\k$.  For $\bX \in Z(\CC)$, $V \in \CC$
 and $(m, l) \in \BZ \times \BZ$, we define the
 \textbf{generalized Frobenius-Schur (GFS) indicator}
 \begin{equation}
 \nu_{m,l}^{\bX}(V)=\Tr\left(E_{\bX,V}^{(m,l)}\right).
 \end{equation}
\end{defn}

\begin{remark}
  Let $\FF\colon\CC\to\mathcal D$ be an equivalence of monoidal
  categories, with the induced equivalence
  $\hat\FF\colon Z(\CC)\to Z(\DD)$ of braided monoidal categories.
  Consider $V\in\CC$ and $\bX\in Z(\CC)$.
  Similar to the reasoning in \cite[Section 4]{NS05}, the
  endomorphisms $E_{\bX,V}^{(m,l)}$ of
  $\CC(X,V^m)$ and $E_{\hat\FF(\bX),\FF(V)}^{(m,l)}$ of
  $\DD(\FF(X),\FF(V)^m)$ are conjugate. More precisely, by the
  general coherence results for monoidal functors there is a
  unique isomorphism
  $\xi^?\colon\FF(V^m)\to\FF(V)^m$ composed from instances of the
  monoidal functor structures of $\FF$, and (if $m<0$) the
  canonical isomorphism $\FF(V\du)\to\FF(V)\du$. For
  $f\colon X\to V^m$ we then have
  \[\xi^?\circ\FF\left(E_{\bX,V}^{(m,l)}(f)\right)=E_{\hat\FF(\bX),\FF(V)}^{(m,l)}(\xi^?\circ\FF(f)).\]
  As a consequence, monoidal category equivalences preserve
  generalized indicators:
  \[\nu_{m,l}^{\hat\FF(\bX)}(\FF(V))=\nu_{m,l}^{\bX}(V).\]
  In particular, to deal with the theory of generalized indicators
  it will be sufficient to treat the case where the category $\CC$
  is strict pivotal.
\end{remark}

\begin{remark} \label{r:EXVml}
Assuming the pivotal category $\CC$ is strict, we have the
following diagrams in the graphical calculus:
\begin{equation}\label{eq:EXVml}
E_{\bX,V}^{(m,l)}(f)
 \quad=\quad
 \def\objectstyle{\scriptstyle}
 \xy (0,0)="ctext", "ctext"+(0,-.1)*{f},
{"ctext"+(-2,-2)*{}; "ctext"+(2,2)*{} **\frm{-}},
{\vcross~{"ctext"+(-5,8)="v0"}{"ctext"+(0,8)="v1"}{"ctext"+(-5,2)="w0"}{"ctext"+(0,2)}},
{"v1"; "v1"+(5,0)="v2" **\crv{"v1"+(0,3)& "v2"+(0,3)}}, {"v0";
"v0"+(0,4)*{X} **\dir{-}},
{"ctext"+(0,-2);"ctext"+(0,-6)**\dir{-}},
{"w0";"w0"+(0,-8)**\dir{-}?(.1)+(-3,0)*{V^{-l}}},
{"v2";"v2"+(0,-14)**\dir{-}?(.3)+(2.5,0)*{V^{l}}},
{"ctext"+(0,-8)="ctext1"},{"ctext1"+(0,-.2)*{J_{m,l}(V)}},
{"ctext1"+(-7,-2)*{}; "ctext1"+(7,2)**\frm{-}},
{"ctext1"+(0,-2);"ctext1"+(0,-6)*{V^m}**\dir{-}},
\endxy
 \quad = \,
 \xy (0,0)="ctext", "ctext"+(0,-.1)*{f},
{"ctext"+(-2,-2)*{}; "ctext"+(2,2)*{} **\frm{-}},
{\vcrossneg~{"ctext"+(0,8)="v0"}{"ctext"+(5,8)="v1"}{"ctext"+(0,2)="w0"}{"ctext"+(5,2)="w1"}},
{"v0"; "v0"+(-5,0)="v" **\crv{"v0"+(0,3)& "v"+(0,3)}}, {"v1";
"v1"+(0,4)*{X} **\dir{-}},
{"ctext"+(0,-2);"ctext"+(0,-6)**\dir{-}},
{"w1";"w1"+(0,-8)**\dir{-}?(.1)+(2.5,0)*{V^{l}}},
{"v";"v"+(0,-14)**\dir{-}?(.1)+(-3,0)*{V^{-l}}},
{"ctext"+(0,-8)="ctext1"},{"ctext1"+(0,-.2)*{J_{m,l}(V)}},
{"ctext1"+(-7,-2)*{}; "ctext1"+(7,2)**\frm{-}},
{"ctext1"+(0,-2);"ctext1"+(0,-6)*{V^m}**\dir{-}},
\endxy \quad\text{for } f \in \CC(X, V^m),
\end{equation}
 where
$$
 \def\objectstyle{\scriptstyle}
 \xy (0,0)="ctext", "ctext"+(0,-.1)*{J_{m,l}(V)},
{"ctext"+(-7,-2)*{}; "ctext"+(7,2)*{} **\frm{-}},
{"ctext"+(0,-2);"ctext"+(0,-6)*{V^m}**\dir{-}},
{"ctext"+(0,2);"ctext"+(0,6)*{V^m}**\dir{-}},
{"ctext"+(-6,2);"ctext"+(-6,6)*{V^{-l}}**\dir{-}},
{"ctext"+(6,2);"ctext"+(6,6)*{V^{l}}**\dir{-}},
\endxy \quad\text{is  }
\xy (0,0)="ctext",, {"ctext"+(-4,4)="v1"*{V^{-l}};
"ctext"+(4,4)="v2"*{V^l} **\crv{"v1"+(0,-6)&"v2"+(0,-6)}},
{"ctext"+(8, 4);"ctext"+(8,-4)**\dir{-}?(.5)+(3,0)*{V^{m}}}\endxy
\quad \text{if} \quad ml \ge 0, \quad \text{but equals} \quad \xy
(0,0)="ctext", {"ctext"+(-4,4)="v1"*{V^{-l}};
"ctext"+(4,4)="v2"*{V^l} **\crv{"v1"+(0,-6)&"v2"+(0,-6)}},
{"ctext"+(-9, 4);"ctext"+(-9,-4)**\dir{-}?(.3)+(-3,0)*{V^{m}}}
\endxy \quad \text{if}\quad ml \le 0.
$$
Moreover, since $J_{m,l}(V\du) = J_{-m,-l}(V)$, \eqref{eq:EXVml} immediately implies
\begin{equation}\label{dualnegE}
  E_{\bX,V}^{(m,l)}=E_{\bX,V\du}^{(-m,-l)}
\end{equation}

for  $V \in \CC$, $\bX \in Z(\CC)$, and $m,l \in \BZ$.
\end{remark}

\begin{remark}\label{firstadditivity}
 It is immediate from the definition or \eqref{eq:EXVml} that $E_{\bX,V}^{(m,l)}$ is natural in $\bX\in Z(\CC)$, i.e.\ for morphisms $g\colon\bX\to\bY$ in $Z(\CC)$ and $f\colon Y\to V$ in $\CC$ we have $E_{\bX,V}^{(m,l)}(fg)=E_{\bY,V}^{(m,l)}(f)g$. As a consequence, the generalized indicator $\nu_{m,l}^{\bX}(V)$ is additive in its parameter $\bX$, that is
 $$
 \nu_{m,l}^{\bX\oplus\bY}(V)=\nu_{m,l}^{\bX}(V)+\nu_{m,l}^{\bY}(V)
 $$
  for $\bX,\bY\in Z(\CC)$, $V\in\CC$ and $(m,l)\in\Z\times\Z$.
\end{remark}

\begin{lem}\label{Ecomposition}
  Consider a pivotal monoidal category $\CC$. Then we have
  \begin{equation}\label{eql:2.1}
  E_{\bX,V}^{(m,k+l)}=E_{\bX,V}^{(m,k)}E_{\bX,V}^{(m,l)}
  \end{equation}
  for all $V\in\CC$, $\bX\in Z(\CC)$, and $m,k,l\in\Z$ such that $kl\geq 0$ or $m\neq 0$.
  In particular
  $E_{\bX,V}^{(m,l)}=\left(E_{\bX,V}^{(m,1)}\right)^l$ if
  $m\neq 0$.
\end{lem}
\begin{proof}
We can assume that $\CC$ is strict pivotal.
By \eqref{dualnegE} we may assume that $m\geq 0$.
 If $k, l\ge 0$, then
 $$
 \def\objectstyle{\scriptstyle}
 \xy (0,0)="ctext", "ctext"+(0,-.1)*{J_{m,l}},
{"ctext"+(-7,-2)*{}; "ctext"+(7,2)*{} **\frm{-}},
{"ctext"+(0,-2);"ctext"+(0,-6)**\dir{-}?(.5)+(3,0)*{V^m}},
{"ctext"+(0,2);"ctext"+(0,6)*{V^m}**\dir{-}},
{"ctext"+(-6,2);"ctext"+(-6,6)*{V^{-l}}**\dir{-}},
{"ctext"+(6,2);"ctext"+(6,6)*{V^{l}}**\dir{-}},
{"ctext"+(0,-8)="ctext1"},"ctext1"+(0,-.1)*{J_{m,k}},
{"ctext1"+(-14,-2)*{}; "ctext1"+(14,2)*{} **\frm{-}},
{"ctext1"+(-12,2);"ctext1"+(-12,14)*{V^{-k}}**\dir{-}},
{"ctext1"+(12,2);"ctext1"+(12,14)*{V^{k}}**\dir{-}},
{"ctext1"+(0,-2);"ctext1"+(0,-6)*{V^m}**\dir{-}},
\endxy\,=\,
\xy (0,4)="ctext", {"ctext"+(-5,0)*{V^{-l}}="v2";
"ctext"*{V^l}="v3" **\crv{"v2"+(0,-6)&"v3"+(0,-6)}}, {"ctext"+(5,
0)*{V^m};"ctext"+(5,-6)**\dir{-}},
{"ctext"+(0,-8)="ctext1"},"ctext1"+(0,-.1)*{J_{m,k}},
{"ctext1"+(-14,-2)*{}; "ctext1"+(14,2)*{} **\frm{-}},
{"ctext1"+(-12,2);"ctext1"+(-12,8)*{V^{-k}}**\dir{-}},
{"ctext1"+(12,2);"ctext1"+(12,8)*{V^{k}}**\dir{-}},
{"ctext1"+(0,-2);"ctext1"+(0,-6)*{V^m}**\dir{-}},
\endxy\,=\,
\xy (0,4)="ctext", {"ctext"+(-5,0)*{V^{-l}}="v2";
"ctext"*{V^l}="v3" **\crv{"v2"+(0,-6)&"v3"+(0,-6)}},
{"ctext"+(-10,0)*{V^{-k}}="v1"; "ctext"+(5,0)*{V^k}="v4"
**\crv{"v1"+(0,-12)&"v4"+(0,-12)}}, {"ctext"+(10,
0)*{V^m};"ctext"+(10,-14)*{V^m}**\dir{-}},
\endxy\,=\, J_{m, k+l}\,.
$$
The same conclusion
$$
J_{m,k}(V)\circ(\id_{V^{-k}}\o
  J_{m,l}(V)\o \id_{V^k})=J_{m,k+l}(V)
$$
holds for $k,l \le 0$ by a calculation which is
  the mirror image of that above.
  Thus, whenever $kl\geq 0$, we have
  \[E_{\bX,V}^{(m,k)}E_{\bX,V}^{(m,l)}(f)\,=\,
 \def\objectstyle{\scriptstyle}
 \xy (0,0)="ctext", "ctext"+(0,-.1)*{f},
{"ctext"+(-2,-2)*{}; "ctext"+(2,2)*{} **\frm{-}},
{\vcross~{"ctext"+(-5,7)="v0"}{"ctext"+(0,7)="v1"}{"ctext"+(-5,2)="w0"}{"ctext"+(0,2)}},
{"v1"; "v1"+(5,0)="v2" **\crv{"v1"+(0,3)& "v2"+(0,3)}},
{\vcross~{"v0"+(-5,4)="z"}{"v0"+(0,4)="z0"}{"v0"+(-5,0)="v"}{"v0"}},
{"z0";"z0"+(15,0)="z1"**\crv{"z0"+(0,6)&"z1"+(0,6)}},
{"ctext"+(0,-2);"ctext"+(0,-6)**\dir{-}},
{"w0";"w0"+(0,-8)**\dir{-}}, {"v2";"v2"+(0,-13)**\dir{-}},
{"ctext"+(0,-8)="ctext1"},{"ctext1"+(0,-.2)*{J_{m,l}}},
{"ctext1"+(-7,-2)*{}; "ctext1"+(7,2)**\frm{-}},
{"ctext1"+(0,-2);"ctext1"+(0,-4)**\dir{-}},
{"ctext1"+(0,-6)="ctext2"},{"ctext2"+(0,-.2)*{J_{m,k}}},
{"ctext2"+(-13,-2)*{}; "ctext2"+(13,2)**\frm{-}},
{"ctext2"+(0,-2);"ctext2"+(0,-6)*{V^m}**\dir{-}},
{"ctext2"+(-10,2);"v"**\dir{-}}, {"ctext2"+(10,2);"z1"**\dir{-}},
{"z"; "z"+(0,6)*{X}**\dir{-}},
\endxy \,=\,
\xy (0,0)="ctext", "ctext"+(0,-.1)*{f}, {"ctext"+(-2,-2)*{};
"ctext"+(2,2)*{} **\frm{-}},
{\vcross~{"ctext"+(-5,7)="v0"}{"ctext"+(0,7)="v1"}{"ctext"+(-5,2)="w0"}{"ctext"+(0,2)}},
{"v1"; "v1"+(5,0)="v2" **\crv{"v1"+(0,3)& "v2"+(0,3)}},
{"ctext"+(0,-2);"ctext"+(0,-6)**\dir{-}},
{"w0";"w0"+(0,-8)**\dir{-}}, {"v2";"v2"+(0,-13)**\dir{-}},
{"ctext"+(0,-8)="ctext1"},{"ctext1"+(0,-.2)*{J_{m,k+l}}},
{"ctext1"+(-7,-2)*{}; "ctext1"+(7,2)**\frm{-}},
{"ctext1"+(0,-2);"ctext1"+(0,-6)*{V^m}**\dir{-}}, {"v0";
"v0"+(0,6)*{X}**\dir{-}},
\endxy
    \,=\,E_{\bX,V}^{(m,k+l)}(f).
  \]
Note that \eqref{eql:2.1} for $k,l$ of the same sign implies $E_{\bX, V}^{m,\pm l}=\left(E_{\bX, V}^{m,\pm 1}\right)^l$ for all non-negative integers $l$. Thus, to prove \eqref{eql:2.1} for $m\neq 0$ and arbitrary $k,l$, it suffices to show that $E_{\bX,V}^{m,\pm 1}$ are mutually inverse. It suffices to assume $m>0$. Then
  \begin{equation}
   E_{\bX,V}^{(m,1)}(f)=
  \def\objectstyle{\scriptstyle}
  \xy (0,-1)="ctext", "ctext"*{f},
{"ctext"+(-4,-2)*{}; "ctext"+(4,2)*{} **\frm{-}},
{"ctext"+(2,-2);"ctext"+(2,-10)*{V^{m-1}}**\dir{-}},
{\vcross~{"ctext"+(-8,14)*{X}}{"ctext"+(0,10)="v0"}{"ctext"+(-8,0)="v2"}{"ctext"+(0,2)}},
{"v0"; "v0"+(8,0)="v1" **\crv{"v0"+(0,4)& "v1"+(0,4)}},
{"v1"; "v1"+(0, -20) **\dir{-}?(1)+(2,0)*{V}},
{"v2";"v2"+(0,-2)="v3" **\dir{-}},
{"v3"; "ctext"+(-2,-2)="y" **\crv{"v3"+(0,-4)& "y"+(0,-4)}},
\endxy\,,\qquad
E_{\bX,V}^{(m,-1)}(f)=
\def\objectstyle{\scriptstyle}
 \xy (0,-1)="ctext", "ctext"*{f},
{"ctext"+(-4,-2)*{}; "ctext"+(4,2)*{} **\frm{-}},
{"ctext"+(2,-2)="c1";"c1"+(6,0)="c2" **\crv{"c1"+(0,-4)& "c2"+(0,-4)}},
{\vcross~{"ctext"+(-8,14)*{X}}{"ctext"+(0,11)="v0"}{"ctext"+(-8,0)="v2"}{"ctext"+(0,2)}},
{"v0"; "v0"+(8,0)="v1" **\crv{"v0"+(0,4)& "v1"+(0,4)}},
{"v1"; "c2" **\dir{-}?(.5)+(3,0)*{V\du}},
{"v2"*{};"ctext"+(-8,-10)*{V} **\dir{-}},
{"ctext"+(-2,-2); "ctext"+(-2,-10)*{V^{m-1}} **\dir{-}}
\endxy
\end{equation}
  and they are inverse of each other.
\end{proof}

\begin{remark}
In particular, for $m>0$ and $\bX$ equal to the unit object of
$Z(\CC)$, the GFS indicator  $\nu_{m,l}^\bX (V)$ coincides with
the $(m,l)$-th Frobenius-Schur indicator $\nu_{m,l}(V)$ defined in
\cite{NS05}.
\end{remark}

\begin{lem}\label{Emm}
  Consider a pivotal monoidal category $\CC$. Then we have
  \begin{equation}
  E_{\bX,V}^{(m,m+l)}(f)=E_{\bX,V}^{(m,l)}(f)\theta\inv_{\bX}
  \end{equation}
  for all $V\in\CC$, $\bX\in Z(\CC)$, and $m,l\in\BZ$ with $m\neq 0$
  and $f\in\CC(X,V^m)$, where $\theta$ is the twist on $Z(\CC)$ associated with the pivotal structure on $\CC$.  In addition, if $\theta_\bX^N=\id_\bX$ for some positive integer $N$, then $E_{\bX,V}^{(m,mN)}=\id$.
\end{lem}
\begin{proof}
  We may assume that $\CC$ is strict. By Lemma \ref{Ecomposition} it is enough
  to treat the case $l=0$, i.e. to show
 $E_{\bX,V}^{(m,m)}(f)=f \circ \theta\inv_{\bX}$. But indeed
  $J_{m,m}(V)=\ev_m\o V^m$, and thus
   $$
 E_{\bX,V}^{(m,m)}(f)\,=
{\def\objectstyle{\scriptstyle}
\xy (0,0)="ctext", "ctext"*{f},
{"ctext"+(-3,-2)*{}; "ctext"+(3,2)*{} **\frm{-}},
{\vcross~{"ctext"+(-6,12)*+{X}}{"ctext"+(0,8)="v0"}{"ctext"+(-6,0)="v2"}{"ctext"+(0,2)}},
{"v0"; "v0"+(7,0)="v1" **\crv{"v0"+(0,4)& "v1"+(0,4)}},
{"v1"; "v1"+(0, -18)*{V^m} **\dir{-}},
{"v2"*{};"ctext"+(-6,-2)="v3" **\dir{-}},
{"v3"; "ctext"+(0,-2)="y" **\crv{"v3"+(0,-4)& "y"+(0,-4)}},
\endxy}
=\,
{\def\objectstyle{\scriptstyle}
\xy (0,0)="ctext", "ctext"+(-10,0)="ctext1"*{f\du},
{"ctext1"+(-3,-2)*{}; "ctext1"+(3,2)*{} **\frm{-}},
{\vcross~{"ctext"+(-10,10)="pX"}{"ctext"+(-4,10)="v0"}{"ctext"+(-10,2)="v2"}{"ctext"+(-4,2)="v4"}},
{"v0"; "v0"+(6,0)="v1" **\crv{"v0"+(0,3)& "v1"+(0,3)}},
{"v1"; "v1"+(0, -20)*{V^m} **\dir{-}},
{"ctext"+(-10,-2)="v3"; "ctext"+(-4,-2)="y" **\crv{"v3"+(0,-4)& "y"+(0,-4)}},
{"v4"; "y" **\dir{-}},
{"pX"; "pX"+(0,2) **\dir{-}},
{"pX"+(0,3.5)*{X}}
\endxy}
 = \,
{\def\objectstyle{\scriptstyle}
\xy (0,-1)="ctext",
{\vcross~{"ctext"+(-10,14)*+{X}}{"ctext"+(-4,10)="v0"}{"ctext"+(-10,0)="v2"}{"ctext"+(-4,0)}},
{"v0"; "v0"+(7,0)="v1" **\crv{"v0"+(0,4)& "v1"+(0,4)}},
{"v1"; "v1"+(0, -7) **\dir{-}},
{"ctext"+(-10,0)="v3"; "ctext"+(-4,0)="y" **\crv{"v3"+(0,-5)& "y"+(0,-5)}},
"v1"+(0,-9)="ctext1"*{f},
{"ctext1"+(-3,-2)*{}; "ctext1"+(3,2)*{} **\frm{-}},
{"v1"+(0,-11); "v1"+(0, -19)*{V^m} **\dir{-}},
\endxy}\,
= \,f\circ \theta\inv_{\bX}\,.
$$
If, in addition, $\theta_\bX^N =\id_\bX$, then
$$
E_{\bX, V}^{(m,mN)}(f) = \left(E_{\bX, V}^{(m,m)}\right)^N(f) = f \circ \theta_\bX^{-N} =f
$$
for all $f \in \CC(X, V^m)$.
\end{proof}

\begin{prop}\label{p:tq}
Let $\CC$ be a pivotal category over $\k$, $V \in \CC$, $\bX \in
Z(\CC)$, and $(m,l) \in \BZ\times \BZ$. Then
  \begin{enumerate}
    \item[\rm (i)] $\nu_{m, l}^{\bX}(I)=\dim \CC(X,I) $.
    \item[\rm (ii)] $\nu_{m,l}^\bX(V^{q}) = \nu_{qm,ql}^\bX(V)$ for $q \in \BZ$.
    \item[\rm (iii)] If $\bX$ is (absolutely) simple and $m>0$, then $\nu_{m, m+l}^\bX(V) = \w^{-1} \nu_{m,l}^\bX(V)$ where $\w \in \k$ is given by $\theta_\bX = \w \id_\bX$. Moreover, $\w=1$ or $\CC(X,I)=0$.
  \end{enumerate}
\end{prop}
\begin{proof}
  We can assume that $\CC$ is strict.
  \begin{enumerate}
  \item[\rm(i)] For $V=I$ the morphism $J_{m,l}(V)$ is the identity,
  and so is $E_{\bX,V}^{(m,l)}$.
  \item[\rm(ii)] From the graphical
  representation of $E_{\bX,V}$ displayed in Remark \ref{r:EXVml}, it is straightforward to read off
  that $E_{\bX,V^q}^{(m,l)}=E_{\bX,V}^{(mq,lq)}$. \item[\rm(iii)]
  The first statement is a direct consequence of the preceding
  lemma. The second then follows by setting $V=I$ and using (i),
  since we have assumed $\k$ to have characteristic zero. Without
  that assumption we could still set $V=I$ and find
  $$
  \id_{\CC(X,I)} = E_{\bX,I}^{(1,1)} = \w \inv E_{\bX, I}^{(1,0)}=\w\inv \id_{\CC(X,I)},
  $$
  and so $\w=1$ or $\CC(X,I)=0$.\qedhere
  \end{enumerate}
\end{proof}

\begin{lem}\label{pairingformulas} Let $\CC$ be a pivotal category over $\k$, $V \in \CC$, $\bX \in
Z(\CC)$, and $m,l \in \BZ$. Then, for all $f \in \CC(X, V^m)$ and $g \in \CC(V^m, X)$, we have
$$
  \begin{aligned}
    (g,E_{\bX,V}^{(m,l)}(f))_r&=
  \ptrr\left(\,
    \def\objectstyle{\scriptstyle}
  \xy
  \vcross~{(-2.5,3)="v1"}{(3.5,3)="v2"}{(-2.5, -3)="z1"}{(3.5, -3)="z2"},
    {"v1"+(0,2)="ctext1"},
    {"ctext1"+(.5,0)*{g}},
    {"ctext1"+(-2,-2)*{}; "ctext1"+(2,2)*{} **\frm{-}},
    {"ctext1"+(0,2); "ctext1"+(0,6)*{V^m}**\dir{-}},
    {"v2"; "v2"+(0,7)**\dir{-}},{"v2"+(1,8.5)*{V^{-l}}},
    {"z2"+(0,-2)="ctext2"},
    {"ctext2"+(0,0)*{f}},
    {"ctext2"+(-2,-2)*{}; "ctext2"+(2,2)*{} **\frm{-}},
     {"z1";"z1"+(0,-9)*{V^{-l}}**\dir{-}},
     {"ctext2"+(0,-2); "ctext2"+(0,-7)*{V^m}**\dir{-}},
\endxy\,\right)\quad\text{if} \quad lm \leq
                0\,,\\ \\
    (g,E_{\bX,V}^{(m,l)}(f))_\ell & =
   \ptrl\left(\,
    \def\objectstyle{\scriptstyle}
    \xy (0,0)="ctext",
    \vcrossneg~{"ctext"+(-2.5,3)="v1"}{"ctext"+(3.5,3)="v2"}{"ctext"+(-2.5, -3)="w1"}{"ctext"+(3.5, -3)="w2"},
    "v2"+(0,2)="ctext1",
    {"ctext1"+(0.5,-.1)*{g}},
    {"ctext1"+(-2,-2)*{}; "ctext1"+(2,2)*{} **\frm{-}},
    {"ctext1"+(0,2); "ctext1"+(0,7)*{V^{m}}**\dir{-}},
    {"v1"; "v1"+(0,9)*{V^l}**\dir{-}},
    {"w2"; "w2"+(0,-9)*{V^l}**\dir{-}},
    {"w1"+(0,-2)="ctext2"},
    {"ctext2"+(0.5,-.1)*{f}},
    {"ctext2"+(-2,-2)*{}; "ctext2"+(2,2)*{} **\frm{-}},
    {"ctext2"+(0,-2); "ctext2"+(0,-7)*{V^{m}}**\dir{-}},
\endxy\,\right)\quad\text{if}\quad lm\geq 0\,.
  \end{aligned}
$$
\end{lem}

\begin{proof}
We may assume that $\CC$ is strict pivotal. We treat the case
$ml\geq 0$ first. Note that in this case

$$
\def\objectstyle{\scriptstyle}
\xy {(0,0)="ctext"+(3,0)*{J_{m,l}(V)}},
{"ctext"+(-4,2)*{}; "ctext"+(10,-2)*{} **\frm{-}},
{"ctext"+(-3,2)="a2"; "a2"+(0,5)*{\,\,V^{-l}}**\dir{-}},
{"ctext"+(3,2)="a3"; "a3"+(0,5)*{V^m}**\dir{-}},
{"ctext"+(9,2)="a4"; "a4"+(0,5)*{V^l}**\dir{-}},
{"ctext"+(3,-2)="a5"; "a5"+(-13,0)="a6"**\crv{"a5"+(0,-5)& "a6"+(0,-5)}},
{"a6"; "a6"+(0,9)*{V^{-m}}**\dir{-}},
\endxy =
\def\objectstyle{\scriptstyle}
\xy {(0,0)="ctext"},
{"ctext"+(-3,2)="a2"; "a2"+(0,5)*{\,\,V^{-l}}**\dir{-}},
{"ctext"+(3,2)="a3"; "a3"+(0,5)*{V^l}**\dir{-}},
{"a2"; "a3"**\crv{"a2"+(0,-3)& "a3"+(0,-3)}},
{"ctext"+(9,-2)="a4"; "a4"+(0,9)*{V^m}**\dir{-}},
{"a4"; "a4"+(-18,0)="a6"**\crv{"a4"+(0,-6)& "a6"+(0,-6)}},
{"a6"; "a6"+(0,9)*{V^{-m}}**\dir{-}},
\endxy
  =\ev_{m+l}.
$$

Therefore
\[(g,E_{\bX,V}^{(m,l)}(f))_\ell
  =\ptrl(E_{\bX,V}^{(m,l)}(f)g)=
   {\def\objectstyle{\scriptstyle}
 \xy (0,-2)="ctext", "ctext"+(.5,-.1)*{f},
{"ctext"+(-2,-2)*{}; "ctext"+(2,2)*{} **\frm{-}},
{\vcross~{"ctext"+(-5,8)="v0"}{"ctext"+(0,8)="v1"}{"ctext"+(-5,2)="w0"}{"ctext"+(0,2)}},
{"v1"; "v1"+(5,0)="v2" **\crv{"v1"+(0,3)& "v2"+(0,3)}},
{"ctext"+(0,-2);"ctext"+(0,-4)**\dir{-}},
{"w0";"w0"+(0,-6)**\dir{-}?(.1)+(-3,0)},
{"v2";"v2"+(0,-12)**\dir{-}?(.3)+(2,0)},
{"ctext"+(0,-6)="ctext1"},{"ctext1"+(0,-.2)*{J_{m,l}(V)}},
{"ctext1"+(-7,-2)*{}; "ctext1"+(7,2)**\frm{-}},
{"v0"+(0,2)="ctext2"}, "ctext2"+(.5,-.1)*{g},
{"ctext2"+(-2,-2)*{}; "ctext2"+(2,2)*{} **\frm{-}},
{"ctext2"+(0,2)="d0"; "d0"+(-5,0)="d1"**\crv{"d0"+(0,3.5)& "d1"+(0,3.5)}},
{"ctext1"+(0,-2)="d4"; "d4"+(-10,0)="d5"**\crv{"d4"+(0,-5)& "d5"+(0,-5)}},
{"d1"; "d5"**\dir{-}},
\endxy}
  =\gbeg47
   \glmpb\gdnot{\db_{m+l}}\gcmpb\gcmpb\grmpb\gnl
   \gcl5\gcl5\gcl2\gcl1\gnl
   \gvac3\gbmp g\gnl
   \gvac2\gibr\gnl
   \gvac2\gbmp f\gcl2\gnl
   \gvac2\gcl1\gnl
   \glmpt\gdnot{\ev_{m+l}}\gcmpt\gcmpt\grmpt\gend
   =
   \ptrl\left(\,
    \def\objectstyle{\scriptstyle}
    \xy (0,0)="ctext",
    \vcrossneg~{"ctext"+(-2.5,3)="v1"}{"ctext"+(3.5,3)="v2"}{"ctext"+(-2.5, -3)="w1"}{"ctext"+(3.5, -3)="w2"},
    "v2"+(0,2)="ctext1",
    {"ctext1"+(0.5,-.1)*{g}},
    {"ctext1"+(-2,-2)*{}; "ctext1"+(2,2)*{} **\frm{-}},
    {"ctext1"+(0,2); "ctext1"+(0,7)*{V^{m}}**\dir{-}},
    {"v1"; "v1"+(0,9)*{V^l}**\dir{-}},
    {"w2"; "w2"+(0,-9)*{V^l}**\dir{-}},
    {"w1"+(0,-2)="ctext2"},
    {"ctext2"+(0.5,-.1)*{f}},
    {"ctext2"+(-2,-2)*{}; "ctext2"+(2,2)*{} **\frm{-}},
    {"ctext2"+(0,-2); "ctext2"+(0,-7)*{V^{m}}**\dir{-}},
\endxy\,\right).
   \]
   The proof in the case $lm\leq 0$ is similar but using
   $
   \def\objectstyle{\scriptstyle}
\xy {(0,0)="ctext"+(3,0)*{J_{m,l}(V)}},
{"ctext"+(-4,2)*{}; "ctext"+(10,-2)*{} **\frm{-}},
{"ctext"+(-3,2)="a2"; "a2"+(0,5)*{\,\,V^{-l}}**\dir{-}},
{"ctext"+(3,2)="a3"; "a3"+(0,5)*{V^m}**\dir{-}},
{"ctext"+(9,2)="a4"; "a4"+(0,5)*{V^l}**\dir{-}},
{"ctext"+(3,-2)="a5"; "a5"+(13,0)="a6"**\crv{"a5"+(0,-5)& "a6"+(0,-5)}},
{"a6"; "a6"+(0,9)*{V^{-m}}**\dir{-}},
\endxy
=\ev_{l-m}
  $
  this time.
\end{proof}

\begin{lem}
  Assume that $\CC$ is a strict spherical monoidal category. Then
  for $V\in\CC$, $\bX=(X, \sigma_X)\in Z(\CC)$, $m,l\in\BZ$, $f\in\CC(X,V^m)$
  and $g\in\CC(V^m,X)$ we have:
  \[\ptr(E_{\bX,V}^{(m,l)}(f)g)=\ptr(fE_{\bX\du,V\du}^{(m,l)}(g\du)\du)\,.\]
\end{lem}
\begin{proof}
\begin{align*}
  \ptr(gE_{\bX,V}^{(m,l)}(f))&=\ptr((V^{-l}\o f)\hbr_{X,V^{-l}}(g\o V^{-l}))\\
    &=\ptr\left((V^{-l}\o f)\hbr_{X,V^{-l}}(g\o V^{-l}))\du\right)\\
    &=\ptr\left((V^{l}\o g\du)\hbr_{X\du,V^l}(f\du\o V^l)\right)\\
    &=\ptr\left(((V\du)^{-l}\o g\du)\hbr_{X\du,(V\du)^{-l}}(f\du\o (V\du)^{-l})\right)\\
    &=\ptr(f\du E_{\bX\du,V\du}^{(m,l)}(g\du))\\
    &=\ptr(E_{\bX\du,V\du}^{(m,l)}(g\du)\du f)\, .\qedhere
\end{align*}
\end{proof}
\begin{remark}
It is worthwhile to rewrite the last lemma slightly in the context of a strict spherical category over $\k$. We define
\begin{equation}\label{Ebar}
\ol E_{\bX,V}^{(m,l)}\colon\CC(V^m,X)\to\CC(V^m,X);\ f\mapsto
E_{\bX\du,V\du}^{(m,l)}(f\du)\du.
\end{equation}
Thus, the definition of the $\ol E$ maps is obtained by turning
that of the $E$ maps upside down; we will return to this aspect
later. Then the above Lemma says that
\begin{equation}
\label{adjointness} %
\left(E_{\bX,V}^{(m,l)}(f),g\right)=\left(f,\ol
E_{\bX,V}^{(m,l)}(g)\right)\, .
\end{equation}
Note also that by definition the linear map $\ol
E_{\bX,V}^{(m,l)}$ is conjugate to $E_{\bX\du,V\du}^{(m,l)}$, so
that
\begin{equation}\label{nuEbar}
  \nu_{m,l}^{\bX\du}(V\du)=\Tr\left(\ol
  E_{\bX,V}^{(m,l)}\right)\, .
\end{equation}
\end{remark}


\section{The Case of a Semisimple Pivotal Category}\label{s:sscase}
In this section, we continue to study the GFS indicators for
semisimple pivotal categories over $\k$.

In such a category, Lemma \ref{pairingformulas} allows us to express the GFS indicators of $V$ as the pivotal traces of certain endomorphisms of tensor powers of $V$ in the category.

In the case where the category is spherical, we will obtain additional properties as well as another expression for the indicators in terms of pivotal traces of certain endomorphisms in the center $Z(\CC)$. The latter expression will be used in the following section to compare our indicators to those defined by Sommerh\"auser and Zhu in the Hopf algebra case.

Let $\CC$ be a semisimple pivotal category over $\k$.  Recall that the pairings $(\cdot,\cdot)_\ell,(\cdot,\cdot)_r$
defined in \eqref{eq:pairing} are always non-degenerate in the
semisimple case. Suppose $\{p_\a\}$ is a basis
for $\CC(V,W)$. Then the non-degenerate pairing $(\cdot,
\cdot)_{\e}$ defines a dual basis $\{\ol p^\e_\a\}$ for
$\CC(W,V)$, where $\e=\ell$ or $r$. The two bases $\{\ol
p^\ell_\a\}$, $\{\ol p^r_\a\}$ may not be the same. However, when
$V$ or $W=I$, these two bases are identical because
$d_r(I)=d_\ell(I)=1$. In addition, if $\CC$ is spherical these two
bases are always identical, and we will simply write $\{\ol
p_\a\}$ for this dual basis in this case.
\begin{prop}\label{p:surgery}
Let $\CC$ be a semisimple strict pivotal category over $\k$, and let
$V \in \CC$,  $\bX=(X, \hbr_X) \in Z(\CC)$ and $(m, l) \in \BZ^2$.
Suppose $\{p_\a\}_\a$ is a basis for $\CC(V^{m}, X)$. Then
 \begin{equation} \label{eq:ts2}
  \nu_{m,l}^\bX(V) =\left\{
  \begin{array}{lcl}
  \displaystyle  \sum_\a \ptrr \left(
{\def\objectstyle{\scriptstyle}
  \xy
  \vcross~{(-2.5,3)="v1"}{(2.5,3)="v2"}{(-2.5, -3)="z1"}{(2.5, -3)="z2"},
    {"v1"+(0,2)="ctext1"},
    {"ctext1"+(.5,0)*{p_\a}},
    {"ctext1"+(-2,-2)*{}; "ctext1"+(2,2)*{} **\frm{-}},
    {"ctext1"+(0,2); "ctext1"+(0,6)*{V^m}**\dir{-}},
    {"v2"; "v2"+(0,7)**\dir{-}},{"v2"+(1,8.5)*{V^{-l}}},
    {"z2"+(0,-2)="ctext2"},
    {"ctext2"+(.5,0)*{ \ol p^r_\a}},
    {"ctext2"+(-2,-2)*{}; "ctext2"+(2,2)*{} **\frm{-}},
     {"z1";"z1"+(0,-9)*{V^{-l}}**\dir{-}},
     {"ctext2"+(0,-2); "ctext2"+(0,-7)**\dir{-}?(.8)+(3,-1)*{V^m}},
\endxy}
\right) & \text{if} & ml \le 0, \\ \\
\displaystyle%
\sum_\a\ptrl\left(
{\def\objectstyle{\scriptstyle}
  \xy (0,0)="ctext",
  \vcrossneg~{"ctext"+(-2.5,3)="v1"}{"ctext"+(2.5,3)="v2"}{"ctext"+(-2.5, -3)="w1"}{"ctext"+(2.5, -3)="w2"},
    "v2"+(0,2)="ctext1",
    {"ctext1"+(0.5,-.1)*{p_\a}},
    {"ctext1"+(-2,-2)*{}; "ctext1"+(2,2)*{} **\frm{-}},
    {"ctext1"+(0,2); "ctext1"+(0,7)*{V^{m}}**\dir{-}},
    {"v1"; "v1"+(0,9)*{V^l}**\dir{-}},
    {"w2"; "w2"+(0,-9)*{V^l}**\dir{-}},
    {"w1"+(0,-2)="ctext2"},
    {"ctext2"+(0.5,-.1)*{ \ol p^\ell_\a}},
    {"ctext2"+(-2,-2)*{}; "ctext2"+(2,2)*{} **\frm{-}},
    {"ctext2"+(0,-2); "ctext2"+(0,-7)*{V^{m}}**\dir{-}},
\endxy} \right)&\text{if} & ml \ge 0\,.
\end{array}
\right.
 \end{equation}
\end{prop}
\begin{proof}This is a direct consequence of the definition
of the indicators and Lemma \ref{pairingformulas}.
\end{proof}
\begin{remark}\label{r:alternative}
  One may also see that
  $$
  \nu_{m,l}^\bX(V)=
\sum_\a \ptrl \left( {\def\objectstyle{\scriptstyle}
  \xy (0,0)="ctext",
  \vcrossneg~{"ctext"+(-2.5,3)="v1"}{"ctext"+(2.5,3)="v2"}{"ctext"+(-2.5, -3)="w1"}{"ctext"+(2.5, -3)="w2"},
    "v2"+(0,2)="ctext1",
    {"ctext1"+(0.5,-.1)*{p_\a}},
    {"ctext1"+(-2,-2)*{}; "ctext1"+(2,2)*{} **\frm{-}},
    {"ctext1"+(0,2); "ctext1"+(0,7)*{V^{m}}**\dir{-}},
    {"v1"; "v1"+(0,9)*{V^l}**\dir{-}},
    {"w2"; "w2"+(0,-9)*{V^l}**\dir{-}},
    {"w1"+(0,-2)="ctext2"},
    {"ctext2"+(0.5,-.1)*{ \ol p^\ell_\a}},
    {"ctext2"+(-2,-2)*{}; "ctext2"+(2,2)*{} **\frm{-}},
    {"ctext2"+(0,-2); "ctext2"+(0,-7)*{V^{m}}**\dir{-}},
\endxy}
\right) =
 \sum_\a \ptrl \left(
{\def\objectstyle{\scriptstyle}
  \xy (0,0)="ctext",
  \vcrossneg~{"ctext"+(-2.5,3)="v1"}{"ctext"+(2.5,3)="v2"}{"ctext"+(-2.5, -3)="w1"}{"ctext"+(2.5, -3)="w2"},
    "v2"+(0,2)="ctext1",
    {"ctext1"+(0.5,-.1)*{p_\a}},
    {"ctext1"+(-2,-2)*{}; "ctext1"+(2,2)*{} **\frm{-}},
    {"ctext1"+(0,2); "ctext1"+(0,7)*{V^{m}}**\dir{-}},
    {"v1"; "v1"+(-5,0)="v0"**\crv{"v1"+(0,3)&"v0"+(0,3)}},
    {"w1"+(-5,0)="w0"; "w0"+(-5,0)="w"**\crv{"w0"+(0,-3)&"w"+(0,-3)}},
    {"w"; "w"+(0,15)*{V^l}**\dir{-}},
    {"w2"; "w2"+(0,-9)*{V^l}**\dir{-}},
    {"w1"+(0,-2)="ctext2"},
    {"ctext2"+(0.5,-.1)*{ \ol p^\ell_\a}},
    {"ctext2"+(-2,-2)*{}; "ctext2"+(2,2)*{} **\frm{-}},
    {"ctext2"+(0,-2); "ctext2"+(0,-7)*{V^{m}}**\dir{-}},
    {"w0"; "v0"**\dir{-}},
\endxy}
\right)=
  \sum_\a \ptrl \left(
{\def\objectstyle{\scriptstyle}
  \xy (0,0)="ctext",
  \vcross~{"ctext"+(-2.5,3)="v1"}{"ctext"+(2.5,3)="v2"}{"ctext"+(-2.5, -3)="w1"}{"ctext"+(2.5, -3)="w2"},
    "v1"+(0,2)="ctext1",
    {"ctext1"+(0.5,-.1)*{p_\a}},
    {"ctext1"+(-2,-2)*{}; "ctext1"+(2,2)*{} **\frm{-}},
    {"ctext1"+(0,2); "ctext1"+(0,7)*{V^{m}}**\dir{-}},
    {"v2"; "v2"+(6,0)="v3"**\crv{"v2"+(0,3)&"v3"+(0,3)}},
    {"w1"; "w1"+(-6,0)="w0"**\crv{"w1"+(0,-3)&"w0"+(0,-3)}},
    {"w0"; "w0"+(0,15)*{V^l}**\dir{-}},
    {"v3"; "v3"+(0,-15)*{V^l}**\dir{-}},
    {"w2"+(0,-2)="ctext2"},
    {"ctext2"+(0.5,-.1)*{ \ol p^\ell_\a}},
    {"ctext2"+(-2,-2)*{}; "ctext2"+(2,2)*{} **\frm{-}},
    {"ctext2"+(0,-2); "ctext2"+(0,-7)*{V^{m}}**\dir{-}},
\endxy}
\right)
  $$
   for $ml \le 0$.
\end{remark}

\begin{prop}\label{p:dual}
  Let $\CC$ be a semisimple  spherical category over $\k$. For $V
  \in \CC$ and $\bX \in Z(\CC)$, we have
  $$
  \nu_{-m,-l}^\bX(V)=\nu_{m,l}^{\bX}(V\du) = \nu_{m,l}^{\bX\du}(V)
  $$
  for all $(m,l)\in \BZ\times \BZ$.
\end{prop}
\begin{proof}
  The first equality follows immediately from Proposition
  \ref{p:tq} (ii) by setting $q=-1$. In the semisimple case
  \eqref{adjointness} says that $\ol E_{\bX,V}^{(m,l)}$ and
  $E_{\bX,V}^{(m,l)}$ are adjoint maps and have the same trace, so
  that, by \eqref{nuEbar},
  \[\nu_{m,l}^\bX(V\du)=\Tr\left(\ol E_{\bX\du,V}^{(m,l)}\right)=
  \Tr\left(E_{\bX\du,V}^{(m,l)}\right)=\nu_{m,l}^{\bX\du}(V).\qedhere\]
  \end{proof}

Next, we will derive an expression for the GFS indicators as the pivotal traces of certain endomorphisms in $Z(\CC)$. For this we will assume that the category $\CC$ is spherical. We will use the two-sided adjoint $K\colon \CC\to Z(\CC)$ to the forgetful functor with the conventions at the end of section \ref{s:1.3}. Associated with the adjunction $\Psi$, we define
\begin{multline}\label{eq:varphi}
\varphi_{\bX, V}^{(m,l)} : = \bigg(\DD(\bX, K(V^m))
\xrightarrow{\Psi_{\bX, V^m}} \CC(X, V^m)
\xrightarrow{E_{\bX,V}^{(m,l)}}
 \CC(X,  V^m)  \xrightarrow{\Psi\inv_{\bX, V^m}} \DD(\bX,  K(V^m)) \bigg)
\end{multline}
for  $m,l \in \BZ$, where $\DD$ {simply} denotes the center $Z(\CC)$. Obviously, $\varphi_{\bX, V}^{(m,l)}$ is natural in
$\bX$. By Yoneda's lemma,
\begin{equation} \label{eq:kappa}
\varphi_{\bX, V}^{(m,l)} (f) = \kappa_{V}^{(m,l)} \circ f
\end{equation}
for $f \in \DD(\bX, K(V^m))$, where
$$
\kappa_{V}^{(m,l)}:=\varphi_{K(V^m), V}^{(m,l)}(\id) : K(V^m) \to
K(V^m)\,.
$$

Note that for $f\in\DD(\bX,K (V^m))$ and $g\in\DD(K (V^m),\bX)$ we
have, abbreviating $\kappa=\kappa_{V}^{(m,l)}$,

\begin{align*}
  (g\kappa, f)&=\ptr(g\kappa f)=(g,\kappa f)
    =\left(g,\Psi\inv E\Psi(f)\right)
    =\left(\ol\Psi\;\ol E\;\ol\Psi\inv(g),f\right),
\end{align*}
where $\ol\Psi_{W, \bX}: \CC(W, X) \to \DD(K(W), \bX)$ and $\ol E_{\bX, V}^{(m,l)}: \CC(V^m, X) \to \CC(V^m, X)$ are respectively the adjoint maps of $\Psi_{\bX, W}$ and $E_{\bX, V}^{(m,l)}$ with respect to the bilinear form $(\cdot, \cdot)$ described in \eqref{arrangement} and \eqref{adjointness}.
Thus, if we define
$\ol \varphi_{V,\bX}^{(m,l)}$ to be
the composition:
\begin{equation}\label{eq:barvarphi}
\ol \varphi_{V, \bX}^{(m,l)}:=\left( \DD(K(V^m), \bX)
\xrightarrow{\ol\Psi\inv} \CC(V^m, X)  \xrightarrow{\ol
E_{\bX,V}^{(m,l)}} \CC(V^m,  X) \xrightarrow{\ol\Psi}
\DD(K(V^m),\bX)\right)\,,
\end{equation}
then by the non-degeneracy of the pairing $(\cdot,\cdot)$ we have
shown
\begin{equation}\label{eq:barb} \ol\varphi_{V, \bX}^{(m,l)}(g) = g
\circ \kappa_{V}^{(m,l)}
\end{equation}
for all $g \in \DD(K(V^m), \bX)$. In particular,
$\kappa_{V}^{(m,l)}=\ol\varphi_{V,K(V^m)}^{(m,l)}(\id_{K(V^m)})$.

The morphisms $\kappa$ defined above can be used to
compute the GFS indicators with the following theorem.
\begin{thm} \label{t:kappaFS}
  Let $\CC$ be a  spherical fusion category over $\k$, and  $\bX$ a simple object of $\DD:=Z(\CC)$.
  For $m,l \in\BZ$, we have
  $$
  \nu_{m,l}^\bX(V) = \dfrac{1}{d(\bX)}\ptr\left(\kappa_{ V}^{(m,l)} \circ z_\bX\right)\,,
  $$
  where $z_\bX$  is the natural projection of $K(V^m)$ onto the
  the isotypic component of $\bX$.
\end{thm}
\begin{proof}
It follows from  \eqref{eq:varphi}  that
\begin{equation}
  \nu_{m,l}^\bX(V) = \Tr\left(E_{\bX, V}^{(m,l)}\right)= \Tr\left(\varphi_{\bX,
  V}^{(m,l)}\right)
\end{equation}
for all $l,m \in \BZ$. Let $\{f_\a\}_\a$ be a basis for
$\DD(\bX, K(V^{m}))$ and $\{{\ol f}_\a\}_\a$ the dual basis for $\DD(K(V^{m}), \bX)$ with respect to
the pairing $(\cdot, \cdot)$. Then
$$
{\ol f}_\a \circ  f_{\a'} = \dfrac{\delta_{\a,\a'}}{ d(\bX)}\id_\bX \quad\text{and}\quad
z_\bX: = d(\bX) \sum_\a f_\a \circ {\ol f}_\a
$$
is the idempotent corresponding to the isotypic component of $\bX$
in $K(V^m)$. Let us write $\kappa$ for $\kappa_{V}^{(m,l)}$. Then,
by \eqref{eq:kappa},
\begin{multline*}
\nu_{m,l}^\bX(V) = \Tr(\varphi_{\bX, V}^{(m,l)}) =  \sum_\a ({\ol
f}_\a, \varphi_{\bX, V}^{(m,l)}(f_\a))
= \sum_\a ({\ol f}_\a, \kappa f_\a) \\
= \sum_\a \ptr ({\ol f}_\a \kappa   f_\a) =  \ptr (\kappa  \sum_\a
f_\a {\ol f}_\a)=\frac{1}{d(\bX)}\ptr(\kappa \circ
z_\bX)\,.\qedhere
\end{multline*}
\end{proof}
Since $\varphi_{\bX,V}^{(m,l)}$ is conjugate to
$E_{\bX,V}^{(m,l)}$, Lemma \ref{Ecomposition} implies analogous
rules for the $\varphi$ maps, as well as for the $\kappa$
morphisms. More precisely, we have
\begin{equation}
  \kappa_V^{(m,k+l)} = \kappa_V^{(m,k)} \kappa_V^{(m,l)}
\end{equation}
for $kl \ge 0$ or $m \ne 0$. We write $\beta_{V^m}:=\kappa_V^{(m,1)}$ for $m\neq 0$ and
$\gamma_V=\kappa_V^{(0,1)}$. Then we have
$$
\b_{V^m}^l=\kappa_V^{(m,l)}, \quad \gamma_V^l = \kappa_V^{(0,l)} \quad \text{for all }l \ge 0\,.
$$
In view of Proposition \ref{p:dual} and Theorem \ref{t:kappaFS}, the GFS indicator for spherical fusion categories can summarized in terms of $\b$ and $\gamma$:
\begin{equation}\label{eq:betagammaind}
  \nu_{m,l}^\bX(V)=\left\{
  \begin{array}{ll}
    \frac{1}{d(\bX)} \ptr  \left(\b_{V^m}^l \circ z_\bX\right) & \text{for } m > 0,\\ \\
    \frac{1}{d(\bX)} \ptr  \left(\gamma_V^l \circ z_\bX\right) & \text{for } m = 0 \text{ and }l \ge 0,\\ \\
    \nu_{-m,-l}^{\bX\du}(V) & \text{for otherwise.}
  \end{array}
  \right.
\end{equation}
Thus, the values of the GFS indicators are completely determined by those values $\nu_{m,l}^\bX(V)$ with $m \ge 0$. This characterization will be useful in the following section.

\section{Equivariant Frobenius-Schur indicators for Semisimple Hopf Algebras}\label{s:hopf}
We will use the results in the preceding section to compare our generalized indicators with the equivariant indicators defined by Sommerh\"auser and Zhu \cite{SZh}.

Let $\CC=\C{H}$ for a semisimple Hopf algebra $H$ over $\k$. We follow the conventions for the Drinfeld
double $D(H)$ of $H$ described in \cite{Kassel} and \cite{Mont93bk}. As a coalgebra,
$D(H)=(H^*)^\cop\o H$. We abbreviate the element $p \o k$ in $D(H)$ as $pk$ and simply write
$p$ for $p1_H$ and $k$ for $1_{H^*}k$.
 Recall that the multiplication in $D(H)$ is given by
$$
pk \cdot qh = \sum_{(k)} p q(S(k_3) ? k_1) \o k_2 h
$$
where $\displaystyle \sum_{(k)} k_1 \o k_2 \o k_3$ is the
Sweedler notation for $(\Delta \o \id)\Delta(k)$, and $q(S(k_3)
? k_1)$ denotes the linear functional $a \mapsto q(S(k_3) a k_1)$
on $H$.

The center $Z(\CC)$ of $\C{H}$ is equal to $\C{D(H)}$ as a rigid monoidal category.
For $\bX \in \C{D(H)}$, the half-braiding
$\hbr_X (V): X \o V \to V \o X$ for $V \in \CC$ is given by
$$
\hbr_X (V)(x \o v) := \sum_i h_i v \o S^*(h^i)x
$$
where $S$ denotes the antipode of $H$, $\{h_i\}$ is a basis for $H$ and $\{h^i\}$ its
dual basis for $H^*$. Note that $S^2=\id_H$ (cf. \cite{LaRa87}, \cite{LaRa88}).
The Drinfeld isomorphism $u_\bX \in \End_{D(H)}(\bX)$ is
given by
$$
u_\bX(x) = \sum_i h_i h^i x = \sum_i h^i h_i x
$$
for all $x \in \bX$.

The induction functor $K(-)=D(H) \o_H -$ is left adjoint to
the forgetful functor from $\C{D(H)} \to \C{H}$ with the adjunction isomorphisms
$\ol \Psi_{V, \bX} :   \Hom_H(V,X) \to \Hom_{D(H)}(K(V), \bX)$ and
$\ol \Psi\inv_{V, \bX} :   \Hom_{D(H)}(K(V), \bX) \to \Hom_H(V,X)$
given by
$$
\ol \Psi (f)(p \o v) = p   f(v)\quad \text{and}\quad
\ol \Psi\inv (g)(v) = g(1_{D(H)} \o v)
$$
for $v \in V$, $p \in H^*$. Note that  $D(H) \o_H V$ naturally
isomorphic to $H^* \o V$ as $\k$-linear spaces.   Every element is
a linear combination of the tensor products $p \o v \in D(H) \o_H
V$ with $p \in H^*$ and $v \in V$.

As we have mentioned following \eqref{Ebar}, the definition of
$\ol E_{\bX,V}^{(m,l)}$ can be obtained in the graphical calculus
by turning the definition of $E_{\bX,V}^{(m,l)}$ upside down.
Explicitly, this gives, for $m>0$:

$$
\ol E_{\bX,V}^{(m,1)}(g) =
\def\objectstyle{\scriptstyle}
\xy (0,1)="ctext", "ctext"*{g},
{"ctext"+(-3, -2)*{}; "ctext"+(3,2)*{} **\frm{-}},
{"ctext"+(-1,2)*{};"ctext"+(-1,10)*{V^{m-1}}="v" **\dir{-}},
{"ctext"+(1.5,2);"ctext"+(1.5,3.5)="w" **\dir{-}},
{"w"; (5,0)+"w"="w1" **\crv{"w"+(0,3)& "w1"+(0,3)}},
{"w1"; (0,-5.5)+"w1"="w2" **\dir{-}},
{\vcross~{"ctext"+(0,-2)="x"}{"w2"}{"x"+(0,-6)="x1"}{"w2"+(0,-6)="w4"}},
{"ctext"+(-7,-8)="w3";"ctext"+(-7,10)*{V} **\dir{-}},
{"w3"; "x1" **\crv{"w3"+(0,-4)& "x1"+(0,-4)}},
{"w4"; "w4"+(0,-5)*{X}**\dir{-}},
\endxy
    \quad\text{and}\quad
 \ol E_{\bX,V}^{(0,1)}(g) =
  {\def\objectstyle{\scriptstyle}
  \xy
    (0,-7)="ctext",
    \vcross~{"ctext"+(-5,7)="x1"}{"ctext"+(0,7)="x2"}{"ctext"+(-5,2)="y1"}{"ctext"+(0,2)="y2"},
    {"x1"+(0,2)*{g}},
    {"x1"+(-2,0)*{}; "x1"+(2,4)*{} **\frm{-}},
    {"x2"+(-10,4)="z0";"x2"+(0,4)="z2"**\crv{"z0"+(0,4) & "z2"+(0,4)}},
     {"z2";"x2"**\dir{-}},
     {"z0";"y1"+(-5,0)="y0"**\dir{-}?(.1)+(-1,0)*{V}},
    {"y2"; "y2"+(0,-6)*{X}**\dir{-}},
    {"y0";"y1"**\crv{"y0"+(0,-3)& "y1"+(0,-3)}},
\endxy}\,.
$$

Thus, for $V\in\CC$, $\bX\in Z(\CC)$, $m>0$, $f \in \Hom_H(V^m,
X)$ and $g \in \Hom_H(\k, \bX)$, it is straightforward to verify
that
$$
\ol E_{\bX,V}^{(m,1)}(f)(v\o w)= \sum_{i,j} (h_i v^j)(v)  S^*(h^i)
f(w \o v_j) =
\sum_{i}  h^i f(w \o h_i v)
$$
for  $v\in V$ and $w\in V^{m-1}$,  and
$$
\ol E_{\bX,V}^{(0,1)}(g)(1_\k)= \sum_{i,j} (h_iv^j)(v_j)
S^*(h^i)g(1_\k)=\sum_{i,j} v^j(h_i v_j) h^i g(1_\k)= \chi_V
g(1_\k),
$$
where $\chi_V$ denotes the character afforded by $V$,
$\{v_j\}$  a basis for $V$ and $\{v^j\}$  its dual basis.
Therefore, for $p \in H^*$, $v \in V$ and $w \in V^{m-1}$,
\begin{equation}\label{eq:bg}
\b_{V^m} (p \o (v \o w))  = \sum_i ph^i \o (w \o h_i v) \quad \text{and}\quad
\gamma_V(p \o 1_\k) =p \chi_V \o 1_\k\,.
\end{equation}
The above formula for
$
\b_{V^m}$ is identical to the map $\b_{V, V^{m-1}}$ defined in
\cite{SZh}. Let $\rho_m: D(H) \to \End(D(H) \o_H V^m)$ be the
corresponding representation of the $D(H)$-module $D(H) \o_H V^m$
and $z$ an element in the center of $D(H)$. The $(m,l)$-th
\textbf{equivariant Frobenius-Schur indicator} of $V$ and
$z=\sum_i p_i k_i$ is defined in \cite{SZh} as %


\newcommand{\ool}{}

\begin{equation}\label{def:SZ}
\SZ_V((m,l), z):= \left\{
\begin{array}{ll}
\Tr(\ool\b_{V^{m}}^l \circ \rho_m (z)) & \text{if } m >0, \\ \\
\displaystyle \dim H \sum_i \e(k_i) (p_i \chi_{V^l})(\Lambda) & \text{if } m =0 \text{ and } l\ge 0, \\\\
\SZ_V((-m,-l), S_D(z)) & \text{ otherwise},
\end{array}
\right.
\end{equation}
where $S_D$ is the antipode of $D(H)$, $\e$ the counit of $H$,
$\chi_{V^l}$ the character of $H$ afforded by $V^l$, and $\Lambda$
the \emph{normalized integral} of $H$, i.e. the integral of $H$
satisfying $\e(\Lambda)=1$.

The following corollary highlights the relationship between our GFS indicators and
the equivariant FS indicators defined for semisimple Hopf algebras.
\begin{cor}\label{c:Hopfcase}
Let $\CC = \C{H}$ for some finite-dimensional semisimple Hopf algebra $H$ over $\k$.
For simple $\bX \in \C{D(H)}$, $V \in \C{H}$ and $(m,l) \in \BZ \times \BZ$,
$$
\nu_{m,l}^\bX(V)=\frac{1}{\dim\bX} \SZ_V((m,l), e_\bX)
$$
where $e_\bX$ is the central idempotent of $D(H)$ associated with $\bX$.
\end{cor}
\begin{proof}
We first consider $m >0$. Since $e_\bX$ is the central idempotent
of $D(H)$ associated with the simple $D(H)$-module $\bX$,
$\rho_m(e_\bX)$ is the central idempotent $z_\bX$ of
$\End_{D(H)}(D(H) \o_H V^m)$ corresponding to the isotypic
component of $\bX$ in $D(H) \o_H V^m$. The pivotal traces in
$\C{H}$ as well as $\C{D(H)}$ are identical to the ordinary trace of linear operators.
Therefore, by Theorem \ref{t:kappaFS} or \eqref{eq:betagammaind},
$$
\nu_{m,l}^\bX(V) = \frac{1}{\dim \bX} \Tr( \ool \b_{V^m}^l \circ z_\bX) =
\frac{1}{\dim \bX} \Tr(\ool\b_{V^m}^l \circ \rho_m (e_\bX)) =
\frac{1}{\dim \bX} \SZ_V((m,l), e_\bX)\,.
$$

Let $e_\bX = \sum_i p_i k_i$ for some $p_i \in H^*$ and $k_i \in H$, $\{h_j\}$ a basis for $H$ and
$\{h^j\}$ its dual basis for $H^*$. Then $\{h^j \o 1_\k\}$ is a basis for $D(H) \o_H \k$ and
$$
\ool\gamma_{V^l}\circ \rho_0(e_\bX)(h^j \o 1_\k)= \ool\gamma_{V^l}(e_\bX h^j \o 1_\k) =
\ool\gamma_{V^l}( h^j e_\bX \o 1_\k) = \sum_i \e(k_i) h^j p_i \chi_{V^l} \o 1_\k\,.
$$
Let $\Lambda$ be the normalized integral, and $\chi_H$ the regular character of $H$. Then
$\chi_H$ is a two-sided integral of $H^*$ and $\chi_H(\Lambda)=1$ (cf. \cite[Theorem 4.4]{LaRa88}). By
\cite[Proposition 2]{Radf94},
$$
\Tr(\ool\gamma_{V^l}\circ \rho_0(e_\bX)) =
\sum_{i,j} \e(k_i)(h^j p_i \chi_{V^l})(h_j) = \chi_H(1) \sum_i \e(k_i) p_i \chi_{V^l}(\Lambda)\,.
$$
It follows from \eqref{eq:betagammaind} that
$$
\nu_{0,l}^\bX(V) = \frac{1}{\dim \bX} \Tr(\ool\gamma_{V^l}\circ \rho_0(e_\bX))
=\frac{\dim H}{\dim \bX}  \sum_i \e(k_i) p_i \chi_{V^l}(\Lambda)
= \frac{1}{\dim \bX} \SZ_V((0,l), e_\bX)
$$
for $l \ge 0$.

Thus, if (i) $m<0$, or (ii) $m=0$ and $l<0$, then, by Proposition \ref{p:dual}, we find
$$
\nu_{m,l}^\bX(V) =\nu_{-m,-l}^{\bX\du}(V) = \frac{1}{\dim \bX} \SZ_V((-m,-l), e_{\bX\du})
=\frac{1}{\dim \bX} \SZ_V((m,l), e_\bX)\,.
$$
The last equality follows from the fact $S_D(e_\bX) = e_{\bX\du}$ and the definition of
equivariant FS-indicators illustrated in \eqref{def:SZ}.
\end{proof}

\section{$\SL$-equivariant Indicators for spherical fusion categories}\label{s:equiv}
 Given a pair $(m,l)$ of integers, and an object $V$ in a pivotal
 fusion category, the values of the GFS indicators
 $\nu_{m,l}^\bX(V)$ for  $\bX \in Z(\CC)$ can be
 extended to a functional on the Grothendieck algebra
 $\KK_\k(Z(\CC))=\KK_0(Z(\CC))\o_\BZ \k$. In this section, these
 functionals are introduced as the equivariant indicators, and
 studied in detail for a spherical fusion category $\CC$. In this
 case, the center $Z(\CC)$ is a modular tensor category, and so
 $\KK_\k(Z(\CC))$ admits a natural representation of $\SL$
 as described in Section \ref{s:modular}. We show that
 the set of all equivariant indicators for a spherical fusion category $\CC$
 is closed under the contragredient action of $\SL$ on
 $\KK_\k(Z(\CC))^*$, and this action on the equivariant indicators is
 compatible with the action of $\SL$ on $\BZ^2$. This property
 of equivariant indicators implies the additivity of the GFS
 indicator $\nu_{m,l}^\bX(V)$ in $V$, and that its values lie in the
 cyclotomic field $\BQ_N$ where $N$ is the Frobenius-Schur
 exponent of $\CC$. Moreover, a formula for the GFS
 indicators for spherical fusion categories is obtained in Corollary \ref{c:nu}
 as a consequence. This formula implies the FS indicator formula discovered in
 \cite[Theorem 4.1]{NS07}.

 Throughout the section, we consider a spherical fusion category $\CC$, and we let $\{[\bX_j]|j\in\hat\Gamma\}$ with $\bX_j=(X_j,\hbr_{X_j})$ be the set of isomorphism classes of the simple objects in $\DD:=Z(\CC)$.

The equivariant indicators for $\CC$ are defined as follows.
\begin{defn}\label{def:EGFS}
   For $m,l \in \BZ$, the \textbf{$(m,l)$-th equivariant indicator} of $V \in \CC$ is defined as the functional $I_V((m,l), -)\in (\KK_\k(\DD))^*$ determined by the assignment
  $$
  I_V((m,l), [\bX]):= \nu_{m,l}^{\bX} (V)
  $$
  for $\bX \in Z(\CC)$; this is well-defined in view of remark \ref{firstadditivity}.
\end{defn}
\begin{remark}{\rm (i)
 Definition \ref{def:EGFS} obviously makes sense for pivotal fusion categories. However, there is no natural modular action on the Grothendieck algebras of the centers of these categories, so we will reserve the term for the spherical case.{\smallskip}\\
 (ii) For $\CC = \C{H}$ for some semisimple Hopf algebra $H$ over $\k$, it follows from Corollary \ref{c:Hopfcase} that
$$
I_V\left((m,l), z\right) = \SZ_V\left((m,l), \psi(z)\right)
$$
where $ \psi: \KK_\k(Z(\CC)) \to \Cent(D(H))$ is the $\k$-linear isomorphism given by
$\psi([\bX]) = \frac{1}{\dim \bX} e_\bX$ for every simple $D(H)$-module $\bX$. Therefore,
the equivariant indicator defined in Definition \ref{def:EGFS} is a generalization of the
corresponding notion introduced by Sommerh\"auser and Zhu in \cite{SZh} to spherical fusion categories.
}
\end{remark}

Recall from Section \ref{s:modular} the canonical modular representation $\rho_{Z(\CC)}: \SL \to \GL(\KK_\k(Z(\CC)))$ of $Z(\CC)$. The association action of $\SL$ on $\KK_\k(Z(\CC))$ is given by
\begin{equation}\label{eq:sl2action}
\fs[\bX_j] = \dfrac{1}{\dim \CC} \sum_{i \in \hat\Gamma} S_{ij} [\bX_i]\quad\text{and}\quad \ft[\bX_j] = \w_j [\bX_j],
\end{equation}
where $[S_{ij}]_{\hat\Gamma}$ and $[\delta_{ij}\w_j]_{\hat\Gamma}$ are the $S$ and $T$-matrices of $Z(\CC)$. For the convenience of the remaining discussion, we summarize some properties of  the equivariant indicators in the following lemma.

\begin{lem}\label{l:elementary2} Let $\CC$ be a spherical fusion category over $\k$.
  For $z \in \KK_\k(Z(\CC))$, $V \in \CC$,  and $q, m,l  \in \BZ$, we have
  \begin{enumerate}
    \item[\rm (i)] $I_V((-m,-l), z)= I_{V\du}((m,l), z) = I_V((m,l), \fs^2 z)$,\\
    \item[\rm (ii)] $I_V((qm,ql), z)= I_{V^{q}}((m,l), z)$.
  \end{enumerate}
\end{lem}
\begin{proof} By \eqref{eq:strelations}, $\fs^2 [\bX_j] = [\bX_{\ol j}] = [\bX_j\du]$ for $j \in \hG$. Therefore, the statements follow immediately from Propositions \ref{p:tq}, \ref{p:dual} and Definition \ref{def:EGFS}.
\end{proof}

We define an \emph{outer} automorphism $\sigma$ of $\SL$ by
\begin{equation}\label{eq:tildefg}
\sigma(\fg) = \fj \fg \fj\inv  \quad \text{where} \quad \fj = \mtx{1 & 0\\ 0 & -1}
\end{equation}
for $\fg \in \SL$.
We will write $\tilde \fg$ for $\sigma(\fg)$ in the sequel. In
particular, $\tilde \fs =\fs\inv$ and $\tilde \ft = \ft\inv$, and
so $\tilde {\tilde \fg} =\fg$ for all $\fg \in \SL$. If $\rho: \SL \to \GL(V)$ is a representation, then we denote by $\tilde\rho$ its twist by the automorphism $\sigma$, i.e.
\begin{equation}
  \tilde \rho(\fg):= \rho(\tilde \fg)\quad \text{for all } \fg  \in \SL\,.
\end{equation}

We proceed to show the $\SL$-equivariance of the equivariant indicators as stated in the following theorem.
\begin{thm}\label{t:equivariance}
Let $\CC$ be a spherical fusion category over $\k$, $V \in \CC$ and $m,l \in \BZ$. Then
$$
I_V((m,l)\fg, z) = I_V((m,l), \tilde\fg z)
$$
for all $\fg \in \SL$ and $z \in \KK_\k(Z(\CC))$.
\end{thm}
  \begin{proof}
Since $\SL$ is generated by $\fs$ and $\ft$, it suffices to prove the
equality holds for $\fg=\fs, \ft$. Let $\{[V_i]\mid i \in
\Gamma\}$  denote the set
of isomorphism classes of simple objects in $\CC$.
 For any $n \in \BZ$, $i \in \Gamma$ and $k \in \hG$, we
  let $\{p^n_{i, \a}\}_\a$ be a basis for $\CC(V^n, V_i)$, and $\{\iota_{i,k, \b}\}_\b$ a
  basis for $\CC(V_i, X_k)$. Then $\{\iota_{i,k, \b} \circ p^n_{i, \a}\}_{\a, \b, i}$ is a basis for
  $\CC(V^n, X_k)$. Let $\{q^n_{i, \a}\}_\a$ be the basis for $\CC(V_i, V^n)$, and
  $\{\pi_{i,k, \b}\}_\b$ the basis for $\CC(X_k, V_i)$ such that
  $$
  p^n_{i, \a} \circ q^n_{i, \a'} = \delta_{\a, \a'} \id_{V_i}, \quad \text{and} \quad
  \pi_{i,k, \b}\circ \iota_{i,k, \b'} = \delta_{\b, \b'} \id_{V_i}\,.
  $$
  Then $\{\frac{1}{d(V_i)} q^n_{i,\a}\circ \pi_{i,k, \b}\}_{\a, \b, i}$ forms a basis for
  $\CC(X_k,V^n)$ dual to $\{\iota_{i,k, \b} \circ p^n_{i, \a}\}_{\a, \b, i}$ relative to the
  non-degenerate bilinear form $(\cdot, \cdot)$.\\ \\
  (i) $\fs$-\emph{equivariance}.
  We first consider the case $ml \ge 0$.
  Then, by Proposition \ref{p:surgery} and Remark \ref{r:alternative},
  $$
  I_V((m,l), [\bX_k]) = \!\sum_{\a, \b, j} \frac{1}{d(V_j)} \ptr\left(
{\def\objectstyle{\scriptstyle}
  \xy
   \vcross~{(-4,4)}{(4,4)}{(-4, -4)}{(4, -4)},
    (-4,6)="ctext1",
    {"ctext1"*{  \iota_{j,k,\b}}},
    {"ctext1"+(-4,-2); "ctext1"+(4,2) **\frm{-}},
    {"ctext1"+(0,8)="ctext3"*{  p^m_{j, \a}}},
    {"ctext3"+(-4,-2); "ctext3"+(4,2) **\frm{-}},
    {"ctext1"+(0,2); "ctext1"+(0,6)**\dir{-}?(0.5)+(2,0)*{V_j}},
    {"ctext3"+(0,2); "ctext3"+(0,6)*{V^m}**\dir{-}},
    {(-10,20)*{V^l}; (-10,-4)**\dir{-}},
    {(-10,-4); (-4,-4)**\crv{(-10,-8)&(-4,-8)}},
    {(4,4); (10,4)**\crv{(4,8) & (10,8)}},
    {(10,4); (10,-20)*{V^l}**\dir{-}},
    (4,-6)="ctext2",
    {"ctext2"*{  \pi_{j,k,\b}}},
    {"ctext2"+(-4.5,-2); "ctext2"+(4,2) **\frm{-}},
    {"ctext2"+(0,-2); "ctext2"+(0,-6)**\dir{-}?(0.5)+(2,0)*{V_j}},
    {"ctext2"+(0,-8)="ctext4"*{  q^m_{j,\a}}},
     {"ctext4"+(-4,-2); "ctext4"+(4,2) **\frm{-}},
    {"ctext4"+(0,-2); "ctext4"+(0,-6)*{V^m}**\dir{-}},
\endxy}
  \right)=\!
  \sum_{\a, \a' \b, i, j}\frac{1}{d(V_j)} \ptr\left(\,
{\def\objectstyle{\scriptstyle}
  \xy
   \vcross~{(-4,4)}{(4,4)}{(-4, -4)}{(4, -4)},
    (-4,6)="ctext1",
    {"ctext1"*{ \iota_{j,k,\b}}},
    {"ctext1"+(-4,-2); "ctext1"+(4,2) **\frm{-}},
    {"ctext1"+(0,8)="ctext3"*{ p^m_{j, \a}}},
    {"ctext3"+(-3,-2); "ctext3"+(3,2) **\frm{-}},
    {"ctext1"+(0,2); "ctext1"+(0,6)**\dir{-}?(0.5)+(2,0)*{V_j}},
    {"ctext3"+(0,2); "ctext3"+(0,6)*{V^m}**\dir{-}},
    {"ctext3"+(-8,0)="ctext5"+(0,-.2)*{p^l_{i, \a'}}},
    {"ctext5"+(-3.5,-2); "ctext5"+(3,2) **\frm{-}},
    {"ctext5"+(0,2); "ctext5"+(0,6)*{V^l}**\dir{-}},
    {"ctext5"+(0,-2); "ctext5"+(0,-18)="v1"**\dir{-}?(.5)+(-1.5,0)*{V_i}},
    {"v1"; (-4,-4)**\crv{"v1"+(0,-5) &(-4,-9)}},
    (4,-6)="ctext2",
    {"ctext2"*{  \pi_{j,k,\b}}},
    {"ctext2"+(-4.5,-2); "ctext2"+(4,2) **\frm{-}},
    {"ctext2"+(0,-2); "ctext2"+(0,-6)**\dir{-}?(0.5)+(2,0)*{V_j}},
    {"ctext2"+(0,-8)="ctext4"*{ q^m_{j,\a}}},
     {"ctext4"+(-3,-2); "ctext4"+(3,2) **\frm{-}},
    {"ctext4"+(0,-2); "ctext4"+(0,-6)*{V^m}**\dir{-}},
    {"ctext4"+(8,0)="ctext6"+(0,-.2)*{q^l_{i, \a'}}},
    {"ctext6"+(-3.5,-2); "ctext6"+(3,2) **\frm{-}},
    {"ctext6"+(0,-2); "ctext6"+(0,-6)*{V^l}**\dir{-}},
    {"ctext6"+(0,2); "ctext6"+(0,18)="v2"**\dir{-}?(.5)+(2,0)*{V_i}},
    {(4,4); "v2"**\crv{(4,8) & "v2"+(0,4)}},
\endxy}\,
  \right).
$$
It follows from \cite[Lemma 5.9]{MugerII03} that
$$
\frac{1}{d(V_j)}\sum_\b
{\def\objectstyle{\scriptstyle}
  \xy
   \vcross~{(-4,5)}{(4,5)}{(-4, -5)}{(4, -5)},
    (-4,7)="ctext1",
    {"ctext1"*{  \iota_{j,k,\b}}},
    {"ctext1"+(-4.5,-2)*{}; "ctext1"+(4.5,2)*{} **\frm{-}},
    {(-4,13)*{V_j}; (-4,9)**\dir{-}},
    {(-11,13)*{V_i}; (-11,-5)**\dir{-}},
    {(-11,-5); (-4,-5)**\crv{(-11,-9)&(-4,-9)}},
    {(4,5); (11,5)**\crv{(4,9) & (11,9)}},
    {(11,5); (11,-13)*{V_i}**\dir{-}},
    {(4,-13)*{V_j}; (4,-9)**\dir{-}},
    (4,-7)="ctext2",
    {"ctext2"*{  \pi_{j,k,\b}}},
    {"ctext2"+(-4.5,-2)*{}; "ctext2"+(4.5,2)*{} **\frm{-}},
\endxy}
= \sum_{\b, k'}  \frac{S_{kk'}}{d(V_i)\dim \CC}
{\def\objectstyle{\scriptstyle}
  \xy
  \vcross~{(-4,5)}{(4,5)}{(-4, -5)}{(4, -5)},
    (-4,7)="ctext1",
    {"ctext1"*{  \iota_{i,k',\b}}},
    {"ctext1"+(-5,-2)*{}; "ctext1"+(4.5,2)*{} **\frm{-}},
    {(-4,13)*{V_i}; (-4,9)**\dir{-}},
    {(4,13)*{V_j}; (4,5)**\dir{-}},
    (4,-7)="ctext2",
    {"ctext2"*{   \pi_{i,k',\b}}},
    {"ctext2"+(-5,-2)*{}; "ctext2"+(4.5,2)*{} **\frm{-}},
     {(-4,-13)*{V_j}; (-4,-5)**\dir{-}},
     {(4,-13)*{V_i}; (4,-9)**\dir{-}},
\endxy}
$$
for all $i,j \in \Gamma$ and $k \in \hG$, where $[S_{ab}]_{\hG}$ is the
$S$-matrix of $Z(\CC)$. Thus, we have
$$
  I_V((m,l), [\bX_k]) =\sum_{\a, \a' \b, i, j, k'} \frac{S_{kk'}}{d(V_i)\dim \CC}\ptr\left(\,
{\def\objectstyle{\scriptstyle}
  \xy
   \vcross~{(-6,4)}{(6,4)}{(-6, -4)}{(6, -4)},
    (-6,6)="ctext1",
    {"ctext1"+(.5,0)*{  \iota_{i,k',\b}}},
    {"ctext1"+(-4,-2); "ctext1"+(4.5,2) **\frm{-}},
    {"ctext1"+(0,8)="ctext3"+(0,-.2)*{ p^l_{i, \a'}}},
    {"ctext3"+(-3.5,-2); "ctext3"+(3,2) **\frm{-}},
    {"ctext1"+(0,2); "ctext1"+(0,6)**\dir{-}?(0.5)+(2,0)*{V_i}},
    {"ctext3"+(0,2); "ctext3"+(0,6)*{V^l}**\dir{-}},
    {"ctext3"+(12,0)="ctext5"*{ p^m_{j, \a}}},
    {"ctext5"+(-3,-2); "ctext5"+(3,2) **\frm{-}},
    {"ctext5"+(0,2); "ctext5"+(0,6)*{V^m}**\dir{-}},
    {"ctext5"+(0,-2); "ctext5"+(0,-11)="v1"**\dir{-}?(.5)+(-2,0)*{V_j}},
    (6,-6)="ctext2",
    {"ctext2"*{  \pi_{i,k',\b}}},
    {"ctext2"+(-4.5,-2); "ctext2"+(4,2) **\frm{-}},
    {"ctext2"+(0,-2); "ctext2"+(0,-6)**\dir{-}?(0.5)+(2,0)*{V_i}},
    {"ctext2"+(0,-8)="ctext4"+(0,-.2)*{ q^l_{i,\a'}}},
     {"ctext4"+(-3.5,-2); "ctext4"+(3,2) **\frm{-}},
    {"ctext4"+(0,-2); "ctext4"+(0,-6)*{V^l}**\dir{-}},
    {"ctext4"+(-12,0)="ctext6"*{ q^m_{j, \a}}},
    {"ctext6"+(-3,-2); "ctext6"+(3,2) **\frm{-}},
    {"ctext6"+(0,-2); "ctext6"+(0,-6)*{V^m}**\dir{-}},
    {"ctext6"+(0,2); "ctext6"+(0,11)="v2"**\dir{-}?(.5)+(2,0)*{V_j}},
\endxy}\,\right)\,.
$$
By Proposition \ref{p:surgery}, the expression on the right hand side is equal to
$$
\displaystyle \sum_{k'} \frac{S_{kk'}}{\dim \CC} \nu_{l,-m}^{\bX_{k'}}(V)
= \sum_{k'} \frac{S_{kk'}}{\dim \CC} I_V((l,-m),  [\bX_{k'}])
= I_V((l,-m), \fs [\bX_k]) = I_V((m, l)\fs, \fs [\bX_k]) ,
$$
where the second last equality holds by definition of the $\SL$-action on $\KK_\k(Z(\CC))$ (cf. \eqref{eq:sl2action}).
  Therefore, $I_V((m,l)\fs, \fs [\bX_k]))=I_V((m,l),  [\bX_k]))$ for all $k \in \hG$, and
  hence  $I_V((m,l), z))=I_V((m,l)\fs, \fs z))$ for all $z \in \KK_\k(Z(\CC))$.
  Replacing $z$ by $\tilde\fs z$,
  we obtain
  $$
  I_V((m,l), \tilde\fs z))= I_V((m,l)\fs, \fs \tilde \fs z)) = I_V((m,l)\fs,   z))
  $$
  for all $z \in \KK_\k(Z(\CC))$ and  $(m,l)\in \BZ^2$ with $ml \ge 0$.

  If $ml \le 0$, then $m(-l) \ge 0$ and $(m,l)=(-l,m)\fs$. In view of the preceding discussion,
  we find
  $$
  I_V((m,l), \tilde \fs z)= I_V((-l,m)\fs , \tilde \fs  z) = I_V((-l,m), \fs^2 z)= I_V((l,-m), z) = I_V((m, l)\fs, z).
  $$
  Here the second last equality follows from Proposition
  \ref{p:dual}, and this completes the proof of $\fs$-equivariance.\\ \\
  (ii) $\ft$\emph{-equivariance}. For  $k \in \hG$ and $m, l \in
\BZ$ with $m>0$,  it follows from Proposition \ref{p:tq} (iii)
that
  $$
  I_V((m, l)\ft, [\bX_k]) =  \nu^{\bX_k}_{m,m+l}(V) = \w_k\inv \nu_{m,l}(V) =
  I_V((m,l), \tilde\ft [\bX_k])
  $$
  where $[\delta_{ab}\w_a]_{\hG}$ is the $T$-matrix of $Z(\CC)$.
  If $m < 0$, then, by Lemma \ref{l:elementary2}, we also have
  $$
  I_V((m, l)\ft, [\bX_k]) = I_{V\du}((-m, -l)\ft, [\bX_k])=I_{V\du}((-m,-l), \tilde\ft [\bX_k])=
  I_{V}((m,l), \tilde\ft [\bX_k])\,.
  $$
  Therefore, $I_V((m, l)\ft, z) = I_V((m, l), \tilde\ft z)$ for all $z \in \KK_\k(Z(\CC))$ whenever $m \ne 0$.
  By Lemma \ref{l:elementary2} (ii), $I_V((0,0), z) = I_I((1,0), z)=I_I((1,1), z)$. Thus,
  $$
  I_V((0,0), \tilde\ft z) = I_I((1,0), \tilde\ft  z) = I_I((1,1), z) =I_V((0,0),  z).
  $$
  Note that $(\tilde \fs \tilde\ft)^3=\tilde\fs^2$. Applying what we have just obtained, we find
  \begin{multline*}
  I_V((0,l), z)=I_V((-l,0)\fs, z) = I_V((-l,0), \tilde\fs z)
  =I_V((-l,0), \tilde\ft \tilde\fs \tilde\ft \tilde \fs \tilde \ft z)
  =I_V((-l,-l),  \tilde\fs \tilde\ft \tilde \fs \tilde \ft z)\\
  =I_V((-l,l),  \tilde\ft \tilde \fs \tilde \ft z)
  =I_V((-l,0),  \tilde \fs \tilde \ft z)
  =I_V((-l,0)\fs,  \tilde \ft z)
  =I_V((0,l),  \tilde \ft z).
  \end{multline*}
  for $l \in \BZ$. In conclusion, we have $I_V((m, l)\ft, z) = I_V((m, l), \tilde\ft z)$
  for all $(m,l) \in \BZ^2$, and  the proof of $\ft$-equivariance is complete.
\end{proof}
The theorem implies that the space of equivariant indicators is closed under the contragredient action of $\SL$ on $\KK_\k(Z(\CC))^*$, and
$$
\fg\cdot  I_V((m,l), -) = I_V((m,l)\tilde \fg\inv, -) \quad \text{for all } \fg \in \SL\,.
$$

One consequence of the $\SL$-equivariance of the indicators is the additivity of $\nu_{m,l}^\bX(-)$ for any pair of relatively prime integers $(m,l)$ and $\bX \in Z(\CC)$. A different proof will be given in section \ref{s:endom}.
\begin{cor}\label{c:additivity}
Let $\CC$ be a spherical fusion category over $\k$, $(m,l)$ a pair of relatively
prime integers, and $z \in \KK_\k(Z(\CC))$. Then
$$
I_{V\oplus W}((m,l), z) = I_{V}((m,l), z) + I_{W}((m,l), z)\quad\text{for }V,W \in \CC \,.
$$
\end{cor}
\begin{proof}
  For  $V \in \CC$ and simple $\bX \in Z(\CC)$, $\nu_{1,0}^\bX(V) = \dim \CC(X, V)$. Therefore,
  $\nu_{1,0}^\bX(V)$ is additive in $V$, and so is $I_V((1,0), z)$ for all $z \in \KK_\k(Z(\CC))$.
  Since $m,l$ are relatively prime, there exists $\fg \in \SL$ such that $(m,l)=(1,0)\fg$.
  By Theorem \ref{t:equivariance}, we find
  \begin{multline*}
   I_{V \oplus W}((m,l), z) = I_{V \oplus W}((1,0)\fg, z)=
  I_{V \oplus W}((1,0), \tilde \fg z) = I_{V}((1,0), \tilde \fg z)
  +I_{W}((1,0), \tilde \fg z) \\
  =  I_{V}((1,0)\fg,  z) +I_{W}((1,0)\fg, z)
  = I_{V}((m,l), z) + I_{W}((m,l), z)\,. \qedhere
  \end{multline*}
\end{proof}

The equivariant indicators shed  new light on the relationship
between the higher indicators for spherical fusion category $\CC$
and the modular data of the $Z(\CC)$. The following corollary
which generalizes \cite[Theorem 4.1]{NS07} is one of the examples.
\begin{cor}\label{c:nu}
Let $\CC$ be a spherical fusion category over $\k$, and $Z(\CC)$ the center of $\CC$ with the ribbon
structure $\theta$ and braiding $c$. Suppose
$\{\bX_i \mid i \in \hG\}$ is a complete set of non-isomorphic simple objects of
$Z(\CC)$ and $[S_{ij}]_{ij \in \hG}$, $[\delta_{ij}\w_j]_{ij \in \hG}$
are the corresponding $S$ and $T$ matrices. Then, for $m \in \BZ$, $i \in \hG$ and $V \in \CC$, we have
\begin{equation}\label{eq:nu(m1)}
  \nu_{m,1}^{\bX_i} (V)  =  \frac{1}{\dim \CC} \sum_{k}\w^{m}_k S_{i k}N_V^{X_k}
  = \frac{1}{\dim \CC}
  \ptr \left(c_{K(V), \bX_{\ol i}} \circ c_{\bX_{\ol i}, K(V)}
  \circ (\id_{\bX_{\ol i}} \o \theta^m_{K(V)})\right)
\end{equation}
where $N_V^{X_k} =\dim \CC(X_k , V)$ and
$K$ is a left adjoint to the forgetful functor from $Z(\CC)$ to $\CC$. In particular, if
$N=\ord \theta$, then
$$
\nu_{N}(V) = d(V) \quad\text{for all $V \in \CC$}.
$$
\end{cor}
\begin{proof}
By Theorem \ref{t:equivariance}, we find
\begin{multline*}
\nu_{m,1}^{\bX_i}(V)= I_V((m,1), [\bX_i]) =  I_V((-m,-1), [\bX_{\ol i}]) =
I_V((1,0)\ft^{-m}\fs, [\bX_{\ol i}]) \\ =  I_V((1,0),\ft^m \fs\inv [\bX_{\ol i}])  =
 I_V((1,0),\ft^m \fs [\bX_i])  =
 \frac{1}{\dim \CC} \sum_{k \in \hG} \w_k^m S_{ik} I_V((1,0), [\bX_k]) \,.
\end{multline*}
It follows from the definition that
$$
I_V((1,0), [\bX_k]) =\nu_{1,0}^{\bX_k}(V) = \dim \CC(X_k, V)=N_V^{X_k}
$$
and so the first equality follows.
Since
$
K(V)\cong \sum_{k \in \hG} N_V^{X_k} \bX_k,
$
we can conclude that
$$
\frac{1}{\dim \CC}\, \ptr\left({ \def\objectstyle{\scriptstyle}
\xy (0,-6)="ctext",
{\vcross~{"ctext"+(-3,5)="m1"}{"ctext"+(3,5)="m2"}{"ctext"+(-3,-5)="y1"}
{"ctext"+(3,-5)="y2"}},
{\vcross~{"m1"+(0,10)="z1"}{"m2"+(0,10)="z2"*+[F]{ \theta^{m}}}
{"m1"}{"m2"}},
{"y2"; "y2"+(0,-4)*{K(V)}**\dir{-}},
{"z2"+(0,2); "z2"+(0,7)*{K(V)}**\dir{-}},
{"z1"; "z1"+(0,6)="z0"**\dir{-}},
{"y1"; "y1"+(0,-4)="y0"**\dir{-}},
{"z0"+(-2,1)*{\bX_i\du}},
{"y0"+(-2.5,0)*{\bX_i\du}},
\endxy }\right)
= \frac{1}{\dim \CC}\sum_{k \in \hG} \w_k^m N_V^{X_k}\ptr\left(
{ \def\objectstyle{\scriptstyle}
\xy (0,-4)="ctext",
{\vcross~{"ctext"+(-3,4)="m1"}{"ctext"+(3,4)="m2"}{"ctext"+(-3,-4)="y1"}
{"ctext"+(3,-4)="y2"}},
{\vcross~{"m1"+(0,8)="z1"}{"m2"+(0,8)="z2"}{"m1"}{"m2"}},
{"y2"; "y2"+(0,-4)*+{\bX_k}**\dir{-}},
{"z2"; "z2"+(0,4)*+{\bX_k}**\dir{-}},
{"y1"; "y1"+(0,-4)*+{\bX_{\ol i}}**\dir{-}},
{"z1"; "z1"+(0,4)*{\bX_{\ol i}}**\dir{-}},
\endxy }\right)
= \frac{1}{\dim \CC}\sum_{k}\w^{m}_k S_{i k}N_V^{X_k}.
$$
In view of \cite[Proposition 5.4]{ENO} (or \cite[Proposition 4.5]{NS07}),
$$
d(K(V)) = (\dim \CC)d(V)\quad \text{for all $V \in \CC$}.
$$
Therefore,
$$
\nu_N(V) = \nu_{N,1}^{\bX_0}(V) =
 \frac{d(K(V))}{\dim \CC} =d(V). \qedhere
$$
\end{proof}

Following \cite[Definition 5.1]{NS07}, the Frobenius-Schur exponent $\FSexp(\CC)$
 of a pivotal category  $\CC$ over $\k$ with simple unit object is defined as the minimum of the set
$$
\{n \in \BN \mid \nu_{n}(V) = d_\ell(V) \text{ for all } V \in \CC\}.
$$
It has been proved in \cite[Theorem 5.5]{NS07} that if $\CC$ is a spherical fusion category
over $\BC$, then $\FSexp(\CC)=\ord \theta$ where $\theta$ is the ribbon structure of $Z(\CC)$.
Indeed, the theorem holds for any spherical fusion category over $\k$.

If $\CC$ is a spherical fusion over $\k$, then we learn immediately from
Corollary \ref{c:nu} that $\FSexp(\CC) \le N$ where $N$ is the order of the ribbon
structure $\theta$ of $Z(\CC)$. Let $\BQ_N$ be the subfield of $\k$ obtained by adjoining
a primitive $N$-th root of unity in $\k$ to $\BQ$.
For $V \in \CC$,  $\nu_N(V)$ is an algebraic integer in $\BQ_N$ (cf. \cite{NS05}),
and so is $d(V)$.
Obviously, if $\w \in \k$ such that
$\theta_\bX = \w \id_\bX$ for some simple $\bX \in Z(\CC)$, then $\w \in   \BQ_N$.
The subfield $\BQ_N$ of $\k$ can be
identified with the $N$-th cyclotomic field contained in $\BC$.
Under this identification, and  by \cite[Corollary 2.10]{ENO}, $d(V)$ is totally real for
simple $V \in \CC$.
Using the same proof of \cite[Theorem 5.5]{NS07}, we have $\FSexp(\CC) = N$.

\begin{prop} \label{p:inQN}
Let $\CC$ be a spherical fusion category over $\k$ with Frobenius-Schur exponent $N$, and
let $\bX_i, i \in \hG$, form a complete set of non-isomorphic simple objects of $Z(\CC)$.
Then
$$\dfrac{1}{\dim \CC} S_{ij}, \,  \nu_{m,l}^{\bX_i}(V)\, \in\,\BQ_N
$$
for all $m,l \in \BZ$, $i,j \in \hG$ and $V \in \CC$, where
$S=[S_{ij}]$ denotes the $S$-matrix of $Z(\CC)$.
\end{prop}
\begin{proof}
>From the above remark, we find $d(V) \in \BQ_N$ for all $V \in \CC$. Hence
  $$
  \dim \CC \quad \text{and} \quad d(\bX_k) \in \BQ_N \quad\text{for $k \in \hG$}.
  $$
Let $\theta$ be the ribbon structure of $Z(\CC)$ and $\w_k \in \k$
such that $\theta_{\bX_k}=\w_k \id_{\bX_k}$ for $k \in \hG$. By
\cite[3.1.2]{BaKi}, we also have
  \begin{equation}\label{eq:S_in_QN}
 S_{ij}=  \w_i\inv \w_j\inv \sum_{k \in \hG} N_{\ol i j}^k \w_k d(\bX_k) \in \BQ_N\,.
  \end{equation}
   Therefore, $\dfrac{1}{\dim \CC} S_{ij} \in \BQ_N$. Note that
   $s=\frac{1}{\dim \CC}S$ is the normalized matrix of $\Z(\CC)$.
   Since $s^4 = 1$, $s\inv$ is also a matrix over $\BQ_N$.

  For any non-zero $(m,l) \in \BZ^2$, it follows from Proposition
  \ref{p:tq} that $\nu_{m,l}^\bX(V) = \nu_{m', l'}^\bX(V^q)$ where
  $q =\gcd(l,m)$, $l'=l/q$ and $m'=m/q$. By Corollary
  \ref{c:additivity}, it suffices to show that $\nu_{m',l'}^\bX(V)
  \in \BQ_N$ for all $V \in \CC$. Let $\fg \in \SL$ such that
  $(m',l') =(1,0)\fg$. By Theorem \ref{t:equivariance},
  \begin{multline}\label{eq:inQN}
   \nu_{m',l'}^{\bX_j}(V) = I_V((1,0)\fg, [\bX_j]) =
   I_V((1,0), \tilde \fg [\bX_j]) \\ = \sum_{i \in \hG} g_{ij}I_V((1,0), [\bX_i]) =
   \sum_{i \in \hG} g_{ij}\nu_{1,0}^{\bX_i}(V) = \sum_{i \in \hG} g_{ij} N^{X_i}_V
  \end{multline}
  where $[g_{ij}]_{\hG} = \rho_{Z(\CC)}(\tilde \fg)$. Since $\fs, \ft$ generate $\SL$,
  $[g_{ij}]$ is a product of the matrices
  $$
  \rho_{Z(\CC)}(\tilde \fs)^{\pm 1} = s^{\mp 1} \quad \text{and}\quad \rho(\tilde \ft)^{\pm 1}=T^{\mp 1}\,,
  $$
  where $T=[\delta_{ij}\w_i]_{\hG}$ is the $T$-matrix of $Z(\CC)$.
  These matrices have been shown to be matrices over $\BQ_N$, and so is $[g_{ij}]_{\hG}$.
  Therefore, the last term in \eqref{eq:inQN} is an element of $\BQ_N$.
\end{proof}

\begin{remark}\label{r:inQN}
  Since $E_{\bX, V}^{(m,1)}$ has finite
  order dividing $mN$ for $m>0$, $\nu_{m,l}^\bX(V)$ is  a cyclotomic integer
  in $\BQ_N$ for $m \ne 0$. It has been shown in \cite{CG}, \cite{BG91} and \cite[Theorem 10.1]{ENO}
  that $\dfrac{1}{\dim \CC} S_{ij} \in \BQ(\xi)$ for some
   root of unity $\xi$. The above proposition proves $\xi$
  can be chosen as a primitive $N$-th root of unity for the modular
  tensor category $Z(\CC)$. It will be developed in Theorem \ref{t:finite} that the image  $\rho^{\lambda, \zeta}_\A(\fs)$ of a modular representation $\rho^{\lambda, \zeta}_\A$ of a modular category $\A$ is a matrix over $\BQ_m$ where $m=\ord(\rho^{\lambda, \zeta}_\A(\ft))$.
\end{remark}
\section{The center of a modular tensor category and congruence subgroups }\label{s:congsub}
In this section, we study the GFS indicators for a modular tensor category. We obtain a generalization of Bantay's formula in Proposition \ref{p:Bantay}, and we prove  in Theorem \ref{t:cong2} that the kernel of the projective modular representation $\ol\rho_\A$ associated with
a modular category $\A$ is a level $N$ congruence subgroup of $\SL$, where $N$ is the
Frobenius-Schur exponent of $\A$. In the case that $\A$ is the center $Z(\CC)$ of some spherical fusion
category $\CC$, we know more precisely that the kernel of the canonical modular representation $\rho_{Z(\CC)}$ of $Z(\CC)$ is a level $N$ congruence subgroup of $\SL$. In fact this is proved first, in Theorem \ref{t:cong1}, and used in the proof of Theorem \ref{t:cong2}. An example for the congruence subgroup theorem is provided at the end this section. We begin with the discussion
of the center of a modular tensor category.

Let $\A=(\A, c, \theta)$ be a modular category over $\k$ with
a complete set of non-isomorphic simple objects $\{U_i\mid i \in \Pi\}$, the $S$-matrix
$S=[S_{ij}]_{\Pi}$, and the
$T$-matrix $T=[\delta_{ij}\w_i]_{\Pi}$. Without loss of generality, we may further assume that
the underlying spherical fusion category of $\A$ is strict.

Let $\bU_{ij}=(U_i\o U_j,\sigma_{U_i\o U_j})$, $(i,j) \in \Pi \times \Pi$, be the complete set of simple objects of $Z(\A)$  described in Section \ref{s:modular}. We have noted $\bU_{ij}\du \cong \bU_{\ol i \ol j}$ in \eqref{dualUij}, and so the $(ij,kl)$-entry of the  $S$-matrix $\uS=[\uS_{ij,kl}]_{\Pi\times \Pi}$ of $Z(\A)$ is given by
\begin{equation}\label{eq:uS}
\uS_{ij,kl}= \ptr\left(
\def\objectstyle{\scriptstyle}
{\xy
{(0,0)="c"},
{"c"+(-4.5, 0)="a1"}, {"c"+(-1.5,0)="a2"},{"c"+(1.5,0)="a3"},{"c"+(4.5,0)="a4"},
{\vcross~{"a2"+(0,4)="b2"}{"a3"+(0,4)="b3"}{"a2"}{"a3"}},
{\vcross~{"a1"+(0,8)="c1"}{"a2"+(0,8)="c2"}{"a1"+(0,4)="b1"}{"b2"}},
{\vcrossneg~{"a3"+(0,8)="c3"}{"a4"+(0,8)="c4"}{"b3"}{"a4"+(0,4)="b4"}},
{\vcrossneg~{"c2"+(0,4)="d2"}{"c3"+(0,4)="d3"}{"c2"}{"c3"}},
{"a1"+(0,4)="b1"; "a1"+(0,-4)**\dir{-}},
{"a4"+(0,4)="b4"; "a4"+(0,-4)**\dir{-}},
{"c4"+(0,4)="d4"; "c4"**\dir{-}},
{"c1"+(0,4)="d1"; "c1"**\dir{-}},
{"d1"+(0,1.5)*{U_{\ol i}}, "d2"+(0,1.5)*{U_{\ol j}},"d3"+(0,1.5)*{U_k}, "d4"+(0,1.5)*{U_l} },
{\vcross~{"a2"+(0,-4)="b2"}{"a3"+(0,-4)="b3"}{"a2"}{"a3"}},
{\vcrossneg~{"a1"+(0,-8)="c1"}{"a2"+(0,-8)="c2"}{"a1"+(0,-4)="b1"}{"b2"}},
{\vcross~{"a3"+(0,-8)="c3"}{"a4"+(0,-8)="c4"}{"b3"}{"a4"+(0,-4)="b4"}},
{\vcrossneg~{"c2"+(0,-4)="d2"}{"c3"+(0,-4)="d3"}{"c2"}{"c3"}},
{"c1"+(0,-4)="d1"; "c1"**\dir{-}},
{"c4"+(0,-4)="d4"; "c4"**\dir{-}},
{"d1"+(1,-1.5)*{U_{\ol j}}, "d2"+(1,-1.5)*{U_{\ol i}},"d3"+(1,-1.5)*{U_k}, "d4"+(1,-1.5)*{U_l} },
\endxy}
\right)
=
 \ptr\left(
\def\objectstyle{\scriptstyle}
{\xy
{(0,0)="c"},
{"c"+(-4.5, 0)="c1"}, {"c"+(-1.5,0)="c2"},{"c"+(1.5,0)="c3"},{"c"+(4.5,0)="c4"},
{"c1"+(0, 4)="a1"}, {"c2"+(0, 4)="a2"}, {"c3"+(0, 4)="a3"}, {"c4"+(0, 4)="a4"},
{"c1"+(0, -4)="b1"}, {"c2"+(0, -4)="b2"}, {"c3"+(0, -4)="b3"}, {"c4"+(0, -4)="b4"},
{"b2"+(0, -4)="e2"}, {"b3"+(0, -4)="e3"}, {"a2"+(0, 4)="d2"}, {"a3"+(0, 4)="d3"},
{\vcross~{"a1"}{"a2"}{"c1"}{"c2"}},
{\vcross~{"c1"}{"c2"}{"b1"}{"b2"}},
{\vcrossneg~{"a3"}{"a4"}{"c3"}{"c4"}},
{\vcrossneg~{"c3"}{"c4"}{"b3"}{"b4"}},
{\vcross~{"b2"}{"b3"}{"e2"}{"e3"}},
{\vcrossneg~{"d2"}{"d3"}{"a2"}{"a3"}},
{"a1"+(-3,0)="a0"}, {"a4"+(3,0)="a5"},
{"b1"+(-3,0)="b0"}, {"b4"+(3,0)="b5"},
{"a0";"a1"**\crv{"a0"+(0,1.5)&"a1"+(0,1.5)}},
{"a4";"a5"**\crv{"a4"+(0,1.5)&"a5"+(0,1.5)}},
{"b0";"b1"**\crv{"b0"+(0,-1.5)&"b1"+(0,-1.5)}},
{"b4";"b5"**\crv{"b4"+(0,-1.5)&"b5"+(0,-1.5)}},
{"a0";"b0"**\dir{-}?(.5)+(-2,0)*{U_i}},
{"a5";"b5"**\dir{-}?(.5)+(2,0)*{U_{\ol l}}},
{"e2"+(-1,-1.5)*{U_{\ol j}}},{"e3"+(1,-1.5)*{U_k}},
{"d2"+(-1,1.5)*{U_{\ol j}}},{"d3"+(1,1.5)*{U_k}},
\endxy}
\right)
 =
 S_{ik} S_{\ol j l}\,.
\end{equation}

Since $\theta_{U_i}= \w_i \id_{U_i}$, we have the equalities
$$
\w_i \,
\def\objectstyle{\scriptstyle}\xy {(0,6)*{U_i}; (0,-6)*{U_i} **\dir{-}}\endxy
\quad =\quad
\xy
{\vcross~{(-2,2)="x1"}{(2,2)="x2"}{(-2,-2)="y1"}{(2,-2)="y2"}},
{"x1"+(-4, 0)="x0"; "x1"**\crv{"x0"+(0,2.5)& "x1"+(0,2.5)}},
{"y1"+(-4, 0)="y0"; "y1"**\crv{"y0"+(0,-2.5)& "y1"+(0,-2.5)}},
{"x0"; "y0"**\dir{-}},
{"y2"+(0,-4)*{U_i}; "y2"**\dir{-}},
{"x2"+(0,4)*{U_i}; "x2"**\dir{-}},
\endxy
\quad \text{and} \quad
\w_i\inv \,
\def\objectstyle{\scriptstyle}\xy {(0,6)*{U_i}; (0,-6)*{U_i} **\dir{-}}\endxy
\quad =\quad
\def\objectstyle{\scriptstyle}
\xy
{\vcrossneg~{(-2,2)="x1"}{(2,2)="x2"}{(-2,-2)="y1"}{(2,-2)="y2"}},
{"x1"+(-4, 0)="x0"; "x1"**\crv{"x0"+(0,2.5)& "x1"+(0,2.5)}},
{"y1"+(-4, 0)="y0"; "y1"**\crv{"y0"+(0,-2.5)& "y1"+(0,-2.5)}},
{"x0"; "y0"**\dir{-}},
{"y2"+(0,-4)*{U_i}; "y2"**\dir{-}},
{"x2"+(0,4)*{U_i}; "x2"**\dir{-}},
\endxy\,.
$$
The   $\bU_{ij}$ component of the ribbon structure of $Z(\A)$ is given by
\begin{equation}\label{eq:centerribbon}
\def\objectstyle{\scriptstyle}
\xy
{(0,-6)="c"},
{"c"+(-4.5, 0)="a1"}, {"c"+(-1.5,0)="a2"},{"c"+(1.5,0)="a3"},{"c"+(4.5,0)="a4"},
{\vcross~{"a2"+(0,4)="b2"}{"a3"+(0,4)="b3"}{"a2"}{"a3"}},
{\vcross~{"a1"+(0,8)="c1"}{"a2"+(0,8)="c2"}{"a1"+(0,4)="b1"}{"b2"}},
{\vcrossneg~{"a3"+(0,8)="c3"}{"a4"+(0,8)="c4"}{"b3"}{"a4"+(0,4)="b4"}},
{\vcrossneg~{"c2"+(0,4)="d2"}{"c3"+(0,4)="d3"}{"c2"}{"c3"}},
{"a3"; "a3"+(0,-4)*{U_i}**\dir{-}},
{"b4"; "a4"+(0,-4)*{U_j}**\dir{-}},
{"c4"+(0,8)="d4"*{U_j}; "c4"**\dir{-}},
{"d3"+(0,4)*{U_i}; "d3"**\dir{-}},
{"d2"+(-10,0)="d1"; "d2"**\crv{"d1"+(0,4)& "d2"+(0,4)}},
{"a2"+(-10,0)="a1"; "a2"**\crv{"a1"+(0,-4)& "a2"+(0,-4)}},
{"d1"; "a1"**\dir{-}},
{"c1"+(-4,0)="c0"; "c1"**\crv{"c0"+(0,2)& "c1"+(0,2)}},
{"b1"+(-4,0)="b0"; "b1"**\crv{"b0"+(0,-2)& "b1"+(0,-2)}},
{"c0"; "b0"**\dir{-}},
\endxy =\frac{\w_i}{\w_j} \id_{U_i \o U_j}\,.
\end{equation}
Thus, the $T$-matrix of $Z(\A)$ is
\begin{equation}\label{eq:uT}
  \underline{T}=\left[\delta_{ij, kl} \frac{\w_i}{\w_j}\right]_{\Pi \times \Pi}.
\end{equation}
Using
Corollary \ref{c:nu}, we can prove the following generalization of \cite[Theorem 7.5]{NS07} which is also a further generalization of Bantay's formula \cite{Bantay97} to GFS indicators.
\begin{prop}\label{p:Bantay}
  Let $\A$ be a modular category over $\k$ with a complete set of non-isomorphic simple objects
  $\{U_i\mid i \in \Pi\}$.
  Then, for $\bU_{ij}=(U_i \o U_j, \hbr_{U_i \o U_j}) \in Z(\A)$,
  $$
  \nu_{m,1}^{\bU_{ij}}(U_a) = \frac{1}{\dim \A}\sum_{k,l \in \Pi} \left(\frac{\w_k}{\w_l}\right)^m
  S_{ik} S_{\ol j l} N_{kl}^a
  $$
where $N_{kl}^a=\dim \A(U_k\o U_l, U_a)$ and $[S_{ij}]_\Pi$, $[\delta_{ij} \w_i]_\Pi$ are the
$S$ and $T$-matrices of $\A$ respectively.
\end{prop}
\begin{proof}
  The $S$ and $T$-matrices of $Z(\A)$ have been shown in \eqref{eq:uS} and \eqref{eq:uT}. By Corollary
  \ref{c:nu}, we find
  $$
  \nu_{m,1}^{\bU_{ij}}(U_a) = \frac{1}{\dim \A} \sum_{k, l\in \Pi} \left(\frac{\w_k}{\w_l}\right)^m
  \uS_{ij, kl} N_{kl}^a =  \frac{1}{\dim \A}\sum_{k,l \in \Pi} \left(\frac{\w_k}{\w_l}\right)^m
  S_{ik} S_{\ol j l} N_{kl}^a. \qedhere
  $$
\end{proof}
Note that the $(2,1)$-st indicator for the unit object $\bU_{00}$ given by
$$
\nu_{2,1}^{\bU_{00}}(U_a)=\frac{1}{\dim \A}\sum_{k,l \in \Pi} \left(\frac{\w_k}{\w_l}\right)^2
  S_{0k} S_{0 l} N_{kl}^a = \sum_{k,l \in \Pi} \left(\frac{\w_k}{\w_l}\right)^2
  s_{0k} s_{0 l} N_{kl}^a
$$
is identical to Bantay's indicator formula for RCFT introduced in \cite{Bantay97}.

The representation  $\rho_{Z(\A)}$ is determined by the actions
\begin{equation}\label{eq:MZC}
\fs [\bU_{ij}]=\frac{1}{\dim \A}\sum_{k,l} \uS_{kl, ij} [\bU_{kl}] = \frac{1}{\dim \A}\sum_{k,l} S_{ki} S_{\ol l j} [\bU_{kl}],
\quad \text{and}\quad
\ft [\bU_{kl}]=\w_k \w_l\inv [\bU_{kl}]\,,
\end{equation}
and it is isomorphic to a tensor product of two representations  as described in the following lemma.
\begin{lem} \label{l:iso}
  Let $\A$ be a modular category over $\k$ with a complete set of non-isomorphic simple objects $\{U_i\mid i \in \Pi\}$, and let $\rho$ denote the representation  $\rho_\A^{\lambda, \zeta}$ for
  some $\lambda, \zeta \in \k$ such that
  $\lambda^2=\dim \A$ and $\zeta^3 = p_\A^+/\lambda$.  Then:
  \begin{enumerate}
  \item[(i)] The $\k$-linear isomorphism $\phi: \KK_\k(\A)\o\KK_\k(\A)\to\KK_\k(Z(\A)),\, [U_i] \o [U_j] \mapsto [\bU_{ij}]$ from \eqref{tpiso}
   defines an isomorphism $\rho\otimes\tilde\rho\to\rho_{Z(\A)}$ of representations of $\SL$.
   \item[(ii)] The bilinear form $\langle\cdot,\cdot\rangle\colon\KK_\k(\A)\o\KK_\k(\A)\to \k$ defined by $\langle[U_i],[U_j]\rangle=\delta_{ij}$ is $\SL$-invariant under the representation $\rho\o\tilde\rho$.
  \end{enumerate}
\end{lem}
\begin{proof}  Note that
the $\SL$-action
on $\KK_\k(\A)$ associated with the representation $\rho$ is given by
$$
\fs[U_j] = \frac{1}{\lambda}\sum_{i \in \Pi}S_{ij}[U_i], \quad \text{and}\quad
 \ft[U_j] = \frac{\w_j}{\zeta} [U_j]
$$
where $[S_{ij}]_\Pi$ and $[\delta_{ij}\w_j]_\Pi$ are the $S$ and $T$ matrices of $\A$.
Note that $S_{ik}=S_{ki}$ and $S_{\ol j l}=S_{j \ol l}$.
By \eqref{eq:Sprop} the representation $\rho\o\tilde\rho$ satisfies
\begin{equation}\label{eq:Saction}
 \fs([U_i]\o[U_j])=\fs[U_i]\o\tilde\fs[U_j]
  =\fs[U_i]\o\fs\inv[U_j]
  =\frac{1}{\lambda^2}\sum_{k,l \in \Pi} S_{ki}S_{l \ol j}
   [U_k] \o [U_l]
\end{equation}
and also
\begin{equation}\label{eq:Taction}
\ft([U_i] \o [U_j])= \ft[U_i] \o \tilde\ft [U_j]=
  \ft[U_i] \o \ft\inv [U_j] = \frac{\w_i}{\zeta} \left(\frac{\w_j}{\zeta}\right)\inv
  [U_i] \o [U_j]\,.
\end{equation}
Comparing with \eqref{eq:MZC} we see that $\phi$ satisfies $\phi(\fg([U_i]\o[U_j]))=\fg\phi([U_i]\o[U_j])$ for $\fg\in\{\fs,\ft\}$ which implies that $\phi$ is $\SL$-equivariant.

Now apply the bilinear form $\langle\cdot,\cdot\rangle$ to the rightmost expressions in \eqref{eq:Saction} and \eqref{eq:Taction}. It follows from \eqref{eq:Sprop} that they are both equal to $\delta_{ij}$, and this proves the second statement.
\end{proof}
\begin{defn}
 The kernel $\Gamma(n)$ of the natural group homomorphism $\SL \to \qsl{n}$
is called the \textbf{principal congruence subgroup of level $n$}. A finite index subgroup $G$ of $\SL$ is called
a \textbf{congruence subgroup} if $G$ contains a principal congruence
subgroup of $\SL$.  If $n$ is the least positive integer such that
$\Gamma(n) \subseteq G$, then $G$ is called a \textbf{congruence subgroup of level $n$}.
\end{defn}

In view of \eqref{eq:tildefg},
\begin{equation}\label{eq:tildeinv}
\widetilde{\Gamma(n)} = \Gamma(n) \quad\text{for all positive integers } n\,.
\end{equation}

We proceed to show that the principal congruence subgroup $\Gamma(N)$, where $N$ is the Frobenius-Schur exponent, always fixes the equivariant indicators.

\begin{lem}\label{l:inv1}
Let $\CC$ be a spherical fusion category over $\k$ with $\FSexp(\CC)=N$. Then
$$
I_V((m,l), \tilde\fg z) =  I_V((m,l),  z)=I_V((m,l), \fg z)
$$
for all $m,l \in \BZ$, $V \in \CC$, $z \in \KK_\k(Z(\CC))$ and $\fg \in \Gamma(N)$.
\end{lem}
\begin{proof} Since $\widetilde{\Gamma(N)} =\Gamma(N)$, the first and second equality are equivalent.
It suffices to show one of these two equalities holds.
Note that $\ft^N \in \Gamma(N)$ and $\rho_{Z(\CC)}(\ft^N)=\id$.
 Therefore, $\ker \rho_{Z(\CC)}$ contains the normal closure of $\ft^N$ in $\SL$.
For $N=2$, it is well-known that the normal closure of $\ft^2$ in $\SL$ is $\Gamma(2)$
(cf. \cite{Bre60}). Thus, $\fg z = z$ for all $\fg \in \Gamma(2)$ and $z \in \KK_\k(Z(\CC))$.
In particular, we have $I_V((m,l), \fg z) =  I_V((m,l),  z)$.

  Now, we may assume $N >2$ and  consider the relation $\sim$ on $\BZ^2$ defined by
  $$
  (m,l) \sim (m',l') \quad \text{iff}\quad I_V((m,l), z) =I_V((m',l'), z) \quad \text{for all }z \in \KK_\k(Z(\CC)),\,   V \in \CC \,.
  $$
  It is obvious that $\sim$ is an equivalence relation on $\BZ^2$. By Theorem \ref{t:equivariance},
  if $(m,l) \sim (m',l')$, then
  $$
  (m,l)\fg \sim (m',l')\fg \quad \text{for all $\fg \in \SL$}.
  $$
  We need to show that
  $$
  (m,l) \sim (m, l)\fg \quad \text{for all $(m,l) \in \BZ^2$ and $\fg \in \Gamma(N)$}.
  $$
  To prove this, we use  a version of \cite[Theorem 1.3]{SZh} which requires
  to verify the following conditions for each $(m,l) \in \BZ^2$:
  \begin{enumerate}
    \item[(i)] $(m,l) \sim (m, mN+l)$ and
    \item[(ii)] $(m,l) \sim (m,kl)$ whenever $\gcd(m,l)=\gcd(m,kl)$ for some integer $k \equiv 1 \pmod N$.
  \end{enumerate}
  The first condition follows directly from
  Theorem \ref{t:equivariance} and the fact that $\rho_{Z(\CC)}(\tilde\ft^N)=\id$. For the second condition, we consider $m,l,k \in \BZ$ such that
  $k\equiv 1 \pmod N$ and $\gcd(m,l)=\gcd(m,kl)=q$. Obviously, if $l=0$,
  then $(m,l) \sim (m,kl)$. We may assume $l \ne 0$. In this case, $q \ge 1$ and
  $\gcd(m/q,l/q)=\gcd(m/q,kl/q)=1$. If
  $(m/q, l/q)\sim (m/q, kl/q)$, then
  $$
  I_V((m,l), z)=I_{V^q}((m/q, l/q), z) = I_{V^q}((m/q, kl/q), z) = I_V((m, kl), z)
  $$
  for all $V \in \CC$ and $z \in \KK_\k(Z(\CC))$. Hence $(m,l)\sim (m,kl)$. Therefore,
  it suffices to prove $(m,l)\sim (m,kl)$ for $\gcd(m,l)=\gcd(m,kl)=1$.
  If $m=0$, then this condition forces $k = \pm 1$.
  Since $k \equiv 1 \pmod N$ and $N>2$,  $k=1$, and
  hence $(0,l)\sim (0, kl)$. So, we may further assume $m \ne 0$. Since $\gcd(m,kl)=1$,
  $k$  and $mN$ are relatively prime. Let $\xi\in \k$ be a primitive $|m|N$-th root of unity and consider the
  automorphism $\sigma_k \in \Gal(\BQ_{|m|N}/\BQ)$
  defined by $\sigma_k : \xi \mapsto \xi^k$.  Since $k \equiv 1 \pmod N$, we
  have $\sigma_k(\xi^m) =\xi^m$ or equivalently $\sigma_k|_{\BQ_N}=\id$.
  Since $\theta_\bX^N=\id_\bX$ for $\bX \in Z(\CC)$, by Lemmas \ref{Ecomposition} and \ref{Emm}, we have
  $\left(E^{(m)}_{\bX, V}\right)^{mN}=\id$. Therefore,
  $$
  \sigma_k(\nu_{m,l}^{\bX}(V))
  = \sigma_k \left(\Tr\left( \left(E^{(m)}_{\bX, V}\right)^l\right)\right)
  =\Tr\left( \left(E^{(m)}_{\bX, V}\right)^{kl}\right)=
  \nu_{m,kl}^{\bX}(V)\,.
  $$
  On the other hand, $\nu_{m,l}^{\bX}(V) \in \BQ_N$ by
  Proposition \ref{p:inQN}, and so it is fixed by $\sigma_k$. Thus,
  $$
  I_V((m,l), [\bX]) = I_V((m,kl), [\bX])\quad  \text{for all $V \in \CC$ and simple $\bX \in Z(\CC)$}.
  $$
  Hence, $(m,l) \sim (m,kl)$.
\end{proof}

\begin{lem}\label{l:inv2}
  Let $\CC$ be a spherical fusion category over $\k$ with
  $\FSexp(\CC)=N$. Suppose $\bX_0$ is the unit object of $Z(\CC)$.
 Then $[\bX_0]\in \KK_\k(Z(\CC))$ is $\Gamma(N)$-invariant.
\end{lem}
\begin{proof}
   Let $f \in \KK_\k(Z(\CC))^*$ be defined by
  $$
  f(z)=\frac{1}{\dim \CC} \sum_{k \in \Gamma} d(V_k) I_{V_k}((0, 1),z) \quad\text{for}\quad z \in \KK_\k(Z(\CC)),
  $$
  where $\{V_k \mid k \in \Gamma\}$ is a complete  set of non-isomorphic
   simple objects in $\CC$.
  As a consequence of Lemma \ref{l:inv1}, we have $f(\fg z) = f(z)$ for all
  $z \in \KK_\k(Z(\CC))$ and
  $\fg \in \Gamma(N)$.
  By Theorem \ref{t:equivariance}, we find
  \begin{multline*}
  f(\fs[\bX_j]) =  \frac{1}{\dim \CC}\sum_{k \in \Gamma} d(V_k) I_{V_k}((0, 1),\fs[\bX_j])
  = \frac{1}{\dim \CC} \sum_{k \in \Gamma} d(V_k) I_{V_k}((0, 1)\fs\inv,[\bX_j]) \\
  = \frac{1}{\dim \CC} \sum_{k \in \Gamma} d(V_k) I_{V_k}((1, 0),[\bX_j])
  = \frac{1}{\dim \CC}d(\bX_j)
   =s_{0j} =\langle[\bX_0],\fs[\bX_j]\rangle
  \end{multline*}
  and thus $f(z)=\langle[\bX_0],z\rangle$ for all $z\in\KK_\k(Z(\CC))$.
  Now
  $ \langle\fg[\bX_0],z\rangle
   =\langle[\bX_0],\tilde\fg\inv z\rangle
   =f(\tilde\fg\inv z)
   =f(z)=\langle[\bX_0],z\rangle$
  for all $z\in\KK_\k(Z(\CC))$ by Lemma \ref{l:iso} (ii), and the result follows.
\end{proof}
\begin{lem}\label{l:inv3}
 Let $\A$ be a modular tensor category over $\k$ with $\FSexp(\A)=N$.
  Suppose $\lambda, \zeta \in \k$
  such that $\lambda^2=\dim \A$ and $\zeta^3= p_\A^+/\lambda$ and consider the representation $\rho_\A^{\lambda,\zeta}$ of $\SL$ on $\KK_\k(Z(\CC))$.
   Then, for $z, z' \in \KK_\k(\A)$,
   $$
   (\fg z)(\tilde \fg z') = zz' \quad \text{for all }\fg \in \Gamma(N)\,.
   $$
\end{lem}
\begin{proof}
  Consider the non-degenerate bilinear form $\langle\cdot,\cdot\rangle$  on $\KK_\k(\A)$ from lemma \ref{l:iso}. Let $m\colon\KK_\k(\A)\o\KK_\k(\A)\to\KK_\k(\A)$ denote multiplication in the Grothendieck algebra. Then, by Lemma \ref{l:iso},
  we have
  $$
  \langle m\phi\inv(\bU_{ij}), [U_k] \rangle = \langle [U_i][U_j], [U_k]\rangle =
  \dim \A(U_i \o U_j, U_k) = I_{U_k}((1,0), [\bU_{ij}])
  $$
  for all $i, j \in \Pi$. Therefore,
  $$
  \langle m\phi\inv(w), [U_k] \rangle = I_{U_k}((1,0), w) \quad\text{for all }w \in \KK_\k(Z(\A)).
  $$
  By Lemma \ref{l:inv1}, for  $\fg\in\Gamma(N)$,
  $$\langle m\phi\inv(\fg w),[U_k]\rangle=I_{U_k}((1,0),\fg w)=I_{U_k}((1,0),w)=\langle m\phi\inv(w),[U_k]\rangle$$
  and hence $m\phi\inv(\fg w)=m\phi\inv(w)$. Now for $w=\phi(z\o z')$ we have $\phi\inv(\fg w)=\fg z\o \tilde\fg z'$ by Lemma \ref{l:iso} and the claim follows.
\end{proof}
\begin{thm}\label{t:cong1}
  Let $\CC$ be a spherical fusion category over $\k$ with $\FSexp(\CC)=N$.
  The kernel of the canonical modular representation
  $\rho_{Z(\CC)}: \SL \to \GL(\KK_\k(Z(\CC)))$ of $Z(\CC)$ is a congruence subgroup of level $N$.
  In particular, $\KK_\k(Z(\CC))$ is $\Gamma(N)$-invariant,
\end{thm}
\begin{proof}
  Recall  that $\KK_\k(Z(\CC))$ is a $\k$-algebra
  with  $[\bX_0]$ as the identity element. It follows from Lemmas \ref{l:inv2} and \ref{l:inv3} that
  $$
  z=z [\bX_0]= (\fg z) (\tilde\fg [\bX_0])= (\fg z)[\bX_0] =\fg z
  \quad \text{for all }\fg \in \Gamma(N), z \in \KK_\k(Z(\CC))\,.
  $$
  Therefore, $\Gamma(N) \subseteq \ker \rho_{Z(\CC)}$. Suppose
  $\Gamma(N') \subseteq \ker \rho_{Z(\CC)}$ for some positive integer $N' \le N$.
  Then $\ft^{N'} \in  \ker \rho_{Z(\CC)}$ or $\ft^{N'}z =z$ for all $z \in \KK_\k(Z(\CC))$. Therefore, $T^{N'}=1$
  where $T$ is the $T$-matrix of $Z(\CC)$. Since $\ord(T) =N$ (cf. \cite[Theorem 5.5]{NS07}),
  $N \mid N'$ and so $N=N'$.
\end{proof}

\begin{thm}\label{t:cong2}
  Let $\A$ be a modular category over $\k$ with $\FSexp(\A)=N$. Then the kernel of the
  projective modular representation $\ol \rho_\A$ of  $\A$ is a congruence subgroup of level $N$.
\end{thm}
\begin{proof}
  Let $\lambda, \zeta \in \k$ such that  the representation $\rho:=\rho_\A^{\lambda,\zeta}$ is well-defined.
  In view of Lemma \ref{l:iso},  $\rho\o\tilde\rho\cong\rho_{Z(\A)}$.
  Therefore, by Theorem \ref{t:cong1},
  \begin{equation}\label{eq:g-invariant}
 \fg z \o \tilde \fg z' = z \o z'\quad\text{for all $z, z' \in \KK_\k(\A)$ and
  $\fg \in \Gamma(N)$}.
  \end{equation}
  Pick $z'\in\KK_\k(\A)$ and $\e\in\KK_\k(\A)^*$ with $\e(z')=1$. Then \eqref{eq:g-invariant} implies $\fg z\e(\tilde\fg z')=z$ for all $z\in\KK_\k(\A)$.
   In particular, $\rho_{\A}^{\lambda, \zeta}(\fg)$ is a scalar
  multiple of $\id_{\KK_\k(\A)}$ and hence
  $\ol \rho_\A(\fg)$ is the unit of $\PGL(\KK_\k(\A))$. Thus, $\Gamma(N) \subseteq \ker \ol \rho_\A$.

   Suppose $\Gamma(N') \subseteq \ker \ol \rho_\A$ for some  positive integer $N' \le N$. Then
   $\rho_{\A}^{\lambda, \zeta}(\ft^{N'})=\a \id$ for some nonzero scalar $\a \in \k$.
   Therefore,
   $$
   \frac{1}{\zeta^{N'}} T^{N'} = \a\id
   $$
   where $T=[\delta_{ij}\w_i]_\Pi$ is the $T$-matrix of $\A$. Since $\w_0 =1$,
   we find
   $$
   1 = \w_0 =\zeta^{N'} \a
   $$
   and hence $T^{N'}=\id$. This implies  $N\mid N'$.
\end{proof}

\begin{example}{\rm
  Let $G$ be a finite abelian group and  $\w$ a normalized complex valued 3-cocycle on $G$ such that quasi-Hopf algebra $D^\w(G)$, introduced in \cite{DPR90}, is commutative.  By \cite[Corollary 3.6]{MN01},  the function $\theta_z$ defined by
  $$
  \theta_z(x,y)=\frac{\w(z,x,y)\w(x,y,z)}{\w(x,z,y)}
  $$
  is a 2-coboundary for all $z \in G$. Let $t_z: G \to \BC^\times$ be a normalized cochain such that $\theta_z= \delta t_z$, i.e. $\theta_z(x,y) = t_z(x)t_z(y)/t_z(xy)$. In addition, we chose $t_1=1$.

    Following Section 9 of \cite{MN01}, the irreducible characters $\chi_{\a, u}$ of $D^\w(G)$ are indexed by the group $\hat G \times G$ where $\hat G$ is the character group of $G$. As a vector space $D^\w(G)=\BC[G]^* \o \BC[G]$. Let $\{e(u)\mid u \in G\}$ be the basis of $\BC[G]^*$ dual to $G$. Then
  \begin{equation}\label{eq:chi}
   \chi_{\a,u}(e(h) \o y) = \a(y) t_u(y) \delta_{h,u}
  \end{equation}
  for $\a \in \hat G$ and $h,u,y \in G$. The universal $R$-matrix of $D^\w(G)$ and the canonical ribbon structure (cf. \cite{AC92} and \cite[p869]{GMN}) are given respectively by
  $$
  R = \sum_{g,h \in G} (e(g) \o 1) \o (e(h) \o g),\quad  v =\sum_{g \in G} e(g) \o g\,.
  $$
  Since the pivotal trace of the canonical pivotal structure of $D^\w(G)$ is equal to the ordinary trace, it follows from \eqref{eq:chi} that
   the $((\a_1, u_1), (\a_2, u_2))$-entry of the $S$-matrix for $\C{D^\w(G)}$ is
 \begin{equation} \label{eq:S}
    b_\w((\a_1, u_1), (\a_2, u_2))=(\chi_{\a_1, u_1} \o \chi_{\a_2, u_2}) (R^{21}R) = \a_1(u_2)\a_2(u_1) t_{u_1}(u_2) t_{u_2}(u_1)
 \end{equation}
 and the $((\a, u), (\a, u))$-entry of the $T$-matrix  is
 \begin{equation} \label{eq:T}
   q(\a,u)=\chi_{\a, u} (v)=\a(u) t_{u}(u)\,.
  \end{equation}
 It is worth to note that $q$ is a quadratic form canonically defined on the group $\G^\w$ of group-like elements of $D^\w(G)$ and $b_\w$ is the associated non-degenerate bicharacter on $\G^\w$ defined in \cite[p3491]{MN01}. By \cite[Theorem 9.2]{NS07}, the Frobenius-Schur exponent of $\C{D^\w(G)}$ is given by the formula
 \begin{equation} \label{eq:exp}
\FSexp(D^\w(G))=\lcm |\w_C||C|
 \end{equation}
 where $C$ runs through all the maximal cyclic subgroups of $G$ and $|\w_C|$ denotes the order of the restriction of the cohomology class of $\w$ to $C$. Moreover, $\C{D^\w(G)} = Z(\C{H})$ for a certain semisimple quasi-Hopf algebra $H$ of dimension $|G|$ \cite{Maj98}.

 Now we consider  the order 2 multiplicative group  $G=\{1,x\}$, and let $\a$ be the non-trivial character of $G$.  Then $H^3(G, \BC^\times) =\BZ_2$, $D^\w(G)$ is commutative for all normalized 3-cocycles $\w$ of $G$, and the irreducible characters of $D^\w(G)$ are indexed by $\{(1,1), (\a,1), (1,x), (\a,x)\}$.

  (i) For $\w\equiv 1$, we can chose $t_z=1$ for all $z \in G$. It follows from \eqref{eq:S} and \eqref{eq:T} that  the normalized $S$-matrix and $T$-matrix are
 $$
s=\frac{1}{2}\left[
\begin{array}{cccc}
    1   &   1    &   1   &   1  \\
    1   &   1    &  - 1   &  - 1  \\
    1   &   -1    &   1   &  - 1  \\
    1   &   -1    &  - 1   &   1
\end{array}\right], \quad
T=\left[
\begin{array}{cccc}
    1   &   0   &   0   &   0  \\
    0   &   1    &   0   &  0  \\
    0   &   0    &   1   &   0  \\
    0   &   0    &   0   &   -1
\end{array}\right]\,.
$$
Since $s^2=T^2=(sT)^3=\id$ and $s(sT)s = Ts=(sT)\inv$, the image of the canonical representation $\rho$ of $\C{D(G)}$ is isomorphic $S_3$ and so $\ker \rho =\G(2)$.

(ii) For the non-trivial class of $H^3(G, \BC^\times)$, we consider $\w: G^3 \to \BC^\times$ defined by $\w(a,b,c)=-1$ if $a=b=c=x$ and $\w(a,b,c)=1$ otherwise. Then $\w$ is a non-trivial normalized 3-cocycle on $G$. If one  defines $t_z: G \to \BC^\times$ as
    $t_z(a)=i$ whenever $a=z=x$ and $t_z(a)=1$ otherwise, then $\theta_z = \delta t_z$ (cf. \cite[p857]{GMN}). Using the same index set for the irreducible characters as in case (i), the normalized $S$ and $T$-matrices of $\C{D^\w(G)}$ are
     $$
s=\frac{1}{2}
\left[
\begin{array}{cccc}
    1   &   1    &   1   &   1  \\
    1   &   1    &  - 1   &  - 1  \\
     1   &   -1    &  - 1   &   1\\
    1   &   -1    &   1   &  - 1
\end{array}\right], \quad
T=\left[
\begin{array}{cccc}
    1   &   0   &   0   &   0  \\
    0   &   1    &   0   &  0  \\
    0   &   0    &   i   &   0  \\
    0   &   0    &   0   &   -i
\end{array}\right]\,.
$$
Then $s^2=T^4=(sT^2)^4=(sT)^3=\id$ and $\FSexp(D^\w(G))=4$. Since $ssT^2s=T^2s=(sT^2)^{-1}$, the subgroup generated by $s, sT^2$ is a dihedral group of order 8 and hence the image of the canonical representation $\rho$ of $\C{D^\w(G)}$ contains at least 24 elements. Since $s^2=\id$ and $\fs^2=-1$, we have $\G(4)\langle 1, -1\rangle \subseteq \ker \rho$, and thus $\im \rho$ is a homomorphic image of $\SL/ \G(4)\langle 1, -1\rangle =\mbox{PSL}(2, \BZ_4)\cong S_4$.  Thus we have
$$
\ker \rho = \G(4)\langle 1, -1\rangle \quad \text{and}\quad \im \rho \cong S_4\,.  \qed
$$
  }
\end{example}

More examples of small modular categories can be found in \cite{RSW}.
\section{Modular representations and a conjecture of Eholzer}\label{s:eholzer}
A  matrix representation $\rho: \SL \to \GL(n, \k)$ which has finite image is called $\ft$-\emph{rational} if
$\im \rho \subseteq \GL(n, \BQ_m)$ where $m=\ord(\rho(\ft))$.
It is conjectured in \cite{E95} that the representation $\rho :\SL \to \GL(n, \BC)$ associated with a RCFT satisfies the conditions:
\begin{enumerate}
  \item[(i)] The kernel of $\rho$ is a congruence subgroup of $\SL$, and
  \item[(ii)] $\rho$ is $\ft$-\emph{rational}.
\end{enumerate}
In this section, we prove that every modular representation of a modular category has finite image and is $\ft$-rational.

Let $\A$ be a modular category over $\k$ with a complete set   of non-isomorphic simple objects $\{U_i\mid i \in \Pi\}$. We denote by $M_\Pi(R)$ the ring of square matrices indexed by $\Pi$ over a commutative ring $R$, and $\GL(\Pi, R)$ the group of invertible matrices in $M_\Pi(R)$. \emph{A modular representation of $\A$} is an ordinary group representation $\xi: \SL \to \GL(\Pi, \k)$ such that $\ol\rho_\A(\fg) = \pi (\xi(\fg))$ for all $\fg \in \SL$, where $\pi: \GL(\Pi, \k) \to \PGL(\Pi, \k)$ is the natural surjection. In particular, $\rho_{\A}^{\lambda, \zeta}$, for $\lambda, \zeta \in \k$ satisfying \eqref{eq:ls}, is a modular representation of $\A$.

Suppose $\xi_i: \SL \to \GL(\Pi, \k)$, $i=1,2$, are modular representations of $\A$. Then there exist $x_\fs$, $x_\ft \in \k^\times$  such that
$$
\xi_2(\fs) = x_\fs  \xi_1(\fs)  \quad \text{and}\quad \xi_2(\ft) = x_\ft \xi_1(\ft) \,.
$$
Using the relations of the $\fs$ and $\ft$, we find
$$
x_\fs^4 =1, \quad x_\ft^3 x_\fs^3 = x_\fs^2,
$$
and hence $x_\ft^{12}=1$ and $x_\fs = x_\ft^{-3}$. This implies there are  12  modular representations of $\A$. Moreover, for $\fg \in \ker\xi_1$, $\xi_2(\fg) = \a(\fg) \id$ where $\a(\fg)$ is some power of $x_\ft$. Since $x_\ft^{12}=1$, $\xi_2(\ker \xi_1)$ is isomorphic to a subgroup of $\BZ_{12}$. In particular, we find
$$
\frac{(\ker \xi_1)(\ker \xi_2)}{\ker \xi_2} \cong \frac{\ker \xi_1}{\ker \xi_1 \cap \ker \xi_2} \le \BZ_{12}\,.
$$
By the same argument, $(\ker \xi_1)(\ker \xi_2)/\ker \xi_1$ is a cyclic group of order dividing $12$.
Consequently, $\ker \xi_1$ is a finite index subgroup of $\SL$ if, and only if, $\ker \xi_2$ is of finite index.

Now, we consider a modular representation $\xi=\rho_\A^{\lambda, \zeta}$ where $\lambda, \zeta \in \k$ satisfy \eqref{eq:ls}. Suppose  $N=\FSexp(\A)$, $[S_{ij}]_\Pi$ and $[\delta_{ij}\w_i]_\Pi$ are the $S$ and $T$ matrices of $\A$ respectively. We have pointed out in the paragraph preceding Proposition \ref{p:inQN}  that $d(U_i) \in \BQ_N$ for $i \in \Pi$. Therefore,
$$
S_{ij}, \,  p_\A^\pm \in \BQ_N
$$
and so $\zeta^6 = p_\A^+/p_\A^- \in \BQ_N$. Since $\zeta^6 \in \BQ_N$ is a root of unity, $\ord(\zeta^6)$ divides $2N$. Thus, $\zeta^{12N}=1$ and hence
  \begin{equation} \label{eq:im_t}
  \xi(\ft)^{12N}=\id.
  \end{equation}
   Obviously, $\xi(\fs)^4=\id$, and so we have $(\det \xi(\fg))^{12N}=1$ for all $\fg \in \SL$.

   By Theorem \ref{t:cong2}, there exists a group homomorphism $\a : \Gamma(N) \to \k^\times$ such that
   $\xi(\fg)=\a(\fg)\id$ for all $\fg \in \Gamma(N)$. Therefore, $\a(\fg)^{12N|\Pi|}=1$. In particular, the image of $\a$ is a finite cyclic group. Now we find
   $$
   \frac{(\ker \xi) \Gamma(N)}{\ker \xi} \cong \frac{\Gamma(N)}{ \Gamma(N)\cap \ker \xi} \cong \xi(\Gamma(N)) \cong \a(\Gamma(N))\,.
   $$
  Since $\Gamma(N)$ is a finite index subgroup of $\SL$,  so is $\ker \xi$. Hence, all the modular representations of $\A$ have finite images.

  In view of the preceding remark, $\{\xi_x\mid x\in \k, x^{12}=1\}$ is the set of all modular representations of $\A$, where $\xi_x: \SL \to \GL(\Pi, \k)$ given by
  \begin{equation} \label{eq:xi_x}
    \xi_x(\fs) := \frac{1}{x^3} \xi(\fs)= \frac{1}{x^3 \lambda}S , \quad  \xi_x(\ft) := x \xi(\ft)=\frac{1}{\zeta/x}T\,.
  \end{equation}
  It follows from \eqref{eq:im_t} that $m:=\ord(\xi_x(\ft))$ divides $12N$. Since $(\frac{\w_i} {\zeta/x})^m=1$ for all $i \in \Pi$ and $\w_0=1$, we find
  $(\zeta/x)^m=1$ and so $\w_i^m=1$ for all $i \in \Pi$. Therefore, $N \mid m \mid 12 N$ and
  $
  \zeta/x, \w_i \in \BQ_m \, \text{for all }i \in \Pi.
  $
  Since $\lambda \zeta^3 = p_\A^+ \in \BQ_m$, we also have $x^3 \lambda = \frac{p_\A^+}{(\zeta / x)^3} \in \BQ_m$. Since $S \in \GL(\Pi, \BQ_N)$, we have $\xi_x(\fs)$ and $\xi_x(\ft) \in \GL(\Pi,\BQ_m)$. This completes the proof of

\begin{thm}\label{t:finite}
  Let $\A$ be a modular tensor category over $\k$. Then every modular representation $\rho$ of $\A$ has finite image, and is $\ft$-rational. Moreover,
   $$
  \FSexp(\A) \mid \ord (\rho(\ft)) \mid 12\cdot \FSexp(\A)\,.
  $$
  If one sets $\xi =\rho^{\lambda, \zeta}_\A$ for some $\lambda, \zeta \in \k$ satisfying \eqref{eq:ls}, then $\xi_x$, $x \in \k$ a 12-th root of unity, are all the modular representations of $\A$. \qed
\end{thm}

 In the proof of the above theorem, we have seen that $\Gamma(N)/(\Gamma(N)\cap \ker \xi)$ is always a finite cyclic group.
 However, there exist many linear characters $\a :\Gamma(N) \to \BC^\times$ with finite images whose kernel are \emph{noncongruence} subgroups of $\SL$ (cf. \cite{KL}). So, it is still unclear whether there always exists a modular representation of a modular tensor category whose kernel is a congruence subgroup. In view of Theorem \ref{t:cong1}, this is true when the modular category is the center of a spherical fusion category.

\section{GFS indicators and integers in the open string}\label{s:conjecture}
A family of scalars $Y_{ab}^c$ indexed by the primary fields $a,b,c$ of a RCFT was introduced by Pradisi, Sagnotti and Stanev. It is conjectured in \cite{PSS} that
\begin{equation}\label{eq:conj1}
Y_{ab}^c \in \BZ\,.
\end{equation}
Borisov, Halpern and Schweigert also considered these scalars, and they conjecture in \cite{BHS} that
\begin{equation}\label{eq:conj2}
\sum_{d }\frac{s_{ad}^2 s_{bd}^{} s_{\ol cd}^{}}{s_{0d}^2} \pm Y_{ab}^c \ge 0
\end{equation}
for all primary fields $a,b,c$, where $[s_{ab}]$ denotes the $S$-matrix of the RCFT. By considering the Galois group actions, Gannon has shown these two conjectures under the assumption that the $T$-matrix of the RCFT has odd order \cite{Gan}. In this section, we use the GFS indicators to prove these two conjectures for all modular categories.

Let $\A$ be a modular category with the set of simple objects $\{U_i\mid i \in \Pi\}$, and  the $S$ and $T$ matrices $[S_{ij}]_\Pi$ and $[\delta_{ij}\w_i]_\Pi$. Recall from Section \ref{s:modular} that $\{\bU_{ij} =(U_i \o U_j, \sigma_{U_i \o U_j})\mid i,j \in \Pi\}$ forms a complete set of simple objects of $Z(\A)$.
In the remainder of this section, we consider the normalizations $s$, $t$ of $S$ and $T$ respectively:
 \begin{equation}\label{eq:norms}
s:=\frac{1}{\lambda}\, S, \quad t:=\frac{1}{\zeta} T
\end{equation}
where $\lambda, \zeta \in \k$ satisfy \eqref{eq:ls}. In particular, $\lambda^2 =\dim \A$. The assignment $\fs \mapsto s, \, \ft \mapsto t$ defines an ordinary representation of $\SL$.

The fusion coefficients $N_{ab}^c =\dim \A(U_c, U_a \o U_b)$ of $\A$ and  $s$ are related by Verlinde's formula (cf. \cite{BaKi}):
\begin{equation}\label{eq:verlinde}
N_{ab}^c = \sum_{d \in \Pi}\frac{s_{ad} s_{bd}  s_{\ol c d}}{s_{0d}} \,.
\end{equation}
Defining the matrix $N_a \in M_\Pi(\k)$ by $(N_a)_{bc} = N_{ab}^c$, the assignment $\KK_0(\A) \to M_\Pi(\BZ); [U_a] \mapsto N_a$ is the regular representation of $\KK_0(\A)$ in matrix form. The Verlinde formula \eqref{eq:verlinde} can also be rewritten in matrix form as
\begin{equation}\label{eq:matrix_verlinde}
N_a = s D_a s\inv
\end{equation}
where $D_a$ is the diagonal matrix $[\delta_{ij} \frac{s_{aj}}{s_{0j}}]_\Pi$.

\begin{defn}
Let $s$ be the normalized  $S$-matrix of the modular category $\A$ described in \eqref{eq:norms}. For $J, K \in \GL(\Pi,\k)$, we define $Y_a(J,K) \in M_\Pi(\k)$  with the $(b,c)$-entry given by
$$
Y_{ab}^c(J,K) := \sum_{d \in \Pi}\frac{s_{ad} Q_{bd} \ol Q_{cd}}{s_{0d}},
$$
where $Q=JsKsJ=[Q_{ij}]_\Pi$ and $Q\inv=[\ol Q_{ij}]_\Pi$.
\end{defn}
It is worth noting that $Y_a(J,K)$ is independent of the choice of $\lambda$ used in the normalization \eqref{eq:norms} of $s$. Moreover, for any non-zero scalars $x,y \in \k$,
$$
Y_{ab}^c(xJ,yK) = Y_{ab}^c(J,K)\,.
$$

For any $\w \in \k$ and positive integer $m$, we write $\w^{1/m}$ for an $m$-th root of $\w$. Similarly, for any diagonal matrix $D \in M_\Pi(\k)$, $D^{1/m}$ abbreviates a diagonal matrix in $M_\Pi(\k)$ which satisfies the equation $\left(D^{1/m}\right)^m=D$.
For any $m$-th root $t^{1/m}$ of $t$, there exists an $m$-th root $T^{1/m}$ of $T$ which is a scalar multiple of $t^{1/m}$.
For these,
$$
Y_a(J, t^m) = Y_a(J, T^m), \quad Y_a(t^{1/m}, K) = Y_a(T^{1/m}, K) \quad\text{and}\quad Y_a(t^{1/m}, t^m) = Y_a(T^{1/m}, T^m) .
$$
In particular,  $Y_{ab}^c(t^{1/2}, t^2)$ are the scalars $Y_{ab}^c$ considered in \cite{PSS} and \cite{BHS}. The following lemma suggests a relation between these scalars $Y_{ab}^c$ and the GFS indicators via Proposition \ref{p:Bantay}.

\begin{lem} \label{l:rep}
Let $J,K \in \GL(\Pi,\k)$ such that $K$ is symmetric and J is a diagonal matrix of the form $[\delta_{ij} \eta_i]_\Pi$. Then, for $a \in \Pi$,
$$
Y_a(J,K) = (JsK)N_a (JsK)\inv.
$$
In particular, the assignment $\KK_0(\A) \to M_\Pi(\k)$; $[U_a] \mapsto Y_a(J,K)$ defines a matrix representation of $\KK_0(\A)$.

Moreover, for any positive integer $m$ and $a,b,c \in \Pi$,
$$
Y_{ab}^c(J,T^m) = \frac{\eta_b}{\eta_c}\nu_{m,1}^{\bU_{\ol b c}}(U_a)\,.
$$
\end{lem}
\begin{proof}
 Since $J, K$ are symmetric, and so is $Q=J s K s J$. Let $D_a$ be the diagonal matrix $[\delta_{ij} \frac{s_{ai}}{s_{0i}}]_\Pi$. Then we have
 \begin{multline*}
   Y_a(J,K) = Q D_a Q\inv =  J s K s J D_a J\inv s\inv (J s K)\inv  = J s K s  D_a s\inv (J s K)\inv =  J s K N_a (J s K)\inv.
 \end{multline*}
Here, the third equality follows from the Verlinde formula \eqref{eq:matrix_verlinde}.
In particular,
 $$
  Y_{ab}^c(J, T^m) = \frac{\eta_b}{\eta_c} \sum_{k, l}N_{ak}^{l}\frac{\w_k^m} {\w_l^{m}}
   s_{b k} s_{\ol c l}
   =\frac{\eta_b}{\eta_c} \sum_{k, l}N_{k  l}^{a}
   s_{\ol b k} s_{\ol c l} \frac{\w_k^m} {\w_l^{m}} =\frac{\eta_b}{\eta_c} \nu_{m,1}^{\bU_{\ol b c}}(U_a)\,.
 $$
 Here the last equation is an immediate consequence of Proposition \ref{p:Bantay}.
\end{proof}

\begin{thm} \label{t:solu}
  For any positive integer $m$ and $a,b,c \in \Pi$, $Y_{ab}^c(T^{1/m}, T^m)$ is an algebraic integer in $\BQ_m$ and
  \begin{equation}\label{eq:c2}
   \sum_{d \in \Pi} \frac{s_{ad}^m s_{bd}^{} s_{\ol c d}^{}}{s_{0d}^m} =\dim \A(U_c, U_a^{\o m} \o U_b) \ge |Y_{ab}^c(T^{1/m}, T^m)|
  \end{equation}
  where the inequality is considered in an embedding in $\BQ_m$ into $\BC$. In particular,
  $$
  Y_{ab}^c(T^{1/2}, T^2) \in \BZ \quad \text{and}\quad \sum_{d \in \Pi} \frac{s_{ad}^2 s_{bd}^{} s_{\ol c d}^{}}{s_{0d}^2} \pm Y_{ab}^c(T^{1/2}, T^2) \ge 0\,.
  $$
   If, in addition, $m$ is relatively prime to the Frobenius-Schur exponent $N$ of $\A$, then there exists an $m$-th root $T^{1/m} \in M_\Pi(\BQ_N)$, and
  $Y_{ab}^c(T^{1/m}, T^m)$ is a rational integer for all $T^{1/m} \in M_\Pi(\BQ_N)$.
\end{thm}
\begin{proof}
 In view of Lemma \ref{Emm} and \eqref{eq:centerribbon}, $\left(E_{\bU_{\ol b c}, U_a}^{(m,1)}\right)^m = \frac{\w_b\inv}{\w_c\inv}\id$ as operators on $\A(U_{\ol b}\o U_c, U_a^{\o m})$. If we set
  $$
  \tilde E = \frac{\w_b^{1/m}}{\w_c^{1/m}}E_{\bU_{\ol b c}, U_a}^{(m,1)},
  $$
  then $\tilde{E}^m = \id$ and hence $\Tr(\tilde E)$ is an algebraic integer in $\BQ_m$. Note that
  $$
  \Tr(\tilde E) = \frac{\w_b^{1/m}}{\w_c^{1/m}}\Tr\left(E_{\bU_{\ol b c}, U_a}^{(m,1)}\right) = \frac{\w_b^{1/m}}{\w_c^{1/m}} \nu_{m,1}^{\bU_{\ol b c}}(U_a)= Y_{ab}^c(T^{1/m}, T^m),
  $$
  where the last equality follows from Lemma \ref{l:rep}. Therefore $Y_{ab}^c(T^{1/m}, T^m)$ is an algebraic integer in $\BQ_m$ for all positive integers $m$. Since $\tilde E$ is a finite order $\k$-linear operator on $\A(U_{\ol b}\o U_c, U_a^{\o m})$, if one identifies $\BQ_m$ with a subfield of $\BC$, then
  $$
  |Y_{ab}^c(T^{1/m}, T^m)| =|\Tr(\tilde E)| \le \dim\, \A(U_{\ol b}\o U_c, U_a^{\o m})=\dim\, \A(U_c, U_a^{\o m}\o U_b)\,.
  $$

   By \eqref{eq:matrix_verlinde}, $N_a^m = s D_a^m s\inv$. Note that the $(b,c)$-entries of $N_a^m$ and $s D_a^m s\inv$ are respectively given by
   $$\dim \left(\A(U_c, U_a^{\o m} \o U_b)\right) \quad \text{and}\quad \sum_{d \in \Pi} \frac{s_{ad}^m s_{bd}^{} s_{\ol c d}^{}}{s_{0d}^m}.$$
   Thus, the first equality of \eqref{eq:c2} follows.

  Let $\varsigma_N\in \k$ be a primitive $N$-th root of unity. Then $\w_a$ is a power of $\varsigma_N$ for any $a \in \Pi$. If $m,N$ are relatively prime, then  there is an $m$-th root $\w_a^{1/m} \in \BQ_N$ and hence $T$ has a diagonal $m$-th root $T^{1/m}\in M_\Pi(\BQ_N)$. By Remark  \ref{r:inQN}, $\nu_{m,1}^{\bU_{\ol b c}}(U_a)$ is an algebraic integer in $\BQ_N$, and so is $Y_{ab}^c(T^{1/m}, T^m)$. Therefore, $Y_{ab}^c(T^{1/m}, T^m)$ is an algebraic integer in $\BQ_N \cap \BQ_m$. Since $(m,N)=1$, $\BQ_N \cap \BQ_m =\BQ$ and so $Y_{ab}^c(T^{1/m}, T^m)$ is a rational integer for any $a,b,c \in \Pi$.
\end{proof}
\begin{remark}
The specialization $m=2$ of Theorem \ref{t:solu} implies the conjecture of Pradisi-Sagnotti-Stanev  \eqref{eq:conj1} and the conjecture of Borisov-Halpern-Schweigert \eqref{eq:conj2}. As a consequence of Lemma \ref{l:rep}, for $m=2$ or $(m,N)=1$ with $T^{1/m} \in M_\Pi(\BQ_N)$, the assignment $\KK_0(\A) \to M_\Pi(\BZ)$; $[U_a] \mapsto Y_a(T^{1/m}, T^m)$ defines an integral representation of $\KK_0(\A)$.
\end{remark}

\section{Generalized Frobenius-Schur Endomorphisms}\label{s:endom}
It has been shown in \cite{NS05} that the Frobenius-Schur
indicators of an object $V$ in a pivotal fusion category $\CC$
over $\k$ are the pivotal traces of certain endomorphisms, called
the Frobenius-Schur (FS) endomorphisms. In this section, we
introduce the definition of a generalized Frobenius-Schur (GFS)
endomorphism $\FS_{V, z}^{(m,l)}$ for a pair $(m,l)$ of positive
integers, an object $V \in \CC$, and  a natural endomorphism  $z$
of the identity functor of $Z(\CC)$. These GFS endomorphisms
reduce to the FS endomorphisms defined in \cite{NS05} when $z$ is
the projection onto the trivial component. For a simple object
$\bX \in Z(\CC)$, we show that the GFS indicator
$\nu_{m,l}^\bX(V)$ is the left pivotal trace of $\FS_{V,
z_\bX/d_\ell(\bX)}^{(m,l)}$, where $z_\bX$ is the natural
projection onto the isotypic component of $\bX$. Moreover, if $(m,l)$ is a
pair of relatively prime integers, the GFS endomorphism
$\FS_{V,z}^{(m,l)}$ is natural in $V$.

This implies once again the additivity property
$$
\nu_{m,l}^\bX(U \oplus V) =\nu_{m,l}^\bX(U) +\nu_{m,l}^\bX(V)
$$
for a simple object $\bX \in Z(\CC)$ and a pair $(m,l)$ of
relatively prime positive integers proved already in corollary \ref{c:additivity} above when $\CC$ is spherical.

Let $\CC$ be a  pivotal fusion category over $\k$,  $\DD=Z(\CC)$, and $F: \DD\to \CC$ the natural forgetful functor. Note that $F$ maps the morphisms of $\DD$ injectively to the morphisms of $\CC$ and the pivotal trace of a morphism $f$ in $\DD$ is identical to the pivotal trace of $F(f)$ in $\CC$. Therefore, we may simply use the same notations for a morphism (or an object) in $\DD$ and its images in $\CC$ under $F$.

Now, we consider the two-sided adjoint $K$ to the forgetful functor $F$ with adjunction isomorphisms arranged as in \eqref{arrangement}. For $W \in \CC$, we define
$$
\u_W :=\ol \Psi\inv_{W, K(W)}(\id_{K(W)}): W \to K(W), \quad\text{and}\quad
\c_W:=\Psi_{K(W), W}(\id_{K(W)}): K(W) \to W\,.
$$

Then
$$
\u_W \circ g = K(g) \circ  \u_V \quad \text{and} \quad \c_W \circ K(g) = g \circ \c_V
$$
for all $g \in \CC(V, W)$.

Let  $\bX \in \DD$, and let $\{p_\a\}_\a$ be a  basis for $\CC(W,X)$, and  $\{q_\a\}_\a$ its dual basis for $\CC(X,W)$
with respect to the pairing $(\cdot, \cdot)_\ell$.
Set $P_\a = \ol \Psi(p_\a)$ and $Q_\a = \Psi\inv(q_\a)$.
Then

$$
p_\a = P_\a\circ  \u\,,\quad q_\a = \c \circ Q_\a
\quad \text{and}\quad (P_\a, Q_{\a'})_\ell=\delta_{\a, \a'}\,.
$$
Therefore,
$\{P_\a\}_\a$ and $\{Q_\a\}_\a$  are dual bases for $\DD(K(W),\bX)$ and
$\DD(\bX, K(W))$ respectively. If $\bX$ is simple, then
$$
z_\bX = d_\ell(\bX)\sum_\a Q_\a \circ P_\a
$$
is the natural projection of $K(W)$ onto its isotypic component of $\bX$. Note that $z_\bX/d_\ell(\bX)$
is a natural endomorphism of the identity functor on $\DD$.

Assume $\CC$ strict and set $W=V^m$ for some $V \in \CC$ and positive integer $m$. Then
\begin{equation}\label{eq:FSE1}
\begin{aligned}
 \nu_{m,1}^\bX(V) & = \sum_\a \,
{\def\objectstyle{\scriptstyle}
\xy (0,2)="ctext", "ctext"*{q_\a},
{"ctext"+(-4,-2)*{}; "ctext"+(4,2)*{} **\frm{-}},
{"ctext"+(2,-2);"ctext"+(2,-8)**\dir{-}},
{\vcross~{"ctext"+(-7,8)="x0"}{"ctext"+(-1,8)="v0"}{"ctext"+(-7,2)="v2"}{"ctext"+(-1,2)}},
{"v0"; "ctext"+(6,8)="v1" **\crv{"v0"+(0,3)& "v1"+(0,3)}},
{"v1"; "v1"+(0, -16) **\dir{-}?(.6)+(2,0)*{V}},
{"v2"*{};"ctext"+(-7,-2)="v3" **\dir{-}},
{"v3"; "ctext"+(-2,-2)="y" **\crv{"v3"+(0,-4)& "y"+(0,-4)}},
"ctext"+(4,-10)="ctext1"*{p_a},
{"ctext1"+(-4,-2)*{}; "ctext1"+(4,2)*{} **\frm{-}},
{"ctext1"+(0,-2)="x1"; "ctext1"+(-15,-2)="x2"**\crv{"x1"+(0,-6)& "x2"+(0,-6)}?(.3)+(4,0)*{X}},
{"x0"; "x0"+(-4,0)="x3" **\crv{"x0"+(0,2)& "x3"+(0,2)}},
{"x2"; "x3" **\dir{-}},
\endxy}  \,=\,\sum_\a
{\def\objectstyle{\scriptstyle}
\xy (0,-3)="ctext", "ctext"*{q_\a},
{"ctext"+(-4,-2)*{}; "ctext"+(3,2)*{} **\frm{-}},
{\vcross~{"ctext"+(-6,6)="x1"}{"ctext"+(-2,6)="x2"}{"ctext"+(-6,2)="y1"}{"ctext"+(-2,2)="y2"}},
{"y1"; "y1"+(0,-4)="z1"**\dir{-}},
{"z1"; "z1"+(4,0)="z2"**\crv{"z1"+(0,-3)&"z2"+(0,-3)}},
{"x1";"x1"+(-4,0)="x0"**\crv{"x1"+(0,3)&"x0"+(0,3)}},
{"x0"; "x0"+(0,-8)="z0"**\dir{-}},
{"z2"+(3,0)="z3"; "z3"+(2,6)="w"**\crv{"z3"+(1,-5)& "w"+(6,-6)}},
{"w"; "w"+(2,6)="a1" **\crv{"w"+(-6,6) & "a1"+(-1,5)}},
{"a1"+(-7,0)="a0"; "a1"+(4,0)="a2" **\crv{"a0"+(0,5)& "a2"+(0,5)}?(.7)+(2,1)*{V}},
{"a0"; "x2" **\dir{-}},
{"a1"+(2,-2)="ctext1"*{p_\a}},
{"ctext1"+(-4,-2)*{}; "ctext1"+(4,2)*{} **\frm{-}},
{"z0"; "z3"+(7,0)="z4"**\crv{"z0"+(0,-8)&"z4"+(0,-8)}},
{"ctext1"+(1,-2); "z4" **\dir{-}?(.5)+(2,0)*{X}},
\endxy}
\,=\,\sum_\a\,
{\def\objectstyle{\scriptstyle}
\xy (0,-10)="ctext", "ctext"*{q_\a},
{"ctext"+(-4,-2)*{}; "ctext"+(3,2)*{} **\frm{-}},
{\vcross~{"ctext"+(-6,6)="x1"}{"ctext"+(-2,6)="x2"}{"ctext"+(-6,2)="y1"}{"ctext"+(-2,2)="y2"}},
{"y1"; "y1"+(0,-4)="z1"**\dir{-}},
{"z1"; "z1"+(4,0)="z2"**\crv{"z1"+(0,-3)&"z2"+(0,-3)}},
{\vcrossneg~{"x1"+(-2,3)="b0"}{"b0"+(2,0)="b1"}{"x1"+(-2,0)="x0"}{"x1"}},
{"b0"; "b0"+(-3,0)="b"**\crv{"b0"+(0,2)& "b"+(0,2)}},
{"x0"; "x0"+(-3,0)="x"**\crv{"x0"+(0,-2)& "x"+(0,-2)}},
{"b"; "x"**\dir{-}},
{"b1"+(4,0)="b2"; "x2" **\dir{-}},
\vcrossneg~{"b1"+(0,4)="a1"}{"b2"+(0,4)="a2"}{"b1"}{"b2"},
\vcrossneg~{"a2"+(0,5)="c2"}{"a2"+(7,5)="c3"}{"a2"}{"a2"+(7,0)="a3"},
{"a3"; "a3"+(0,-15)="z4"**\dir{-}},
{"z2"+(3,0)="z3"; "z4"**\crv{"z3"+(0,-3)& "z4"+(0,-3)}},
{"c3"+(0,2)="ctext1"*{p_\a}},
{"ctext1"+(-4,-2)*{}; "ctext1"+(3,2)*{} **\frm{-}},
{"a1";"a1"+(0,9)="d1"**\dir{-}},
{"c2";"c2"+(0,4)="d2"**\dir{-}},
{"d2";"d2"+(5,0)="d3"**\crv{"d2"+(0,3)&"d3"+(0,3)}},
{"d1";"d1"+(12,0)="d4"**\crv{"d1"+(0,6)&"d4"+(0,6)}},
\endxy}\\\\
&=\, \sum_\a\,
{\def\objectstyle{\scriptstyle}
\xy (0,-8)="ctext", "ctext"*{q_\a},
{"ctext"+(-4,-2)*{}; "ctext"+(3,2)*{} **\frm{-}},
{"ctext"+(-7,-2)="z1"; "z1"+(5,0)="z2"**\crv{"z1"+(0,-3)&"z2"+(0,-3)}},
{"z1";"z1"+(0,19)="d1"**\dir{-}},
{"z2"+(0,10)="a2"},
{"a2"+(0,-2)="ctext1"*{\theta\inv}},
{"ctext1"+(-3,-2)*{}; "ctext1"+(2.5,2)*{} **\frm{-}},
{"ctext1"+(0,-2); "ctext1"+(0,-4) **\dir{-}},
\vcrossneg~{"a2"+(0,5)="c2"}{"a2"+(7,5)="c3"}{"a2"}{"a2"+(7,0)="a3"},
{"z2"+(3,0)="z3"; "z3"+(4,0)="z4"**\crv{"z3"+(0,-3)& "z4"+(0,-3)}},
{"a3"; "z4"**\dir{-}},
{"c3"+(0,2)="ctext1"*{p_\a}},
{"ctext1"+(-4,-2)*{}; "ctext1"+(3,2)*{} **\frm{-}},
{"c2";"c2"+(0,4)="d2"**\dir{-}},
{"d2";"d2"+(5,0)="d3"**\crv{"d2"+(0,3)&"d3"+(0,3)}},
{"d1";"d1"+(13,0)="d4"**\crv{"d1"+(0,6)&"d4"+(0,6)}},
\endxy}
\,=\, \sum_\a \ptrl\left(
{\def\objectstyle{\scriptstyle}
\xy (0,-8)="ctext",
{"ctext"+(-2,7)="a2"},
{"a2"+(0,-1.5)="ctext1"*{\theta\inv}},
{"ctext1"+(-3,-1.5)*{}; "ctext1"+(2.5,1.5)*{} **\frm{-}},
{"ctext1"+(0,-3)="ctext2"*{Q_\a}},
{"ctext2"+(-3,-1.5)*{}; "ctext2"+(2.5,1.5)*{} **\frm{-}},
{"ctext2"+(0,-3)="ctext3"*{\c}},
{"ctext3"+(-3,-1.5)*{}; "ctext3"+(2.5,1.5)*{} **\frm{-}},
{"ctext3"+(-1.5,-1.5)="z2"; "z2"+(0,-5)*{V} **\dir{-}},
\vcrossneg~{"a2"+(0,5)="c2"}{"a2"+(7,5)="c3"}{"a2"}{"a2"+(7,0)="a3"},
{"z2"+(2.5,0)="z3"; "z3"+(6,0)="z4"**\crv{"z3"+(0,-3)& "z4"+(0,-3)}},
{"a3"; "z4"**\dir{-}},
{"c3"+(0,1.5)="ctext4"*{P_\a}},
{"ctext4"+(-3,-1.5)*{}; "ctext4"+(2.5,1.5)*{} **\frm{-}},
{"ctext4"+(0,3)="ctext5"*{\u}},
{"ctext5"+(-3,-1.5)*{}; "ctext5"+(2.5,1.5)*{} **\frm{-}},
{"c2";"c2"+(0,6)="d2"**\dir{-}},
{"d2";"d2"+(5.5,0)="d3"**\crv{"d2"+(0,3)&"d3"+(0,3)}},
{"ctext5"+(1,1.5)="d4"; "d4"+(0,5)*{V}**\dir{-}},
\endxy}\right)
\,=\,\sum_\a \ptrl\left(
{\def\objectstyle{\scriptstyle}
\xy (0,-8)="ctext",
{"ctext"+(-2,7)="a2"},
{"a2"+(0,-1.5)="ctext1"*{P_\a}},
{"ctext1"+(-3,-1.5)*{}; "ctext1"+(2.5,1.5)*{} **\frm{-}},
{"ctext1"+(0,-3)="ctext2"*{Q_\a}},
{"ctext2"+(-3,-1.5)*{}; "ctext2"+(2.5,1.5)*{} **\frm{-}},
{"ctext2"+(0,-3)="ctext3"*{\c}},
{"ctext3"+(-3,-1.5)*{}; "ctext3"+(2.5,1.5)*{} **\frm{-}},
{"ctext3"+(-1.5,-1.5)="z2"; "z2"+(0,-5)*{V} **\dir{-}},
\vcrossneg~{"a2"+(0,5)="c2"}{"a2"+(7,5)="c3"}{"a2"}{"a2"+(7,0)="a3"},
{"z2"+(2.5,0)="z3"; "z3"+(6,0)="z4"**\crv{"z3"+(0,-3)& "z4"+(0,-3)}},
{"a3"; "z4"**\dir{-}},
{"c3"+(0,1.5)="ctext4"*{\theta\inv}},
{"ctext4"+(-3,-1.5)*{}; "ctext4"+(2.5,1.5)*{} **\frm{-}},
{"ctext4"+(0,3)="ctext5"*{\u}},
{"ctext5"+(-3,-1.5)*{}; "ctext5"+(2.5,1.5)*{} **\frm{-}},
{"c2";"c2"+(0,6)="d2"**\dir{-}},
{"d2";"d2"+(5.5,0)="d3"**\crv{"d2"+(0,3)&"d3"+(0,3)}},
{"ctext5"+(1,1.5)="d4"; "d4"+(0,5)*{V}**\dir{-}},
\endxy} \right) \\ \\
&=\,   \frac{1}{d_\ell(\bX)} \ptrl\left(
{\def\objectstyle{\scriptstyle}
\xy (0,-9)="ctext",
{"ctext"+(-2,7)="a2"},
{"a2"+(0,-1.5)="ctext1"*{z_\bX}},
{"ctext1"+(-2,-1.5)*{}; "ctext1"+(2,1.5)*{} **\frm{-}},
{"ctext1"+(0,-1.5); "ctext1"+(0,-2.5)**\dir{-}},
{"ctext1"+(0,-4)="ctext3"*{\c}},
{"ctext3"+(-3,-1.5)*{}; "ctext3"+(2.5,1.5)*{} **\frm{-}},
{"ctext3"+(-1.5,-1.5)="z2"; "z2"+(0,-5)*{V} **\dir{-}},
\vcrossneg~{"a2"+(0,5)="c2"}{"a2"+(7,5)="c3"}{"a2"}{"a2"+(7,0)="a3"},
{"z2"+(2.5,0)="z3"; "z3"+(6,0)="z4"**\crv{"z3"+(0,-3)& "z4"+(0,-3)}},
{"a3"; "z4"**\dir{-}},
{"c3"+(0,1.5)="ctext4"*{\theta\inv}},
{"ctext4"+(-3,-1.5)*{}; "ctext4"+(2.5,1.5)*{} **\frm{-}},
{"ctext4"+(0,1.5); "ctext4"+(0,2.5)**\dir{-}},
{"ctext4"+(0,4)="ctext5"*{\u}},
{"ctext5"+(-3,-1.5)*{}; "ctext5"+(2.5,1.5)*{} **\frm{-}},
{"c2";"c2"+(0,7)="d2"**\dir{-}},
{"d2";"d2"+(5.5,0)="d3"**\crv{"d2"+(0,3)&"d3"+(0,3)}},
{"ctext5"+(1,1.5)="d4"; "d4"+(0,5)*{V}**\dir{-}},
\endxy} \right) =  \ptrl(\FS_{V, z_\bX/d_\ell(\bX)}^{(m)}),
\end{aligned}
\end{equation}
where we use the following
 \begin{defn}\nmlabel{Definition}{FSEdef}
 {\rm The $(m,1)$-st GFS endomorphism of $V$ associated
    to an endomorphism $z$ of the identity functor on $Z(\CC)$ is
    \[\FS^{(m)}_{V,z}\,=\,
      \def\objectstyle{\scriptstyle}
\xy (0,-9)="ctext",
{"ctext"+(-2,7)="a2"},
{"a2"+(0,-1.5)="ctext1"*{z}},
{"ctext1"+(-2,-1.5)*{}; "ctext1"+(2,1.5)*{} **\frm{-}},
{"ctext1"+(0,-1.5); "ctext1"+(0,-2.5)**\dir{-}},
{"ctext1"+(0,-4)="ctext3"*{\c}},
{"ctext3"+(-3,-1.5)*{}; "ctext3"+(2.5,1.5)*{} **\frm{-}},
{"ctext3"+(-1.5,-1.5)="z2"; "z2"+(0,-5)*{V} **\dir{-}},
\vcrossneg~{"a2"+(0,5)="c2"}{"a2"+(7,5)="c3"}{"a2"}{"a2"+(7,0)="a3"},
{"z2"+(2.5,0)="z3"; "z3"+(6,0)="z4"**\crv{"z3"+(0,-3)& "z4"+(0,-3)}},
{"a3"; "z4"**\dir{-}},
{"c3"+(0,1.5)="ctext4"*{\theta\inv}},
{"ctext4"+(-3,-1.5)*{}; "ctext4"+(2.5,1.5)*{} **\frm{-}},
{"ctext4"+(0,1.5); "ctext4"+(0,2.5)**\dir{-}},
{"ctext4"+(0,4)="ctext5"*{\u}},
{"ctext5"+(-3,-1.5)*{}; "ctext5"+(2.5,1.5)*{} **\frm{-}},
{"c2";"c2"+(0,7)="d2"**\dir{-}},
{"d2";"d2"+(5.5,0)="d3"**\crv{"d2"+(0,3)&"d3"+(0,3)}},
{"ctext5"+(1,1.5)="d4"; "d4"+(0,5)*{V}**\dir{-}},
\endxy\,.
\]
}
\end{defn}

\begin{prop}\label{p:FSE1}
  Each of the $(m,1)$-st GFS endomorphisms defines a natural
  endomorphism of the identity functor on $\CC$. In particular,
  $\nu_{m,1}^{\bX}(V)$ is additive in $V$ for any simple $\bX \in
  Z(\CC)$.
\end{prop}
\begin{proof}
  Consider $f\colon V\to W$ in $\CC$, and write
  \[f_k=\id_{W^{k-1}}\o f\o \id_{V^{m-k}}\colon W^{k-1}\o V^{m-k+1}\to W^k\o V^{m-k}\]
  for all $1\leq k\leq m$.  We have
\begin{equation}\label{eq:FSE2}
\def\objectstyle{\scriptstyle}
\xy (0,-8)="ctext",
{"ctext"+(-2,7)="a2"},
{"a2"+(0,-1.5)="ctext1"*{z}},
{"ctext1"+(-2,-1.5)*{}; "ctext1"+(2,1.5)*{} **\frm{-}},
{"ctext1"+(0,-1.5); "ctext1"+(0,-3)**\dir{-}},
{"ctext1"+(0,-4.5)="ctext2"*{\c}},
{"ctext2"+(-3,-1.5)*{}; "ctext2"+(2.5,1.5)*{} **\frm{-}},
{"ctext2"+(-1,-1.5); "ctext2"+(-1,-3)**\dir{-}},
{"ctext2"+(1,-1.5); "ctext2"+(1,-3)**\dir{-}},
{"ctext2"+(0,-4.5)="ctext3"*{f_k}},
{"ctext3"+(-3,-1.5)*{}; "ctext3"+(3,1.5)*{} **\frm{-}},
{"ctext3"+(-1.5,-1.5)="z2"; "z2"+(0,-5)*{W} **\dir{-}},
\vcrossneg~{"a2"+(0,5)="c2"}{"a2"+(7,5)="c3"}{"a2"}{"a2"+(7,0)="a3"},
{"z2"+(2.5,0)="z3"; "z3"+(6,0)="z4"**\crv{"z3"+(0,-3)& "z4"+(0,-3)}},
{"a3"; "z4"**\dir{-}},
{"c3"+(0,1.5)="ctext4"*{\theta\inv}},
{"ctext4"+(-3,-1.5)*{}; "ctext4"+(2.5,1.5)*{} **\frm{-}},
{"ctext4"+(0,1.5); "ctext4"+(0,2.5) **\dir{-}},
{"ctext4"+(0,4)="ctext5"*{\u}},
{"ctext5"+(-3,-1.5)*{}; "ctext5"+(2.5,1.5)*{} **\frm{-}},
{"c2";"c2"+(0,7)="d2"**\dir{-}},
{"d2";"d2"+(5.5,0)="d3"**\crv{"d2"+(0,3)&"d3"+(0,3)}},
{"ctext5"+(1,1.5)="d4"; "d4"+(0,5)*{V}**\dir{-}},
\endxy  \,=\,
\xy (0,-8)="ctext",
{"ctext"+(-2,7)="a2"},
{"a2"+(0,-1.5)="ctext1"*{z}},
{"ctext1"+(-2,-1.5)*{}; "ctext1"+(2,1.5)*{} **\frm{-}},
{"ctext1"+(0,-1.5); "ctext1"+(0,-3)**\dir{-}},
{"ctext1"+(0,-4.5)="ctext2"*{K(f_k)}},
{"ctext2"+(-4,-1.5)*{}; "ctext2"+(4,1.5)*{} **\frm{-}},
{"ctext2"+(0,-1.5); "ctext2"+(0,-3)**\dir{-}},
{"ctext2"+(0,-4.5)="ctext3"*{\c}},
{"ctext3"+(-3,-1.5)*{}; "ctext3"+(3,1.5)*{} **\frm{-}},
{"ctext3"+(-1.5,-1.5)="z2"; "z2"+(0,-5)*{W} **\dir{-}},
\vcrossneg~{"a2"+(0,5)="c2"}{"a2"+(7,5)="c3"}{"a2"}{"a2"+(7,0)="a3"},
{"z2"+(2.5,0)="z3"; "z3"+(6,0)="z4"**\crv{"z3"+(0,-3)& "z4"+(0,-3)}},
{"a3"; "z4"**\dir{-}},
{"c3"+(0,1.5)="ctext4"*{\theta\inv}},
{"ctext4"+(-3,-1.5)*{}; "ctext4"+(2.5,1.5)*{} **\frm{-}},
{"ctext4"+(0,1.5); "ctext4"+(0,2.5) **\dir{-}},
{"ctext4"+(0,4)="ctext5"*{\u}},
{"ctext5"+(-3,-1.5)*{}; "ctext5"+(2.5,1.5)*{} **\frm{-}},
{"c2";"c2"+(0,7)="d2"**\dir{-}},
{"d2";"d2"+(5.5,0)="d3"**\crv{"d2"+(0,3)&"d3"+(0,3)}},
{"ctext5"+(1,1.5)="d4"; "d4"+(0,5)*{V}**\dir{-}},
\endxy  \,=\,
\xy (0,-13)="ctext",
{"ctext"+(-2,7)="a2"},
{"a2"+(0,-1.5)="ctext1"*{z}},
{"ctext1"+(-2,-1.5); "ctext1"+(2,1.5) **\frm{-}},
{"ctext1"+(0,-1.5); "ctext1"+(0,-2.5)**\dir{-}},
{"ctext1"+(0,-4)="ctext2"*{\c}},
{"ctext2"+(-3,-1.5); "ctext2"+(2.5,1.5) **\frm{-}},
{"ctext2"+(-1.5,-1.5)="z2"; "z2"+(0,-5)*{W} **\dir{-}},
\vcrossneg~{"a2"+(0,5)="c2"}{"a2"+(7,5)="c3"}{"a2"}{"a2"+(7,0)="a3"},
{"z2"+(2.5,0)="z3"; "z3"+(6,0)="z4"**\crv{"z3"+(0,-3)& "z4"+(0,-3)}},
{"a3"; "z4"**\dir{-}},
{"c3"+(0,1.5)="ctext3"*{\theta\inv}},
{"ctext3"+(-3,-1.5); "ctext3"+(2.5,1.5) **\frm{-}},
{"ctext3"+(0,1.5);"ctext3"+(0,3)**\dir{-}},
{"ctext3"+(0,4.5)="ctext4"*{K(f_k)}},
{"ctext4"+(-4,-1.5); "ctext4"+(4,1.5) **\frm{-}},
{"ctext4"+(0,1.5);"ctext4"+(0,3)**\dir{-}},
{"ctext4"+(0,4.5)="ctext5"*{\u}},
{"ctext5"+(-3,-1.5)*{}; "ctext5"+(3,1.5)*{} **\frm{-}},
{"ctext5"+(1,1.5)="d4"; "d4"+(0,5)*{V}**\dir{-}},
{"ctext5"+(-1.5,1.5)="d3"; "d3"+(-5.5,0)="d2"**\crv{"d3"+(0,3)&"d2"+(0,3)}},
{"c2"; "d2"**\dir{-}},
\endxy\,=\,
\xy (0,-13)="ctext",
{"ctext"+(-2,7)="a2"},
{"a2"+(0,-1.5)="ctext1"*{z}},
{"ctext1"+(-2,-1.5)*{}; "ctext1"+(2,1.5)*{} **\frm{-}},
{"ctext1"+(0,-1.5); "ctext1"+(0,-2.5)**\dir{-}},
{"ctext1"+(0,-4)="ctext2"*{\c}},
{"ctext2"+(-3,-1.5)*{}; "ctext2"+(2.5,1.5)*{} **\frm{-}},
{"ctext2"+(-1.5,-1.5)="z2"; "z2"+(0,-5)*{W} **\dir{-}},
\vcrossneg~{"a2"+(0,5)="c2"}{"a2"+(7,5)="c3"}{"a2"}{"a2"+(7,0)="a3"},
{"z2"+(2.5,0)="z3"; "z3"+(6,0)="z4"**\crv{"z3"+(0,-3)& "z4"+(0,-3)}},
{"a3"; "z4"**\dir{-}},
{"c3"+(0,1.5)="ctext3"*{\theta\inv}},
{"ctext3"+(-3,-1.5)*{}; "ctext3"+(2.5,1.5)*{} **\frm{-}},
{"ctext3"+(0,1.5);"ctext3"+(0,3)**\dir{-}},
{"ctext3"+(0,4.5)="ctext4"*{\u}},
{"ctext4"+(-3,-1.5)*{}; "ctext4"+(3,1.5)*{} **\frm{-}},
{"ctext4"+(-1,1.5);"ctext4"+(-1,3)**\dir{-}},
{"ctext4"+(1,1.5);"ctext4"+(1,3)**\dir{-}},
{"ctext4"+(0,4.5)="ctext5"*{f_k}},
{"ctext5"+(-3,-1.5)*{}; "ctext5"+(3,1.5)*{} **\frm{-}},
{"ctext5"+(1,1.5)="d4"; "d4"+(0,5)*{V}**\dir{-}},
{"ctext5"+(-1.5,1.5)="d3"; "d3"+(-5.5,0)="d2"**\crv{"d3"+(0,3)&"d2"+(0,3)}},
{"c2"; "d2"**\dir{-}},
\endxy
\end{equation}
and when $k< m$ we can continue
$$
\def\objectstyle{\scriptstyle}
\xy (0,-13)="ctext",
{"ctext"+(-2,7)="a2"},
{"a2"+(0,-1.5)="ctext1"*{z}},
{"ctext1"+(-2,-1.5)*{}; "ctext1"+(2,1.5)*{} **\frm{-}},
{"ctext1"+(0,-1.5); "ctext1"+(0,-2.5)**\dir{-}},
{"ctext1"+(0,-4)="ctext2"*{\c}},
{"ctext2"+(-3,-1.5)*{}; "ctext2"+(2.5,1.5)*{} **\frm{-}},
{"ctext2"+(-1.5,-1.5)="z2"; "z2"+(0,-5)*{W} **\dir{-}},
\vcrossneg~{"a2"+(0,5)="c2"}{"a2"+(7,5)="c3"}{"a2"}{"a2"+(7,0)="a3"},
{"z2"+(2.5,0)="z3"; "z3"+(6,0)="z4"**\crv{"z3"+(0,-3)& "z4"+(0,-3)}},
{"a3"; "z4"**\dir{-}},
{"c3"+(0,1.5)="ctext3"*{\theta\inv}},
{"ctext3"+(-3,-1.5)*{}; "ctext3"+(2.5,1.5)*{} **\frm{-}},
{"ctext3"+(0,1.5);"ctext3"+(0,3)**\dir{-}},
{"ctext3"+(0,4.5)="ctext4"*{\u}},
{"ctext4"+(-3,-1.5)*{}; "ctext4"+(3,1.5)*{} **\frm{-}},
{"ctext4"+(-1,1.5);"ctext4"+(-1,3)**\dir{-}},
{"ctext4"+(1,1.5);"ctext4"+(1,3)**\dir{-}},
{"ctext4"+(0,4.5)="ctext5"*{f_k}},
{"ctext5"+(-3,-1.5)*{}; "ctext5"+(3,1.5)*{} **\frm{-}},
{"ctext5"+(1,1.5)="d4"; "d4"+(0,5)*{V}**\dir{-}},
{"ctext5"+(-1.5,1.5)="d3"; "d3"+(-5.5,0)="d2"**\crv{"d3"+(0,3)&"d2"+(0,3)}},
{"c2"; "d2"**\dir{-}},
\endxy\,=\,
\xy (0,-13)="ctext",
{"ctext"+(-2,7)="a2"},
{"a2"+(0,-1.5)="ctext1"*{z}},
{"ctext1"+(-2,-1.5)*{}; "ctext1"+(2,1.5)*{} **\frm{-}},
{"ctext1"+(0,-1.5); "ctext1"+(0,-2.5)**\dir{-}},
{"ctext1"+(0,-4)="ctext2"*{\c}},
{"ctext2"+(-3,-1.5)*{}; "ctext2"+(2.5,1.5)*{} **\frm{-}},
{"ctext2"+(-1.5,-1.5)="z2"; "z2"+(0,-5)*{W} **\dir{-}},
\vcrossneg~{"a2"+(0,5)="c2"}{"a2"+(7,5)="c3"}{"a2"}{"a2"+(7,0)="a3"},
{"z2"+(2.5,0)="z3"; "z3"+(6,0)="z4"**\crv{"z3"+(0,-3)& "z4"+(0,-3)}},
{"a3"; "z4"**\dir{-}},
{"c3"+(0,1.5)="ctext3"*{\theta\inv}},
{"ctext3"+(-3,-1.5)*{}; "ctext3"+(2.5,1.5)*{} **\frm{-}},
{"ctext3"+(0,1.5);"ctext3"+(0,3)**\dir{-}},
{"ctext3"+(0,4.5)="ctext4"*{\u}},
{"ctext4"+(-3,-1.5)*{}; "ctext4"+(3,1.5)*{} **\frm{-}},
{"ctext4"+(-1.5,1.5);"ctext4"+(-1.5,3)**\dir{-}},
{"ctext4"+(1.5,1.5)="d4";"d4"+(0,9.5)*{V}**\dir{-}},
{"ctext4"+(-2 ,4.5)="ctext5"*{f_k}},
{"ctext5"+(-2,-1.5)*{}; "ctext5"+(2,1.5)*{} **\frm{-}},
{"ctext5"+(0,1.5)="d3"; "d3"+(-5,0)="d2"**\crv{"d3"+(0,3)&"d2"+(0,3)}},
{"c2"; "d2"**\dir{-}},
\endxy\,=\,
\xy (0,-8)="ctext",
{"ctext"+(-2,7)="a2"},
{"a2"+(0,-1.5)="ctext1"*{z}},
{"ctext1"+(-2,-1.5)*{}; "ctext1"+(2,1.5)*{} **\frm{-}},
{"ctext1"+(0,-1.5); "ctext1"+(0,-3)**\dir{-}},
{"ctext1"+(0,-4.5)="ctext2"*{\c}},
{"ctext2"+(-3,-1.5)*{}; "ctext2"+(2.5,1.5)*{} **\frm{-}},
{"ctext2"+(1.5,-1.5); "ctext2"+(1.5,-3)**\dir{-}},
{"ctext2"+(-1.5,-1.5); "ctext2"+(-1.5,-11)*{W}**\dir{-}},
{"ctext2"+(2,-4.5)="ctext3"*{f_k}},
{"ctext3"+(-2,-1.5)*{}; "ctext3"+(2,1.5)*{} **\frm{-}},
\vcrossneg~{"a2"+(0,5)="c2"}{"a2"+(7,5)="c3"}{"a2"}{"a2"+(7,0)="a3"},
{"ctext3"+(0,-1.5)="z3"; "z3"+(5,0)="z4"**\crv{"z3"+(0,-3)& "z4"+(0,-3)}},
{"a3"; "z4"**\dir{-}},
{"c3"+(0,1.5)="ctext4"*{\theta\inv}},
{"ctext4"+(-3,-1.5)*{}; "ctext4"+(2.5,1.5)*{} **\frm{-}},
{"ctext4"+(0,1.5); "ctext4"+(0,2.5) **\dir{-}},
{"ctext4"+(0,4)="ctext5"*{\u}},
{"ctext5"+(-3,-1.5)*{}; "ctext5"+(2.5,1.5)*{} **\frm{-}},
{"c2";"c2"+(0,7)="d2"**\dir{-}},
{"d2";"d2"+(5.5,0)="d3"**\crv{"d2"+(0,3)&"d3"+(0,3)}},
{"ctext5"+(1,1.5)="d4"; "d4"+(0,5)*{V}**\dir{-}},
\endxy \,=\,
\xy (0,-8)="ctext",
{"ctext"+(-2,7)="a2"},
{"a2"+(0,-1.5)="ctext1"*{z}},
{"ctext1"+(-2,-1.5)*{}; "ctext1"+(2,1.5)*{} **\frm{-}},
{"ctext1"+(0,-1.5); "ctext1"+(0,-3)**\dir{-}},
{"ctext1"+(0,-4.5)="ctext2"*{\c}},
{"ctext2"+(-3,-1.5)*{}; "ctext2"+(2.5,1.5)*{} **\frm{-}},
{"ctext2"+(-1,-1.5); "ctext2"+(-1,-3)**\dir{-}},
{"ctext2"+(1,-1.5); "ctext2"+(1,-3)**\dir{-}},
{"ctext2"+(0,-4.5)="ctext3"*{f_{k+1}}},
{"ctext3"+(-3.2,-1.5)*{}; "ctext3"+(3,1.5)*{} **\frm{-}},
{"ctext3"+(-1.5,-1.5)="z2"; "z2"+(0,-5)*{W} **\dir{-}},
\vcrossneg~{"a2"+(0,5)="c2"}{"a2"+(7,5)="c3"}{"a2"}{"a2"+(7,0)="a3"},
{"z2"+(2.5,0)="z3"; "z3"+(6,0)="z4"**\crv{"z3"+(0,-3)& "z4"+(0,-3)}},
{"a3"; "z4"**\dir{-}},
{"c3"+(0,1.5)="ctext4"*{\theta\inv}},
{"ctext4"+(-3,-1.5)*{}; "ctext4"+(2.5,1.5)*{} **\frm{-}},
{"ctext4"+(0,1.5); "ctext4"+(0,2.5) **\dir{-}},
{"ctext4"+(0,4)="ctext5"*{\u}},
{"ctext5"+(-3,-1.5)*{}; "ctext5"+(2.5,1.5)*{} **\frm{-}},
{"c2";"c2"+(0,7)="d2"**\dir{-}},
{"d2";"d2"+(5.5,0)="d3"**\crv{"d2"+(0,3)&"d3"+(0,3)}},
{"ctext5"+(1,1.5)="d4"; "d4"+(0,5)*{V}**\dir{-}},
\endxy\,.
$$
We conclude by induction that
$$
f \FS^{(m)}_{V, z} =
\def\objectstyle{\scriptstyle}
\xy (0,-8)="ctext",
{"ctext"+(-2,7)="a2"},
{"a2"+(0,-1.5)="ctext1"*{z}},
{"ctext1"+(-2,-1.5)*{}; "ctext1"+(2,1.5)*{} **\frm{-}},
{"ctext1"+(0,-1.5); "ctext1"+(0,-3)**\dir{-}},
{"ctext1"+(0,-4.5)="ctext2"*{\c}},
{"ctext2"+(-3,-1.5)*{}; "ctext2"+(2.5,1.5)*{} **\frm{-}},
{"ctext2"+(-2,-1.5); "ctext2"+(-2,-3)**\dir{-}},
{"ctext2"+(-1.5,-4.5)="ctext3"*{f}},
{"ctext3"+(-2,-1.5)*{}; "ctext3"+(2,1.5)*{} **\frm{-}},
{"ctext3"+(0,-1.5); "ctext3"+(0,-6.5)*{W}**\dir{-}},
\vcrossneg~{"a2"+(0,5)="c2"}{"a2"+(7,5)="c3"}{"a2"}{"a2"+(7,0)="a3"},
{"ctext2"+(1.5,-1.5)="z3"; "z3"+(5.5,0)="z4"**\crv{"z3"+(0,-3)& "z4"+(0,-3)}},
{"a3"; "z4"**\dir{-}},
{"c3"+(0,1.5)="ctext4"*{\theta\inv}},
{"ctext4"+(-3,-1.5)*{}; "ctext4"+(2.5,1.5)*{} **\frm{-}},
{"ctext4"+(0,1.5); "ctext4"+(0,2.5) **\dir{-}},
{"ctext4"+(0,4)="ctext5"*{\u}},
{"ctext5"+(-3,-1.5)*{}; "ctext5"+(2.5,1.5)*{} **\frm{-}},
{"c2";"c2"+(0,7)="d2"**\dir{-}},
{"d2";"d2"+(5.5,0)="d3"**\crv{"d2"+(0,3)&"d3"+(0,3)}},
{"ctext5"+(1,1.5)="d4"; "d4"+(0,5)*{V}**\dir{-}},
\endxy\,=\,
\xy (0,-8)="ctext",
{"ctext"+(-2,7)="a2"},
{"a2"+(0,-1.5)="ctext1"*{z}},
{"ctext1"+(-2,-1.5)*{}; "ctext1"+(2,1.5)*{} **\frm{-}},
{"ctext1"+(0,-1.5); "ctext1"+(0,-3)**\dir{-}},
{"ctext1"+(0,-4.5)="ctext2"*{\c}},
{"ctext2"+(-3,-1.5)*{}; "ctext2"+(2.5,1.5)*{} **\frm{-}},
{"ctext2"+(-1,-1.5); "ctext2"+(-1,-3)**\dir{-}},
{"ctext2"+(1,-1.5); "ctext2"+(1,-3)**\dir{-}},
{"ctext2"+(0,-4.5)="ctext3"*{f_1}},
{"ctext3"+(-3,-1.5)*{}; "ctext3"+(3,1.5)*{} **\frm{-}},
{"ctext3"+(-1.5,-1.5)="z2"; "z2"+(0,-5)*{W} **\dir{-}},
\vcrossneg~{"a2"+(0,5)="c2"}{"a2"+(7,5)="c3"}{"a2"}{"a2"+(7,0)="a3"},
{"z2"+(2.5,0)="z3"; "z3"+(6,0)="z4"**\crv{"z3"+(0,-3)& "z4"+(0,-3)}},
{"a3"; "z4"**\dir{-}},
{"c3"+(0,1.5)="ctext4"*{\theta\inv}},
{"ctext4"+(-3,-1.5)*{}; "ctext4"+(2.5,1.5)*{} **\frm{-}},
{"ctext4"+(0,1.5); "ctext4"+(0,2.5) **\dir{-}},
{"ctext4"+(0,4)="ctext5"*{\u}},
{"ctext5"+(-3,-1.5)*{}; "ctext5"+(2.5,1.5)*{} **\frm{-}},
{"c2";"c2"+(0,7)="d2"**\dir{-}},
{"d2";"d2"+(5.5,0)="d3"**\crv{"d2"+(0,3)&"d3"+(0,3)}},
{"ctext5"+(1,1.5)="d4"; "d4"+(0,5)*{V}**\dir{-}},
\endxy\,=\,\cdots\,=\,
\xy (0,-8)="ctext",
{"ctext"+(-2,7)="a2"},
{"a2"+(0,-1.5)="ctext1"*{z}},
{"ctext1"+(-2,-1.5)*{}; "ctext1"+(2,1.5)*{} **\frm{-}},
{"ctext1"+(0,-1.5); "ctext1"+(0,-3)**\dir{-}},
{"ctext1"+(0,-4.5)="ctext2"*{\c}},
{"ctext2"+(-3,-1.5)*{}; "ctext2"+(2.5,1.5)*{} **\frm{-}},
{"ctext2"+(-1,-1.5); "ctext2"+(-1,-3)**\dir{-}},
{"ctext2"+(1,-1.5); "ctext2"+(1,-3)**\dir{-}},
{"ctext2"+(0,-4.5)="ctext3"*{f_m}},
{"ctext3"+(-3,-1.5)*{}; "ctext3"+(3,1.5)*{} **\frm{-}},
{"ctext3"+(-1.5,-1.5)="z2"; "z2"+(0,-5)*{W} **\dir{-}},
\vcrossneg~{"a2"+(0,5)="c2"}{"a2"+(7,5)="c3"}{"a2"}{"a2"+(7,0)="a3"},
{"z2"+(2.5,0)="z3"; "z3"+(6,0)="z4"**\crv{"z3"+(0,-3)& "z4"+(0,-3)}},
{"a3"; "z4"**\dir{-}},
{"c3"+(0,1.5)="ctext4"*{\theta\inv}},
{"ctext4"+(-3,-1.5)*{}; "ctext4"+(2.5,1.5)*{} **\frm{-}},
{"ctext4"+(0,1.5); "ctext4"+(0,2.5) **\dir{-}},
{"ctext4"+(0,4)="ctext5"*{\u}},
{"ctext5"+(-3,-1.5)*{}; "ctext5"+(2.5,1.5)*{} **\frm{-}},
{"c2";"c2"+(0,7)="d2"**\dir{-}},
{"d2";"d2"+(5.5,0)="d3"**\crv{"d2"+(0,3)&"d3"+(0,3)}},
{"ctext5"+(1,1.5)="d4"; "d4"+(0,5)*{V}**\dir{-}},
\endxy\,=\,
\xy (0,-13)="ctext",
{"ctext"+(-2,7)="a2"},
{"a2"+(0,-1.5)="ctext1"*{z}},
{"ctext1"+(-2,-1.5)*{}; "ctext1"+(2,1.5)*{} **\frm{-}},
{"ctext1"+(0,-1.5); "ctext1"+(0,-2.5)**\dir{-}},
{"ctext1"+(0,-4)="ctext2"*{\c}},
{"ctext2"+(-3,-1.5)*{}; "ctext2"+(2.5,1.5)*{} **\frm{-}},
{"ctext2"+(-1.5,-1.5)="z2"; "z2"+(0,-5)*{W} **\dir{-}},
\vcrossneg~{"a2"+(0,5)="c2"}{"a2"+(7,5)="c3"}{"a2"}{"a2"+(7,0)="a3"},
{"z2"+(2.5,0)="z3"; "z3"+(6,0)="z4"**\crv{"z3"+(0,-3)& "z4"+(0,-3)}},
{"a3"; "z4"**\dir{-}},
{"c3"+(0,1.5)="ctext3"*{\theta\inv}},
{"ctext3"+(-3,-1.5)*{}; "ctext3"+(2.5,1.5)*{} **\frm{-}},
{"ctext3"+(0,1.5);"ctext3"+(0,3)**\dir{-}},
{"ctext3"+(0,4.5)="ctext4"*{\u}},
{"ctext4"+(-3,-1.5)*{}; "ctext4"+(3,1.5)*{} **\frm{-}},
{"ctext4"+(-1,1.5);"ctext4"+(-1,3)**\dir{-}},
{"ctext4"+(1,1.5);"ctext4"+(1,3)**\dir{-}},
{"ctext4"+(0,4.5)="ctext5"*{f_m}},
{"ctext5"+(-3,-1.5)*{}; "ctext5"+(3,1.5)*{} **\frm{-}},
{"ctext5"+(1,1.5)="d4"; "d4"+(0,5)*{V}**\dir{-}},
{"ctext5"+(-1.5,1.5)="d3"; "d3"+(-5.5,0)="d2"**\crv{"d3"+(0,3)&"d2"+(0,3)}},
{"c2"; "d2"**\dir{-}},
\endxy\,=\,
\xy (0,-13)="ctext",
{"ctext"+(-2,7)="a2"},
{"a2"+(0,-1.5)="ctext1"*{z}},
{"ctext1"+(-2,-1.5)*{}; "ctext1"+(2,1.5)*{} **\frm{-}},
{"ctext1"+(0,-1.5); "ctext1"+(0,-2.5)**\dir{-}},
{"ctext1"+(0,-4)="ctext2"*{\c}},
{"ctext2"+(-3,-1.5)*{}; "ctext2"+(2.5,1.5)*{} **\frm{-}},
{"ctext2"+(-1.5,-1.5)="z2"; "z2"+(0,-5)*{W} **\dir{-}},
\vcrossneg~{"a2"+(0,5)="c2"}{"a2"+(7,5)="c3"}{"a2"}{"a2"+(7,0)="a3"},
{"z2"+(2.5,0)="z3"; "z3"+(6,0)="z4"**\crv{"z3"+(0,-3)& "z4"+(0,-3)}},
{"a3"; "z4"**\dir{-}},
{"c3"+(0,1.5)="ctext3"*{\theta\inv}},
{"ctext3"+(-3,-1.5)*{}; "ctext3"+(2.5,1.5)*{} **\frm{-}},
{"ctext3"+(0,1.5);"ctext3"+(0,3)**\dir{-}},
{"ctext3"+(0,4.5)="ctext4"*{\u}},
{"ctext4"+(-3,-1.5)*{}; "ctext4"+(3,1.5)*{} **\frm{-}},
{"ctext4"+(1.5,1.5);"ctext4"+(1.5,3)**\dir{-}},
{"ctext4"+(-1.5,1.5)="d3"; "d3"+(-5.5,0)="d2"**\crv{"d3"+(0,3)&"d2"+(0,3)}},
{"ctext4"+(2 ,4.5)="ctext5"*{f}},
{"ctext5"+(-2,-1.5)*{}; "ctext5"+(2,1.5)*{} **\frm{-}},
{"ctext5"+(0,1.5)="d4";"d4"+(0,5)*{V}**\dir{-}},
{"c2"; "d2"**\dir{-}},
\endxy\,=\, \FS_{W,z}^{(m)}f.
$$
By the naturality of the $(m,1)$-st GFS endomorphisms, we find $\FS_{U \oplus V,z}^{(m)} =
\FS_{U, z}^{(m)}\oplus \FS_{V, z}^{(m)}$ for $U, V \in \CC$. If $\bX \in Z(\CC)$ is simple,
then it follows from \eqref{eq:FSE1} that
$$
\nu_{m,1}^\bX(U\oplus V) = \nu_{m,1}^\bX(U) +\nu_{m,1}^\bX(V). \qedhere
$$
\end{proof}

For $\CC=\C{H}$ for some semisimple quasi-Hopf algebra over $\BC$,  the natural endomorphism
$\FS_{-,z}^{(m)}$ is associated to a central element $\mu_{m,z}(H)$ of $H$. Moreover,
$$
\ptr(\FS_{V,z}^{(m)}) = \chi_V(\mu_{m,z}(H))
$$
where $\chi_V$ is the character afforded by $V \in \CC$. This central element was
determined in \cite{NS052} for $z=z_I$ but, for a
general $z$, is yet to be determined.

We now turn to our more \emph{general} version of GFS endomorphisms.
\begin{defn}
{\rm
  For non-negative integers $k,r,m$ with $l:=k+r+1\leq m$ and a natural
  endomorphism $z$ of the identity functor on $\DD$, we
  define
   \[\FS^{(m, k,r)}_{V,z}\,=\,
      \def\objectstyle{\scriptstyle}
\xy (0,-9)="ctext",
{"ctext"+(-2,7)="a2"},
{"a2"+(0,-1.5)="ctext1"*{z}},
{"ctext1"+(-2,-1.5)*{}; "ctext1"+(2,1.5)*{} **\frm{-}},
{"ctext1"+(0,-1.5); "ctext1"+(0,-2.5)**\dir{-}},
{"ctext1"+(0,-4)="ctext3"},{"ctext3"+(-3,0)*{\c}},
{"ctext3"+(-7.5,-1.5)*{}; "ctext3"+(2.5,1.5)*{} **\frm{-}},
\vcrossneg~{"a2"+(0,5)="c2"}{"a2"+(7,5)="c3"}{"a2"}{"a2"+(7,0)="a3"},
{"ctext3"+(-1.5,-1.5)="z2"; "z2"+(18,0)="z5" **\crv{"z2"+(0,-6)&"z5"+(0,-6)}},
{"z2"+(2.5,0)="z3"; "z3"+(6,0)="z4"**\crv{"z3"+(0,-3)& "z4"+(0,-3)}},
{"a3"; "z4"**\dir{-}?(.2)+(4,0)*{V^{l-m}}},
{"c3"+(0,1.5)="ctext4"*{\theta\inv}},
{"ctext4"+(-3,-1.5)*{}; "ctext4"+(2.5,1.5)*{} **\frm{-}},
{"ctext4"+(0,1.5); "ctext4"+(0,2.5)**\dir{-}},
{"ctext4"+(0,4)="ctext5"}, {"ctext5"+(2,0)*{\u}},
{"ctext5"+(-3,-1.5)*{}; "ctext5"+(7,1.5)*{} **\frm{-}},
{"c2";"c2"+(0,7)="d2"**\dir{-}},
{"d2";"d2"+(5.5,0)="d3"**\crv{"d2"+(0,3)&"d3"+(0,3)}},
{"ctext5"+(1,1.5)="d4"; "d4"+(-18,0)="d1"**\crv{"d4"+(0,6)&"d1"+(0,6)}},
{"d1"; "d1"+(0,-19)="z"**\dir{-}?(.2)+(-3,0)*{V^{-k}}},
{"z"; "z"+(4,0)="z0"**\crv{"z"+(0,-3)&"z0"+(0,-3)}},
{"z5"; "z5"+(0,19)="d7"**\dir{-}?(.2)+(4,0)*{V^{-r}}},
{"d7"; "d7"+(-4,0)="d6"**\crv{"d7"+(0,3)&"d6"+(0,3)}},
{"d4"+(2.5,0)="d5"; "d5"+(0,7)*{V}**\dir{-}},
{"z2"+(-2,0)="z1"; "z1"+(0,-7)*{V}**\dir{-}},
\endxy
    \]
    where the distribution of tensor factors over the legs of the
  graphical symbols for the unit and counit of adjunction is as
  follows:
  $$
  \def\objectstyle{\scriptstyle}
\xy {(0,0)="ctext"*{\u}},
{"ctext"+(-11,2)*{}; "ctext"+(11,-2)*{} **\frm{-}},
{"ctext"+(-9,2)="a1"; "a1"+(0,7)*{V^{m-l}}**\dir{-}},
{"ctext"+(-3,2)="a2"; "a2"+(0,7)*{\,\,V^k}**\dir{-}},
{"ctext"+(3,2)="a3"; "a3"+(0,7)*{V}**\dir{-}},
{"ctext"+(9,2)="a4"; "a4"+(0,7)*{V^r}**\dir{-}},
{"ctext"+(0,-2)="a5"; "a5"+(0,-7)*{K(V^m)}**\dir{-}},
\endxy
  \qquad
  \text{ and }
  \qquad
  \xy {(0,0)="ctext"*{\c}},
{"ctext"+(-11,2)*{}; "ctext"+(11,-2)*{} **\frm{-}},
{"ctext"+(-9,-2)="a1"; "a1"+(0,-7)*{V^k}**\dir{-}},
{"ctext"+(-3,-2)="a2"; "a2"+(0,-7)*{V}**\dir{-}},
{"ctext"+(3,-2)="a3"; "a3"+(0,-7)*{V^r}**\dir{-}},
{"ctext"+(9,-2)="a4"; "a4"+(0,-7)*{V^{m-l}}**\dir{-}},
{"ctext"+(0,2)="a5"; "a5"+(0,7)*{K(V^m)}**\dir{-}},
\endxy\,.
  $$
  We abbreviate $\FS_{V,z}^{(m,l)}=\FS_{V,z}^{(m,l-1,0)}$ for $1\le l \le m$.
  }
\end{defn}
\begin{lem}
  Let $m \in \BN$, $k,r\geq 0$ with $1\leq l:=k+r+1\le m$.
\begin{enumerate}
  \item[(i)]  For simple $\bX \in Z(\CC)$, $\nu^{\bX}_{m,l}(V)=
  \frac{1}{d_\ell(\bX)}\ptrl\left(\FS_{V,z_{\bX}}^{(m,l)}\right)$.
  \item[(ii)] If $r>0$, then $\ptrl\left(\FS_{V,z}^{(m,k,r)}\right)=\ptrr\left(\FS_{V,z}^{(m,k+1,r-1)}\right)$.
  \item[(iii)] In addition, if $\CC$ is spherical, then
  $$
  \ptr\left(\FS_{V,z}^{(m,k,r)}\right)=\ptr\left(\FS_{V,z}^{(m,l)}\right) \quad\text{and}\quad
  \nu^{\bX}_{m,l}(V)=\frac{1}{d(\bX)}\ptr\left(\FS_{V,z_{\bX}}^{(m,k,r)}\right)$$
  for simple $\bX \in \DD$. Moreover,
  $\FS_{V,z}^{(m,k,r)}=\FS_{V,z}^{(m,l)}$ for simple $V \in \CC$.
\end{enumerate}
\end{lem}
\begin{proof}
  The proof of (i) is similar to \eqref{eq:FSE1}, and (ii) can be obtained directly from graphical
  calculus. If $\CC$ is spherical, then by induction we have
  $$
  \ptr\left(\FS_{V,z}^{(m,k,r)}\right)=\ptr\left(\FS_{V,z}^{(m,l)}\right).
  $$
  Hence, by (i), $\nu^{\bX}_{m,l}(V)=\frac{1}{d(\bX)}\ptr\left(\FS_{V,z_{\bX}}^{(m,k,r)}\right)$ for
  simple $\bX \in Z(\CC)$.
  If $V$ is simple,
  then $\FS_{V,z}^{(m,k,r)}$ and $\FS_{V,z}^{(m,l)}$ are scalar multiples of $\id_V$. Since they have
  the same pivotal trace and $d(V) \ne 0$, $\FS_{V,z}^{(m,k,r)}=\FS_{V,z}^{(m,l)}$.
\end{proof}
\begin{prop}\label{p:FSE2}
  $\FS_{V,z}^{(m,k,r)}$ is natural in $V$ provided $m$ and
  $l:=k+r+1$ are relatively prime. In particular,  $\nu^{\bX}_{m,l}(V)$ is additive in $V$.
  In addition, if $\CC$ is spherical, then
  $\FS_{V,z}^{(m,k,r)}=\FS_{V,z}^{(m,l)}$ for all $V \in \CC$.
\end{prop}
\begin{proof}
Let $s\in S_m$ be the permutation determined by requiring
$s(i)\in\{1,\dots,m\}$ to be congruent to $i+l$ modulo $m$. Note
that $s$ is an $m$-cycle since $m$ and $l$ are relatively prime.

Consider $f\colon V\to W$. For any $X_i,Y_j$ in $\CC$ we will
write
 \[f_p
 \colon X_1\o\dots\o X_{p-1}\o V\o Y_1\o\dots\o Y_u
 \to X_1\o\dots\o X_{p-1}\o W\o Y_1\o\dots\o Y_u\]
for the morphism that acts as $f$ in the $p$-th position and the
identity elsewhere. Define a series of objects and morphisms
\[V[0]\xrightarrow{f[1]}V[1]\xrightarrow{f[2]}V[2]\to\dots\xrightarrow{f[n]}V[n]\]
by $V[0]=V^{\o n}$ and $f[i]=f_{s^{i-1}(k+1)}$; this fixes
$V[i]$ which has to be the appropriate target. Note that the
sequence is well-defined since $s$ is transitive, and we have
$V[n]=W^{\o n}$.

Let
$$
G_i:=\,
\def\objectstyle{\scriptstyle}
\xy (0,-9)="ctext",
{"ctext"+(-2,7)="a2"},
{"a2"+(0,-1.5)="ctext1"*{z}},
{"ctext1"+(-2,-1.5)*{}; "ctext1"+(2,1.5)*{} **\frm{-}},
{"ctext1"+(0,-1.5); "ctext1"+(0,-2.5)**\dir{-}},
{"ctext1"+(0,-4)="ctext3"},{"ctext3"+(-3,0)*{\c}},
{"ctext3"+(-7.5,-1.5)*{}; "ctext3"+(2.5,1.5)*{} **\frm{-}},
\vcrossneg~{"a2"+(0,5)="c2"}{"a2"+(7,5)="c3"}{"a2"}{"a2"+(7,0)="a3"},
{"ctext3"+(-1.5,-1.5)="z2"; "z2"+(18,0)="z5" **\crv{"z2"+(0,-6)&"z5"+(0,-6)}},
{"z2"+(2.5,0)="z3"; "z3"+(6,0)="z4"**\crv{"z3"+(0,-3)& "z4"+(0,-3)}},
{"a3"; "z4"**\dir{-}?(.2)+(4,0)},
{"z2"+(-2,0)="z1"; "z1"+(0,-7)*{W}**\dir{-}},
{"c3"+(0,1.5)="ctext4"*{\theta\inv}},
{"ctext4"+(-3,-1.5)*{}; "ctext4"+(2.5,1.5)*{} **\frm{-}},
{"ctext4"+(0,1.5); "ctext4"+(0,2.5)**\dir{-}},
{"ctext4"+(0,4)="ctext5"}, {"ctext5"+(2,0)*{\u}},
{"ctext5"+(-3,-1.5)*{}; "ctext5"+(7,1.5)*{} **\frm{-}},
{"ctext5"+(0,4)="ctext6"}, {"ctext6"+(2,-.1)*{f[i]}},
{"ctext6"+(-3,-1.5)*{}; "ctext6"+(7,1.5)*{} **\frm{-}},
{"c2";"c2"+(0,11)="d2"**\dir{-}},
{"d2";"d2"+(5.5,0)="d3"**\crv{"d2"+(0,3)&"d3"+(0,3)}},
{"ctext6"+(1,1.5)="d4"; "d4"+(-18,0)="d1"**\crv{"d4"+(0,6)&"d1"+(0,6)}},
{"d1"; "d1"+(0,-23)="z"**\dir{-}?(.2)+(-3,0)},
{"z"; "z"+(4,0)="z0"**\crv{"z"+(0,-3)&"z0"+(0,-3)}},
{"z5"; "z5"+(0,23)="d7"**\dir{-}?(.2)+(4,0)},
{"d7"; "d7"+(-4,0)="d6"**\crv{"d7"+(0,3)&"d6"+(0,3)}},
{"d4"+(2.5,0)="d5"; "d5"+(0,7)*{V}**\dir{-}},
{"d6"+(0,-3);"d6"+(0,-4) **\dir{-}},
{"d5"+(0,-3);"d5"+(0,-4) **\dir{-}},
{"d4"+(0,-3);"d4"+(0,-4) **\dir{-}},
{"d3"+(0,-3);"d3"+(0,-4) **\dir{-}},
\endxy\,=\,
\xy (0,-6)="ctext",
{"ctext"+(-2,7)="a2"},
{"a2"+(0,-1.5)="ctext1"*{z}},
{"ctext1"+(-2,-1.5)*{}; "ctext1"+(2,1.5)*{} **\frm{-}},
{"ctext1"+(0,-1.5); "ctext1"+(0,-2.5)**\dir{-}},
{"ctext1"+(0,-4)="ctext3"},{"ctext3"+(-3,0)*{\c}},
{"ctext3"+(-7.5,-1.5)*{}; "ctext3"+(2.5,1.5)*{} **\frm{-}},
{"ctext3"+(0,-4)="ctext0"},{"ctext0"+(-3,-.1)*{f[i]}},
{"ctext0"+(-7.5,-1.5)*{}; "ctext0"+(2.5,1.5)*{} **\frm{-}},
\vcrossneg~{"a2"+(0,5)="c2"}{"a2"+(7,5)="c3"}{"a2"}{"a2"+(7,0)="a3"},
{"ctext0"+(-1.5,-1.5)="z2"; "z2"+(18,0)="z5" **\crv{"z2"+(0,-6)&"z5"+(0,-6)}},
{"z2"+(2.5,0)="z3"; "z3"+(6,0)="z4"**\crv{"z3"+(0,-3)& "z4"+(0,-3)}},
{"a3"; "z4"**\dir{-}?(.2)+(4,0)},
{"c3"+(0,1.5)="ctext4"*{\theta\inv}},
{"ctext4"+(-3,-1.5)*{}; "ctext4"+(2.5,1.5)*{} **\frm{-}},
{"ctext4"+(0,1.5); "ctext4"+(0,2.5)**\dir{-}},
{"ctext4"+(0,4)="ctext5"}, {"ctext5"+(2,0)*{\u}},
{"ctext5"+(-3,-1.5)*{}; "ctext5"+(7,1.5)*{} **\frm{-}},
{"c2";"c2"+(0,7)="d2"**\dir{-}},
{"d2";"d2"+(5.5,0)="d3"**\crv{"d2"+(0,3)&"d3"+(0,3)}},
{"ctext5"+(1,1.5)="d4"; "d4"+(-18,0)="d1"**\crv{"d4"+(0,6)&"d1"+(0,6)}},
{"d1"; "d1"+(0,-23)="z"**\dir{-}?(.2)+(-3,0)},
{"z"; "z"+(4,0)="z0"**\crv{"z"+(0,-3)&"z0"+(0,-3)}},
{"z5"; "z5"+(0,23)="d7"**\dir{-}?(.2)+(4,0)},
{"d7"; "d7"+(-4,0)="d6"**\crv{"d7"+(0,3)&"d6"+(0,3)}},
{"d4"+(2.5,0)="d5"; "d5"+(0,7)*{V}**\dir{-}},
{"z2"+(-2,0)="z1"; "z1"+(0,-7)*{W}**\dir{-}},
{"z0"+(0,3);"z0"+(0,4)**\dir{-}},
{"z1"+(0,3);"z1"+(0,4)**\dir{-}},
{"z2"+(0,3);"z2"+(0,4)**\dir{-}},
{"z3"+(0,3);"z3"+(0,4)**\dir{-}},
\endxy\, \text{ for any } i\in\{1,\dots,m\}.
$$
The last equality follows from a similar argument as  \eqref{eq:FSE2}.
If $p:=s^{i-1}(k+1)\leq m-l$, then
$$
G_i=\,
\def\objectstyle{\scriptstyle}
\xy (0,-11)="ctext",
{"ctext"+(-2,7)="a2"},
{"a2"+(0,-1.5)="ctext1"*{z}},
{"ctext1"+(-2,-1.5)*{}; "ctext1"+(2,1.5)*{} **\frm{-}},
{"ctext1"+(0,-1.5); "ctext1"+(0,-2.5)**\dir{-}},
{"ctext1"+(0,-4)="ctext3"},{"ctext3"+(-3,0)*{\c}},
{"ctext3"+(-7.5,-1.5)*{}; "ctext3"+(2.5,1.5)*{} **\frm{-}},
\vcrossneg~{"a2"+(0,5)="c2"}{"a2"+(7,5)="c3"}{"a2"}{"a2"+(7,0)="a3"},
{"ctext3"+(-1.5,-1.5)="z2"; "z2"+(18,0)="z5" **\crv{"z2"+(0,-6)&"z5"+(0,-6)}},
{"z2"+(2.5,0)="z3"; "z3"+(6,0)="z4"**\crv{"z3"+(0,-3)& "z4"+(0,-3)}},
{"a3"; "z4"**\dir{-}?(.2)+(4,0)},
{"z2"+(-2,0)="z1"; "z1"+(0,-7)*{W}**\dir{-}},
{"c3"+(0,1.5)="ctext4"*{\theta\inv}},
{"ctext4"+(-3,-1.5)*{}; "ctext4"+(2.5,1.5)*{} **\frm{-}},
{"ctext4"+(0,1.5); "ctext4"+(0,2.5)**\dir{-}},
{"ctext4"+(0,4)="ctext5"}, {"ctext5"+(2,0)*{\u}},
{"ctext5"+(-3,-1.5)*{}; "ctext5"+(7,1.5)*{} **\frm{-}},
{"ctext5"+(-2.2,4)="ctext6"},{"ctext6"+(0,-.2)*{f_p}},
{"ctext6"+(-1.7,-1.5)*{}; "ctext6"+(1.4,1.5)*{} **\frm{-}},
{"c2";"c2"+(0,11)="d2"**\dir{-}},
{"d2";"d2"+(5,0)="d3"**\crv{"d2"+(0,3)&"d3"+(0,3)}},
{"ctext5"+(1,5.5)="d4"; "d4"+(-18,0)="d1"**\crv{"d4"+(0,7)&"d1"+(0,7)}},
{"d1"; "d1"+(0,-23)="z"**\dir{-}?(.2)+(-3,0)},
{"z"; "z"+(4,0)="z0"**\crv{"z"+(0,-3)&"z0"+(0,-3)}},
{"z5"; "z5"+(0,19)="d7"**\dir{-}?(.2)+(4,0)},
{"d7"; "d7"+(-4,0)="d6"**\crv{"d7"+(0,3)&"d6"+(0,3)}},
{"d4"+(2.5,0)="d5"},
{"d5"+(0,-4);"d5"+(0,7)*{V} **\dir{-}},
{"d4" ;"d4"+(0,-4) **\dir{-}},
{"d3"+(0,-3);"d3"+(0,-4) **\dir{-}},
\endxy\,=\,
\xy (0,-6)="ctext",
{"ctext"+(-2,7)="a2"},
{"a2"+(0,-1.5)="ctext1"*{z}},
{"ctext1"+(-2,-1.5)*{}; "ctext1"+(2,1.5)*{} **\frm{-}},
{"ctext1"+(0,-1.5); "ctext1"+(0,-2.5)**\dir{-}},
{"ctext1"+(0,-4)="ctext3"},{"ctext3"+(-3,0)*{\c}},
{"ctext3"+(-7.5,-1.5)*{}; "ctext3"+(2.5,1.5)*{} **\frm{-}},
{"ctext3"+(2,-4)="ctext0"},{"ctext0"+(0,-.2)*{f_p}},
{"ctext0"+(-1.7,-1.5)*{}; "ctext0"+(1.4,1.5)*{} **\frm{-}},
\vcrossneg~{"a2"+(0,5)="c2"}{"a2"+(7,5)="c3"}{"a2"}{"a2"+(7,0)="a3"},
{"ctext0"+(-2.8,-1.5)="z2"; "z2"+(17,0)="z5" **\crv{"z2"+(0,-6)&"z5"+(0,-6)}},
{"z2"+(2.5,0)="z3"; "z3"+(5.2,0)="z4"**\crv{"z3"+(0,-3)& "z4"+(0,-3)}},
{"a3"; "z4"**\dir{-}?(.2)+(4,0)},
{"c3"+(0,1.5)="ctext4"*{\theta\inv}},
{"ctext4"+(-3,-1.5)*{}; "ctext4"+(2.5,1.5)*{} **\frm{-}},
{"ctext4"+(0,1.5); "ctext4"+(0,2.5)**\dir{-}},
{"ctext4"+(0,4)="ctext5"}, {"ctext5"+(2,0)*{\u}},
{"ctext5"+(-3,-1.5)*{}; "ctext5"+(7,1.5)*{} **\frm{-}},
{"c2";"c2"+(0,7)="d2"**\dir{-}},
{"d2";"d2"+(5.5,0)="d3"**\crv{"d2"+(0,3)&"d3"+(0,3)}},
{"ctext5"+(1,1.5)="d4"; "d4"+(-18,0)="d1"**\crv{"d4"+(0,6)&"d1"+(0,6)}},
{"d1"; "d1"+(0,-19)="z"**\dir{-}?(.2)+(-3,0)},
{"z"; "z"+(4,0)="z0"**\crv{"z"+(0,-3)&"z0"+(0,-3)}},
{"z5"; "z5"+(0,23)="d7"**\dir{-}?(.2)+(4,0)},
{"d7"; "d7"+(-4,0)="d6"**\crv{"d7"+(0,3)&"d6"+(0,3)}},
{"d4"+(2.5,0)="d5"; "d5"+(0,7)*{V}**\dir{-}},
{"z2"+(-2.7,0)="z1"; "z1"+(0,-7)*{W}**\dir{-}},
{"z1"+(0,4);"z1"+(0,-4)**\dir{-}},
{"z2";"z2"+(0,4)**\dir{-}},
{"z3"+(0,3);"z3"+(0,4)**\dir{-}},
\endxy\, =\,
\xy (0,-6)="ctext",
{"ctext"+(-2,7)="a2"},
{"a2"+(0,-1.5)="ctext1"*{z}},
{"ctext1"+(-2,-1.5)*{}; "ctext1"+(2,1.5)*{} **\frm{-}},
{"ctext1"+(0,-1.5); "ctext1"+(0,-2.5)**\dir{-}},
{"ctext1"+(0,-4)="ctext3"},{"ctext3"+(-3,0)*{\c}},
{"ctext3"+(-7.5,-1.5)*{}; "ctext3"+(2.5,1.5)*{} **\frm{-}},
{"ctext3"+(0,-4)="ctext0"},{"ctext0"+(-3,-.2)*{f[i+1]}},
{"ctext0"+(-7.5,-1.5)*{}; "ctext0"+(2.5,1.5)*{} **\frm{-}},
\vcrossneg~{"a2"+(0,5)="c2"}{"a2"+(7,5)="c3"}{"a2"}{"a2"+(7,0)="a3"},
{"ctext0"+(-1.5,-1.5)="z2"; "z2"+(18,0)="z5" **\crv{"z2"+(0,-6)&"z5"+(0,-6)}},
{"z2"+(2.5,0)="z3"; "z3"+(6,0)="z4"**\crv{"z3"+(0,-3)& "z4"+(0,-3)}},
{"a3"; "z4"**\dir{-}?(.2)+(4,0)},
{"c3"+(0,1.5)="ctext4"*{\theta\inv}},
{"ctext4"+(-3,-1.5)*{}; "ctext4"+(2.5,1.5)*{} **\frm{-}},
{"ctext4"+(0,1.5); "ctext4"+(0,2.5)**\dir{-}},
{"ctext4"+(0,4)="ctext5"}, {"ctext5"+(2,0)*{\u}},
{"ctext5"+(-3,-1.5)*{}; "ctext5"+(7,1.5)*{} **\frm{-}},
{"c2";"c2"+(0,7)="d2"**\dir{-}},
{"d2";"d2"+(5.5,0)="d3"**\crv{"d2"+(0,3)&"d3"+(0,3)}},
{"ctext5"+(1,1.5)="d4"; "d4"+(-18,0)="d1"**\crv{"d4"+(0,6)&"d1"+(0,6)}},
{"d1"; "d1"+(0,-23)="z"**\dir{-}?(.2)+(-3,0)},
{"z"; "z"+(4,0)="z0"**\crv{"z"+(0,-3)&"z0"+(0,-3)}},
{"z5"; "z5"+(0,23)="d7"**\dir{-}?(.2)+(4,0)},
{"d7"; "d7"+(-4,0)="d6"**\crv{"d7"+(0,3)&"d6"+(0,3)}},
{"d4"+(2.5,0)="d5"; "d5"+(0,7)*{V}**\dir{-}},
{"z2"+(-2,0)="z1"; "z1"+(0,-7)*{W}**\dir{-}},
{"z0"+(0,3);"z0"+(0,4)**\dir{-}},
{"z1"+(0,3);"z1"+(0,4)**\dir{-}},
{"z2"+(0,3);"z2"+(0,4)**\dir{-}},
{"z3"+(0,3);"z3"+(0,4)**\dir{-}},
\endxy \,=G_{i+1}
$$
since $p+l=s(p)=s^{i}(k+1)$.
  If $m-l<p\leq m-l+k$, then for $q:=p-m+l$
$$
G_i=\,
\def\objectstyle{\scriptstyle}
\xy (0,-10)="ctext",
{"ctext"+(-2,7)="a2"},
{"a2"+(0,-1.5)="ctext1"*{z}},
{"ctext1"+(-2,-1.5)*{}; "ctext1"+(2,1.5)*{} **\frm{-}},
{"ctext1"+(0,-1.5); "ctext1"+(0,-2.5)**\dir{-}},
{"ctext1"+(0,-4)="ctext3"},{"ctext3"+(-3,0)*{\c}},
{"ctext3"+(-7.5,-1.5)*{}; "ctext3"+(2.5,1.5)*{} **\frm{-}},
\vcrossneg~{"a2"+(0,5)="c2"}{"a2"+(7,5)="c3"}{"a2"}{"a2"+(7,0)="a3"},
{"ctext3"+(-1.5,-1.5)="z2"; "z2"+(18,0)="z5" **\crv{"z2"+(0,-6)&"z5"+(0,-6)}},
{"z2"+(2.5,0)="z3"; "z3"+(6,0)="z4"**\crv{"z3"+(0,-3)& "z4"+(0,-3)}},
{"a3"; "z4"**\dir{-}?(.2)+(4,0)},
{"z2"+(-2,0)="z1"; "z1"+(0,-7)*{W}**\dir{-}},
{"c3"+(0,1.5)="ctext4"*{\theta\inv}},
{"ctext4"+(-3,-1.5)*{}; "ctext4"+(2.5,1.5)*{} **\frm{-}},
{"ctext4"+(0,1.5); "ctext4"+(0,2.5)**\dir{-}},
{"ctext4"+(0,4)="ctext5"}, {"ctext5"+(2,0)*{\u}},
{"ctext5"+(-3,-1.5)*{}; "ctext5"+(7,1.5)*{} **\frm{-}},
{"ctext5"+(.9,4)="ctext6"},{"ctext6"+(0,-.2)*{f_p}},
{"ctext6"+(-1.7,-1.5)*{}; "ctext6"+(1.4,1.5)*{} **\frm{-}},
{"c2";"c2"+(0,7)="d2"**\dir{-}},
{"d2";"d2"+(5,0)="d3"**\crv{"d2"+(0,3)&"d3"+(0,3)}},
{"ctext6"+(0, 1.5)="d4"; "d4"+(-18,0)="d1"**\crv{"d4"+(0,7)&"d1"+(0,7)}},
{"d1"; "d1"+(0,-23)="z"**\dir{-}?(.2)+(-3,0)},
{"z"; "z"+(4,0)="z0"**\crv{"z"+(0,-3)&"z0"+(0,-3)}},
{"z5"; "z5"+(0,19)="d7"**\dir{-}?(.2)+(4,0)},
{"d7"; "d7"+(-4,0)="d6"**\crv{"d7"+(0,3)&"d6"+(0,3)}},
{"d4"+(2.5,0)="d5"},
{"d5"+(0,-4);"d5"+(0,7)*{V} **\dir{-}},
{"d4"+(0,-3) ;"d4"+(0,-4) **\dir{-}},
\endxy\,=\,
\def\objectstyle{\scriptstyle}
\xy (0,-8)="ctext",
{"ctext"+(-2,7)="a2"},
{"a2"+(0,-1.5)="ctext1"*{z}},
{"ctext1"+(-2,-1.5)*{}; "ctext1"+(2,1.5)*{} **\frm{-}},
{"ctext1"+(0,-1.5); "ctext1"+(0,-2.5)**\dir{-}},
{"ctext1"+(0,-4)="ctext3"},{"ctext3"+(-3,0)*{\c}},
{"ctext3"+(-7.5,-1.5)*{}; "ctext3"+(2.5,1.5)*{} **\frm{-}},
\vcrossneg~{"a2"+(0,5)="c2"}{"a2"+(7,5)="c3"}{"a2"}{"a2"+(7,0)="a3"},
{"ctext3"+(-1.5,-1.5)="z2"; "z2"+(18,0)="z5" **\crv{"z2"+(0,-6)&"z5"+(0,-6)}},
{"z2"+(2.5,0)="z3"; "z3"+(6,0)="z4"**\crv{"z3"+(0,-3)& "z4"+(0,-3)}},
{"a3"; "z4"**\dir{-}?(.2)+(4,0)},
{"z2"+(-2,0)="z1"; "z1"+(0,-9)*{W}**\dir{-}},
{"c3"+(0,1.5)="ctext4"*{\theta\inv}},
{"ctext4"+(-3,-1.5)*{}; "ctext4"+(2.5,1.5)*{} **\frm{-}},
{"ctext4"+(0,1.5); "ctext4"+(0,2.5)**\dir{-}},
{"ctext4"+(0,4)="ctext5"}, {"ctext5"+(2,0)*{\u}},
{"ctext5"+(-3,-1.5)*{}; "ctext5"+(7,1.5)*{} **\frm{-}},
{"ctext5"+(.9,4)="ctext6"},
{"c2";"c2"+(0,7)="d2"**\dir{-}},
{"d2";"d2"+(5,0)="d3"**\crv{"d2"+(0,3)&"d3"+(0,3)}},
{"ctext6"+(0, -2)="d4"; "d4"+(-18,0)="d1"**\crv{"d4"+(0,7)&"d1"+(0,7)}},
{"d1"; "d1"+(0,-24)="z"**\dir{-}?(.2)+(-3,0)},
{"z"; "z"+(4,0)="z0"**\crv{"z"+(0,-2)&"z0"+(0,-2)}},
{"z0"+(1.7,3.5)*{}; "z0"+(-1.4,0.5)*{} **\frm{-}},
{"z0"+(.34,1.9)*{f_p}},
{"z0"+(.3,3.5);"z0"+(.3,4.5)**\dir{-}},
{"z5"; "z5"+(0,19)="d7"**\dir{-}?(.2)+(4,0)},
{"d7"; "d7"+(-4,0)="d6"**\crv{"d7"+(0,3)&"d6"+(0,3)}},
{"d4"+(2.5,0)="d5"},
{"d5"+(0,-0.5);"d5"+(0,9)*{V} **\dir{-}},
{"d4" ;"d4"+(0,-0.5) **\dir{-}},
\endxy\,=\,
\xy (0,-6)="ctext",
{"ctext"+(-2,7)="a2"},
{"a2"+(0,-1.5)="ctext1"*{z}},
{"ctext1"+(-2,-1.5)*{}; "ctext1"+(2,1.5)*{} **\frm{-}},
{"ctext1"+(0,-1.5); "ctext1"+(0,-2.5)**\dir{-}},
{"ctext1"+(0,-4)="ctext3"},{"ctext3"+(-3,0)*{\c}},
{"ctext3"+(-7.5,-1.5)*{}; "ctext3"+(2.5,1.5)*{} **\frm{-}},
{"ctext3"+(0,-4)="ctext0"},{"ctext0"+(-3,-.2)*{f[i+1]}},
{"ctext0"+(-7.5,-1.5)*{}; "ctext0"+(2.5,1.5)*{} **\frm{-}},
\vcrossneg~{"a2"+(0,5)="c2"}{"a2"+(7,5)="c3"}{"a2"}{"a2"+(7,0)="a3"},
{"ctext0"+(-1.5,-1.5)="z2"; "z2"+(18,0)="z5" **\crv{"z2"+(0,-6)&"z5"+(0,-6)}},
{"z2"+(2.5,0)="z3"; "z3"+(6,0)="z4"**\crv{"z3"+(0,-3)& "z4"+(0,-3)}},
{"a3"; "z4"**\dir{-}?(.2)+(4,0)},
{"c3"+(0,1.5)="ctext4"*{\theta\inv}},
{"ctext4"+(-3,-1.5)*{}; "ctext4"+(2.5,1.5)*{} **\frm{-}},
{"ctext4"+(0,1.5); "ctext4"+(0,2.5)**\dir{-}},
{"ctext4"+(0,4)="ctext5"}, {"ctext5"+(2,0)*{\u}},
{"ctext5"+(-3,-1.5)*{}; "ctext5"+(7,1.5)*{} **\frm{-}},
{"c2";"c2"+(0,7)="d2"**\dir{-}},
{"d2";"d2"+(5.5,0)="d3"**\crv{"d2"+(0,3)&"d3"+(0,3)}},
{"ctext5"+(1,1.5)="d4"; "d4"+(-18,0)="d1"**\crv{"d4"+(0,6)&"d1"+(0,6)}},
{"d1"; "d1"+(0,-23)="z"**\dir{-}?(.2)+(-3,0)},
{"z"; "z"+(4,0)="z0"**\crv{"z"+(0,-3)&"z0"+(0,-3)}},
{"z5"; "z5"+(0,23)="d7"**\dir{-}?(.2)+(4,0)},
{"d7"; "d7"+(-4,0)="d6"**\crv{"d7"+(0,3)&"d6"+(0,3)}},
{"d4"+(2.5,0)="d5"; "d5"+(0,7)*{V}**\dir{-}},
{"z2"+(-2,0)="z1"; "z1"+(0,-7)*{W}**\dir{-}},
{"z0"+(0,3);"z0"+(0,4)**\dir{-}},
{"z1"+(0,3);"z1"+(0,4)**\dir{-}},
{"z2"+(0,3);"z2"+(0,4)**\dir{-}},
{"z3"+(0,3);"z3"+(0,4)**\dir{-}},
\endxy \,=G_{i+1}
$$
since $q=s(p)=s^i(k+1)$. If $p> m-l+k+1=m-r+1$, then we set $q=p-m+l-k-1$ and find
$$
G_i=\,
\def\objectstyle{\scriptstyle}
\xy (0,-9)="ctext",
{"ctext"+(-2,7)="a2"},
{"a2"+(0,-1.5)="ctext1"*{z}},
{"ctext1"+(-2,-1.5)*{}; "ctext1"+(2,1.5)*{} **\frm{-}},
{"ctext1"+(0,-1.5); "ctext1"+(0,-2.5)**\dir{-}},
{"ctext1"+(0,-4)="ctext3"},{"ctext3"+(-3,0)*{\c}},
{"ctext3"+(-7.5,-1.5)*{}; "ctext3"+(2.5,1.5)*{} **\frm{-}},
\vcrossneg~{"a2"+(0,5)="c2"}{"a2"+(7,5)="c3"}{"a2"}{"a2"+(7,0)="a3"},
{"ctext3"+(-1.5,-1.5)="z2"; "z2"+(18,0)="z5" **\crv{"z2"+(0,-6)&"z5"+(0,-6)}},
{"z2"+(2.5,0)="z3"; "z3"+(6,0)="z4"**\crv{"z3"+(0,-3)& "z4"+(0,-3)}},
{"a3"; "z4"**\dir{-}?(.2)+(4,0)},
{"z2"+(-2,0)="z1"; "z1"+(0,-7)*{W}**\dir{-}},
{"c3"+(0,1.5)="ctext4"*{\theta\inv}},
{"ctext4"+(-3,-1.5)*{}; "ctext4"+(2.5,1.5)*{} **\frm{-}},
{"ctext4"+(0,1.5); "ctext4"+(0,2.5)**\dir{-}},
{"ctext4"+(0,4)="ctext5"}, {"ctext5"+(2,0)*{\u}},
{"ctext5"+(-3,-1.5)*{}; "ctext5"+(7,1.5)*{} **\frm{-}},
{"ctext5"+(6,4)="ctext6"},{"ctext6"+(0,-.2)*{f_p}},
{"ctext6"+(-1.7,-1.5)*{}; "ctext6"+(1.4,1.5)*{} **\frm{-}},
{"c2";"c2"+(0,7)="d2"**\dir{-}},
{"d2";"d2"+(5,0)="d3"**\crv{"d2"+(0,3)&"d3"+(0,3)}},
{"d3"+(2.5,0)="d4"; "d4"+(-18,0)="d1"**\crv{"d4"+(0,7)&"d1"+(0,7)}},
{"d1"; "d1"+(0,-19)="z"**\dir{-}?(.2)+(-3,0)},
{"z"; "z"+(4,0)="z0"**\crv{"z"+(0,-3)&"z0"+(0,-3)}},
{"z5"; "z5"+(0,23)="d7"**\dir{-}?(.2)+(4,0)},
{"d7"; "d7"+(-3.5,0)="d6"**\crv{"d7"+(0,3)&"d6"+(0,3)}},
{"d4"+(2.5,0)="d5"},
{"d5";"d5"+(0,9)*{V} **\dir{-}},
{"d6"+(0,-3);"d6"+(0,-4) **\dir{-}},
\endxy\,=\,
\xy (0,-6)="ctext",
{"ctext"+(-2,7)="a2"},
{"a2"+(0,-1.5)="ctext1"*{z}},
{"ctext1"+(-2,-1.5)*{}; "ctext1"+(2,1.5)*{} **\frm{-}},
{"ctext1"+(0,-1.5); "ctext1"+(0,-2.5)**\dir{-}},
{"ctext1"+(0,-4)="ctext3"},{"ctext3"+(-3,0)*{\c}},
{"ctext3"+(-7.5,-1.5)*{}; "ctext3"+(2.5,1.5)*{} **\frm{-}},
{"ctext3"+(-.7,-4)="ctext0"},{"ctext0"+(0,-.2)*{f_p}},
{"ctext0"+(-1.8,-1.5)*{}; "ctext0"+(1.4,1.5)*{} **\frm{-}},
\vcrossneg~{"a2"+(0,5)="c2"}{"a2"+(7,5)="c3"}{"a2"}{"a2"+(7,0)="a3"},
{"ctext0"+(0,-1.5)="z2"; "z2"+(17,0)="z5" **\crv{"z2"+(0,-6)&"z5"+(0,-6)}},
{"z2"+(2.5,4)="z3"; "z3"+(5.2,0)="z4"**\crv{"z3"+(0,-3)& "z4"+(0,-3)}},
{"a3"; "z4"**\dir{-}?(.2)+(4,0)},
{"c3"+(0,1.5)="ctext4"*{\theta\inv}},
{"ctext4"+(-3,-1.5)*{}; "ctext4"+(2.5,1.5)*{} **\frm{-}},
{"ctext4"+(0,1.5); "ctext4"+(0,2.5)**\dir{-}},
{"ctext4"+(0,4)="ctext5"}, {"ctext5"+(2,0)*{\u}},
{"ctext5"+(-3,-1.5)*{}; "ctext5"+(7,1.5)*{} **\frm{-}},
{"c2";"c2"+(0,7)="d2"**\dir{-}},
{"d2";"d2"+(5.5,0)="d3"**\crv{"d2"+(0,3)&"d3"+(0,3)}},
{"ctext5"+(1,1.5)="d4"; "d4"+(-18,0)="d1"**\crv{"d4"+(0,6)&"d1"+(0,6)}},
{"d1"; "d1"+(0,-19)="z"**\dir{-}?(.2)+(-3,0)},
{"z"; "z"+(4,0)="z0"**\crv{"z"+(0,-3)&"z0"+(0,-3)}},
{"z5"; "z5"+(0,23)="d7"**\dir{-}?(.2)+(4,0)},
{"d7"; "d7"+(-4,0)="d6"**\crv{"d7"+(0,3)&"d6"+(0,3)}},
{"d4"+(2.5,0)="d5"; "d5"+(0,7)*{V}**\dir{-}},
{"z2"+(-2.8,0)="z1"; "z1"+(0,-7)*{W}**\dir{-}},
{"z1"+(0,4);"z1"+(0,-4)**\dir{-}},
{"z2"+(0,3);"z2"+(0,4)**\dir{-}},
\endxy
\, =\,
\xy (0,-6)="ctext",
{"ctext"+(-2,7)="a2"},
{"a2"+(0,-1.5)="ctext1"*{z}},
{"ctext1"+(-2,-1.5)*{}; "ctext1"+(2,1.5)*{} **\frm{-}},
{"ctext1"+(0,-1.5); "ctext1"+(0,-2.5)**\dir{-}},
{"ctext1"+(0,-4)="ctext3"},{"ctext3"+(-3,0)*{\c}},
{"ctext3"+(-7.5,-1.5)*{}; "ctext3"+(2.5,1.5)*{} **\frm{-}},
{"ctext3"+(0,-4)="ctext0"},{"ctext0"+(-3,-.2)*{f[i+1]}},
{"ctext0"+(-7.5,-1.5)*{}; "ctext0"+(2.5,1.5)*{} **\frm{-}},
\vcrossneg~{"a2"+(0,5)="c2"}{"a2"+(7,5)="c3"}{"a2"}{"a2"+(7,0)="a3"},
{"ctext0"+(-1.5,-1.5)="z2"; "z2"+(18,0)="z5" **\crv{"z2"+(0,-6)&"z5"+(0,-6)}},
{"z2"+(2.5,0)="z3"; "z3"+(6,0)="z4"**\crv{"z3"+(0,-3)& "z4"+(0,-3)}},
{"a3"; "z4"**\dir{-}?(.2)+(4,0)},
{"c3"+(0,1.5)="ctext4"*{\theta\inv}},
{"ctext4"+(-3,-1.5)*{}; "ctext4"+(2.5,1.5)*{} **\frm{-}},
{"ctext4"+(0,1.5); "ctext4"+(0,2.5)**\dir{-}},
{"ctext4"+(0,4)="ctext5"}, {"ctext5"+(2,0)*{\u}},
{"ctext5"+(-3,-1.5)*{}; "ctext5"+(7,1.5)*{} **\frm{-}},
{"c2";"c2"+(0,7)="d2"**\dir{-}},
{"d2";"d2"+(5.5,0)="d3"**\crv{"d2"+(0,3)&"d3"+(0,3)}},
{"ctext5"+(1,1.5)="d4"; "d4"+(-18,0)="d1"**\crv{"d4"+(0,6)&"d1"+(0,6)}},
{"d1"; "d1"+(0,-23)="z"**\dir{-}?(.2)+(-3,0)},
{"z"; "z"+(4,0)="z0"**\crv{"z"+(0,-3)&"z0"+(0,-3)}},
{"z5"; "z5"+(0,23)="d7"**\dir{-}?(.2)+(4,0)},
{"d7"; "d7"+(-4,0)="d6"**\crv{"d7"+(0,3)&"d6"+(0,3)}},
{"d4"+(2.5,0)="d5"; "d5"+(0,7)*{V}**\dir{-}},
{"z2"+(-2,0)="z1"; "z1"+(0,-7)*{W}**\dir{-}},
{"z0"+(0,3);"z0"+(0,4)**\dir{-}},
{"z1"+(0,3);"z1"+(0,4)**\dir{-}},
{"z2"+(0,3);"z2"+(0,4)**\dir{-}},
{"z3"+(0,3);"z3"+(0,4)**\dir{-}},
\endxy \,=G_{i+1}
$$
since now $q+k+1=p+l-m=s(p)=s^i(k+1)$. Because $s$ has order $m$ and $s(m-l+k-1)=k+1$, the case
$p=m-l+k+1=s^{m-1}(k+1)$ occurs when $i=m$. Thus, using induction, we have
$$
f\circ \FS_{V,z}^{(m,k,r)} =G_1 = \cdots = G_m = \FS_{V,z}^{(m,k,r)}\circ f\,.
$$
The remaining statements are direct consequences of the naturality of $\FS_{V,z}^{m,k,r}$ and
the previous lemma.
\end{proof}

\noindent
{\bf Acknowledgement:} The authors would like to thank Ling Long for her comments and suggestions on congruence subgroups.

\providecommand{\bysame}{\leavevmode\hbox to3em{\hrulefill}\thinspace}
\providecommand{\MR}{\relax\ifhmode\unskip\space\fi MR }
\providecommand{\MRhref}[2]{%
  \href{http://www.ams.org/mathscinet-getitem?mr=#1}{#2}
}
\providecommand{\href}[2]{#2}

\end{document}